\documentclass[preprint, 11pt]{article}

\usepackage{titlesec}
\usepackage[T1]{fontenc}
\usepackage{lmodern}
\usepackage{bigints}
\usepackage{chngcntr}
\usepackage{blindtext}
\usepackage[a4paper, total={6.5in, 8in}]{geometry}
\usepackage{amsmath,amsthm}
\usepackage{amssymb}
\usepackage{paralist}
\usepackage{graphics} 
\usepackage{tcolorbox}
\usepackage{xcolor}
\usepackage{hyperref}
\usepackage{epsfig} 
\usepackage{subcaption}
\usepackage{graphicx}  
\usepackage{epstopdf}
\usepackage{mathtools}
\usepackage{commath}
\usepackage{centernot}
\usepackage[shortlabels]{enumitem}
\usepackage{extarrows}
\usepackage{mathrsfs}
\hypersetup{urlcolor=blue, citecolor=red}
\usepackage{microtype}
\usepackage[mathscr]{euscript}
\usepackage{float}
\usepackage{csquotes}
\usepackage{scalerel}
\usepackage[utf8]{inputenc}
\usepackage{chngcntr}
\usepackage{apptools}
\usepackage{mathtools}
\usepackage{authblk}

\makeatletter
\newenvironment{sqcases}{%
	\matrix@check\sqcases\env@sqcases
}{%
	\endarray\right.%
}

\def\env@sqcases{%
	\let\@ifnextchar\new@ifnextchar
	\left[
	\def\arraystretch{1.2}%
	\array{@{}l@{\quad}l@{}}%
}
\makeatother

\makeatletter
\newcommand{\raisemath}[1]{\mathpalette{\raisem@th{#1}}}
\newcommand{\raisem@th}[3]{\raisebox{#1}{$#2#3$}}
\makeatother

\makeatletter
\newcommand{\subalign}[1]{%
	\vcenter{%
		\Let@ \restore@math@cr \default@tag
		\baselineskip\fontdimen10 \scriptfont\tw@
		\advance\baselineskip\fontdimen12 \scriptfont\tw@
		\lineskip\thr@@\fontdimen8 \scriptfont\thr@@
		\lineskiplimit\lineskip
		\ialign{\hfil$\m@th\scriptstyle##$&$\m@th\scriptstyle{}##$\hfil\crcr
			#1\crcr
		}%
	}%
}

\makeatother

\allowdisplaybreaks

\def\f{\longrightarrow}

\def\N{\mathbb{N}}

\def\l{\lambda}

\def\J{\mathcal{J}}
\def\x{\bar{x}}

\def\y{\bar{y}}
\def\u{\bar{u}}

\def\a{\alpha}
\def\b{\beta}
\def\<{\langle}
\def\>{\rangle}
\def\r{\mathcal{R}}

\def\I{\mathcal{I}}

\def\R{\mathbb{R}}

\def\inte{\textnormal{int}\,}
\def\clo{\textnormal{cl}\,}
\def\epi{\textnormal{epi}\,}
\def\bdry{\textnormal{bdry}\,}
\def\dom{\textnormal{dom}\,}
\def\conv{\textnormal{conv}\,}

\def\Gr{\textnormal{Gr}\,}
\def\gk{\gamma_k}

\def\CO{\mathcal{C}}
\def\c{\mathsf{c}}
\def\-{\textnormal{-}}
\newcommand*{\tran}{^{\mkern-1.5mu\mathsf{T}}}

\def\bp{\hspace{-0.08cm}}
\def\bbp{\hspace{-0.04cm}}

\def\0gk{\scaleto{0,\gk}{3.5pt}}
\def\supp{\textnormal{supp}\,}

\makeatletter

\theoremstyle{thmstyleone}%
\newtheorem{theorem}{Theorem}[section]
\newtheorem{proposition}[theorem]{Proposition}%
\newtheorem{lemma}[theorem]{Lemma}

\newtheorem{corollary}[theorem]{Corollary}

\theoremstyle{thmstyletwo}%
\newtheorem{example}[theorem]{Example}%
\newtheorem{remark}[theorem]{Remark}%

\theoremstyle{thmstylethree}%

\titleformat{\section}{\large}{\thesection. }{0em}{\textbf}
\titleformat{\subsection}{}{\thesubsection. }{0em}{\textbf}

	\usepackage{authblk}
	\title{Optimal control for coupled sweeping processes under minimal assumptions}
	\author[1]{Samara Chamoun}
	\author[2]{Vera Zeidan\footnote{Research of the second author was partly supported by SIMONS Fondation under grant MP-TSM-00002506}}
	\affil[1]{Department of Mathematics, Michigan State University, East Lansing, MI 48824-1027, USA\\ chamouns@msu.edu}
	\affil[2]{Department of Mathematics, Michigan State University, East Lansing, MI 48824-1027, USA\\ zeidan@msu.edu}
	\date{}                     
\begin{document}

		\maketitle


	\begin{abstract}
  	In this paper,   the study of nonsmooth optimal control problems $(P)$ involving a controlled sweeping process with {\it three} main characteristics is launched. First, the sweeping sets are {\it nonsmooth},   {\it time-dependent},  and uniformly prox-regular. Second, the sweeping process is {\it coupled} with a controlled differential equation. Third,  a {\it joint}-state endpoints constraint set $S$ is present. This general model incorporates different important controlled submodels, such as a class of second order sweeping processes, and coupled evolution variational inequalities. A full form of the  {\it nonsmooth} Pontryagin maximum principle  for {\it strong} local minimizers in $(P)$ is derived for {\it bounded or unbounded moving} sweeping sets  satisfying {\it local} constraint qualifications (CQ) {\it without}  any additional restriction. The existence and uniqueness of a Lipschitz solution for the Cauchy problem of our dynamic is established and the existence of an optimal solution for $(P)$ is obtained.  Two of the novelties in achieving the first goal   are (i) the construction of a  problem  over  {\it truncated}  sweeping sets and {\it truncated} joint endpoints constraint  set that has the same strong local minimizer as $(P)$ and its (CQ) automatically holds, and (ii)  the {\it  complete redesign} of the exponential-penalty approximation technique for problems with  moving sweeping sets that {\it do not require} any  special assumption on the sets, their corners, or on the gradients of their generators.  The  utility of the optimality conditions is illustrated with an  example.
	\end{abstract}	    				
		
\maketitle

\section{Introduction and Preliminaries}\label{intro} 
\setcounter{footnote}{0}
\subsection{Introduction}
\textbf{Sweeping process.} J.J. Moreau introduced the \textit{sweeping process} as being  a differential inclusion in which the set-valued map is the normal cone to a nicely moving closed set $C(t)$, called the \textit{sweeping set} (see \cite{moreau1,moreau2,moreau3}). Its simplest form is given by  \vspace{-0.2cm}\begin{equation*} \vspace{-0.2cm}
	\dot{x}(t) \in -N_{C(t)}(x(t)), \; \text{ a.e. } t \in [0,T].
\end{equation*}  When  $C(t)\subset \mathbb{R}^n$ is a non-empty convex closed set,  $N_{C(t)}$ is taken to be the  normal cone of convex analysis. When $C(t)$  is {\it non-convex},  as is the case of this paper, $N_{C(t)}$ denotes the Clarke normal cone and the set $C(t)$ is commonly taken to be uniformly prox-regular.  When a perturbation or external force $f$ exists, we call the dynamic a perturbed sweeping process, and when $f$ depends on a control $u$, we call it a perturbed controlled sweeping process, that is, \vspace{-0.1cm} \begin{equation}\label{sweeping2}
\dot{x}(t) \in f(t,x(t),u(t)) -N_{C(t)}(x(t)) \text{ a.e. } t \in [0,T], \; x(0)=x_0 \in C(0).
\end{equation}   
 Sweeping processes appear in many applications including elastoplasticity,  hysteresis, ferromagnetism, electric circuits, phase transitions, traffic equilibrium, etc., see, for instance, \cite{outrata,bergqvist1997,brokate1996, krejci1998,visintin1994}. During the last  decade, the study of sweeping process has intensified  due to their natural presence in newly developed applications such as the mobile robot model \cite{colombo2019}, the pedestrian traffic flows model \cite{colombo2019}, and the crowd motion model for emergency evacuation (see, e.g.,  \cite{cao2021} and \cite{cmo}). In these models, the main concern is to  control the state of events efficiently, that is, to optimize a certain objective function over controlled sweeping process. \\
 Due to the unboundedness  and discontinuity of the normal cone, standard results involving differential inclusions cannot be used for sweeping processes.  Extensive literature exists on the question of existence and {\it uniqueness} of an {\it absolutely continuous} or {\it Lipschitz}  solution for the Cauchy problem associated with different forms of the perturbed controlled sweeping process \eqref{sweeping2} 
	in which the constraint $x(t)\in C(t)$ is implicit. Initially, such  results commonly required the   
	{\it absolute or Lipschitz  continuity} of the set-valued map $C(\cdot)$  (see,  e.g., \cite{edmond}). However, motivated by the need to consider set-valued map $C(\cdot)$ for which these conditions are too strong (see \cite{tolstonogov17}), similar results are derived by merely assuming the same conditions on the  $\rho$-{\it truncated} set-valued map $C(\cdot)\cap\rho\bar{B}$ (see e.g., \cite{nacrythibault19,thibault16, tolstonogov17}). 
In \cite{henrion23}, when $C(t)$ is polyhedral, a constraint qualification is shown to be {\it sufficient} for those conditions to be satisfied on the $\rho$-{\it truncated} polyhedral sets.  \\
Numerous efforts have been made to derive existence theory for  {\it optimal} solutions and/or  {\it optimality} conditions in terms of {\it Euler-Lagrange} equation or {\it Pontryagin-type} maximum principle for  optimal control problems driven by variants of \eqref{sweeping2}. The main approach used to solve different versions of such an optimal control problem  is the method of approximation, either  \textit{discrete} (see, e.g., \cite{ccmn,ccmn2,cmo,cmo2,cg,chhm2,chhm,cmn0}), or   \textit{continuous}  (see,  \cite{brokate,depinho1,depinho2,depinho3,nourzeidan,nourzeidan2,nourzeidan3, verachadihassan}). Our focus in this paper is on the latter, and more specifically, on the {\it exponential penalty-type}.\\

\noindent\textbf{Exponential penalty approximation.} The  exponential penalty approximation technique was first used in  \cite{depinho1,pinhoEr} to derive existence of solution of \eqref{sweeping2}, existence of optimal solution  and Pontryagin-type maximum principle for {\it global} minimizers of a Mayer problem over \eqref{sweeping2}, in which  $f$ is {\it smooth},  $C$ is a {\it constant} {\it compact} set defined as  the zero-sublevel set of a $\mathcal{C}^2$-{\it convex} function $\psi$ satisfying a {\it constraint qualification on $\R^n$}, the  initial state-constraint is a set $C_0\subset C$,  and the final state is {\it free}. 
The {\it importance} of this technique resides in  approximating  $N_C(\cdot)$ by the exponential penalty term $\gamma_k e^{\gamma_k\psi(\cdot)}\nabla\psi(\cdot)$ such that the so-obtained approximating dynamic is a {\it standard} control system {\it without} state constraints, but for which the set $C$ is {\it invariant}:
\begin{equation}\label{approxC}
		\dot{x}(t) = f(t,x(t),u(t)) -\gamma_k e^{\gamma_k\psi(x(t))}\nabla\psi(x(t)) \text{ a.e.}, \; \;x(0)=x_0 \in C(0).
	\end{equation} 
The  {\it absence} in \eqref{approxC} of the explicit state constraint, $x(t)\in C$, that is implicitly present in  \eqref{sweeping2}, has also shown to be instrumental in constructing {\it numerical algorithms} for controlled sweeping processes (see \cite{depinhonum,verachadinum,nourzeidan4}). \\
The domain of applicability of the exponential penalization technique for the results  in \cite{depinho1,pinhoEr}  was later enlarged  in \cite{verachadihassan, nourzeidan, nourzeidan3},  to  include {\it strong} local minimizers for  controlled sweeping processes having:  {\it nonsmooth} perturbation $f,$  a {\it final} state constraint set $C_T\subset \mathbb{R}^n$, the cost depends on {\it both} state-endpoints,  a  \textit{constant} sweeping set $C$ that is {\it nonsmooth} (i.e., $C$ is the intersection of a finite number of zero-sublevel sets of $\mathcal{C}^{1,1}$-{\it generators} $\psi_1(x),\cdots ,\psi_r(x)$  near $C$), and the functions $\psi_i$'s satisfy a {\it constraint qualification on} the set $C$.  Furthermore, therein, the normal cone, $N_C$,  in \eqref{sweeping2} is replaced  by a subdifferential, $\partial\varphi$,  of a   function $\varphi$ with domain $C$, and it is shown that such a system  is equivalent  to \eqref{sweeping2}  with  a different $f$.  However, when  $C_T \subsetneq \R^n$, the {\it convexity} of the sets $f(t,x,U(t))$ is  required in \cite{nourzeidan, nourzeidan3}, and   when $C$ is unbounded,  a  {\it restrictive} assumption, (A2.4),  is imposed in \cite{nourzeidan3} on the set $C$ and is shown to hold for  convex, compact boundary, or polyhedral sets, but {\it not}  for {\it general}  prox-regular sets.  In  \cite{depinho2} and later in \cite{depinho3} (independently from \cite{nourzeidan3}), the authors extended, under the {\it compactness} of Gr $C(\cdot)$, their previous {\it smooth} Pontryagin principle for {\it global} minimizers in \cite{depinho1}  to the case where the sweeping set  $C(t)$ is {\it time-dependent} and {\it nonsmooth}. However,   under the Gramian matrix of $(\nabla_x\psi_i(t,x))_{i\in \mathcal{I}^0_{(t,x)}}$ being diagonally dominant,  a  {\it demanding} condition is assumed in \cite{depinho3}, that is,  $\nabla_x\psi_i(t,\cdot)=0$ on the complement in $C(t)$ of a uniform band around the boundary of $C(t)$. The functions $\tilde{\psi}_i$, constructed in \cite[equation 8]{depinho3} to automatically satisfy this latter condition,  does not seem to meet the diagonally dominant condition.  Furthermore, they required   the {\it restrictive} condition:  $ \left<\nabla_x\psi_i(t,x),  \nabla_x\psi_j(t,x)\right>  \ge 0$ in a band around the boundary of $C(t)$,  that is, all the {\it corners} of $C(t)$  must have  {\it obtuse} angles. This last assumption excludes many important sets, including simple ones, like  triangles, polytopes or sets with one or more acute angles, etc. In addition, therein, the sweeping sets  are assumed to have $\mathcal{C}^2$- generators $(\psi_i(t,x))_{i=1}^r$ satisfying a {\it global} constraint qualification, the perturbation $f(t,\cdot,u)$ is smooth, $f(t,x,U)$ is convex, and   the final state  set  $C_T\subset C(T)$  and is compact.\\ 
In \cite{hermosillapalladino}, a {\it different}  approach  is used to establish a  {\it variant} of {\it nonsmooth} Pontryagin maximum principle for  {\it strong} local minimizers of the problem in \cite{depinho3},  in which the initial state is {\it fixed}, the final state is {\it free}, and  $C(\cdot)$ is {\it Lipschitz}. Therein,  instead of the standard {\it nontriviality} condition ($\lambda=1$ in their case),  their maximum principle  has an {\it atypical} nondegeneracy condition.\\ 
\noindent {\bf Conclusion I.} \label{openproblemI}
	Therefore,  the question of establishing a  Pontryagin maximum principle in its {\it expected} form (i.e.,  standard nontriviality condition, adjoint equation, transversality condition, and the maximality condition on the Hamiltonian) 
	for  optimal control problems over the sweeping process \eqref{sweeping2}, remains {\it open} in each of the following settings:  (i) when  the  {\it nonsmooth} {\it moving} sweeping sets $C(t)$  are   {\it bounded} and {\it general} (no restriction);  (ii)   when the  {\it nonsmooth}  sweeping sets  are  {\it unbounded} ({\it constant} or {\it moving}) and are {\it general} (no extra assumptions); 
	(iii)   when {\it joint} state endpoints constraint set is {\it present},  the {\it convexity} of $f(t,x,U(t))$ is {\it absent}, or  the {\it global} constraint qualification is only {\it local}, for all types of sweeping sets:  {\it smooth}, {\it nonsmooth}, {\it constant}, {\it moving}, {\it bounded}, or {\it unbounded}. \\

\noindent\textbf{Coupled model with joint-endpoint constraints.} In addition to the open problems in Conclusion I, new challenges arise when coupling \eqref{sweeping2}   with a standard controlled differential equation, and  when the joint endpoints constraint is on both states, that is, our optimal control problem $(P)$ introduced in Section \ref{assumptions}, is governed by the following coupled dynamic $(D)$,
where $x(t)\in \R^n$, $y(t)\in \R^l$, and $u(t)\in U(t)$ a.e., \vspace{-0.5cm}
\begin{equation*}\vspace{-0.1cm}
	(D)\begin{cases}
		\dot{x}(t) \in f(t,x(t),y(t), u(t)) - N_{C(t)}(x(t)), \text{ a.e. } t \in [0,T], \\
		\dot{y}(t) = g(t, x(t), y(t), u(t)), \text{ a.e. } t \in [0,T], \\
	\end{cases}
	\end{equation*}
	\vspace{-.15 in}
	\begin{equation*}
	\hspace{-1.5 in}(x(0),y(0),x(T),y(T)) \in S.	\end{equation*}
Our model incorporates different controlled submodels as  particular cases:  {\it coupled evolution variational inequalities} (see \cite{outrata},  \cite{bensoussan},  \cite{brokate}),  a subclass of {\it Integro-Differential sweeping processes} of Volterra type (see \cite{bouach1}),  {\it second order sweeping processes}, in which the sweeping set is solely time-dependent (see, e.g., \cite{nacry} for the general setting), and {\it Bolza-type} { problems} associated to $(P)$.
In other words, optimal control problems governed by either of the four submodels can readily be  formulated as a special case of ($P$) to which all the results of this paper are applicable. \\
In \cite{cg},   necessary conditions  in the form of a {\it weak}  maximum principle are derived  for a  certain form of a {\it Bolza}  problem over a sweeping process. Excluding the {\it part} of their {\it integrand} involving $\dot{x}$ that is not covered in our setting,  the remaining problem, therein,   can be phrased as a {\it special} form of our problem $(P)$ over a {\it coupled} sweeping process $(D)$, where the sweeping set  is  a {\it  constant polyhedron},  and the  state endpoints are at most  {\it periodic}.\\ 
On the other hand, in \cite{brokate}, a  smooth Pontryagin maximum principle in its {\it expected} form is derived  for a special case of our problem ($P$), namely, where the sweeping  set is {\it constant}, {\it smooth}, and {\it strictly convex}, the perturbation $f$ is {\it linear} in $u$,  the function $g=(g_1,g_2)$ in the {\it coupled}  controlled differential equations has $g_1$  {\it linear} in $u$ and $g_2$  is {\it  quadratic and convex} in $u$, the   initial  state is {\it fixed}, and the final state is {\it free}.
The authors of \cite{brokate}  clearly noted  that their method  of {\it standard smooth} penalization does not apply 
even for  the case of a {\it constant polyhedron} (which is a particular case of our general  sweeping sets), and  that including an ``{\it additional terminal constraint}'', a fortiori joint endpoints constraint, causes issues that are {\it not} treated therein.\\ 
{\bf Conclusion II.}\label{openproblemII}
Therefore, all the problems stated in Conclusion I are {\it open}  when replacing the sweeping process \eqref{sweeping2} by $(D),$  even when the sweeping set  is {\it constant polyhedral}.\\

\noindent{\textbf{Our findings and results}.}  In this paper, we  solve {\it all} the aforementioned open problems in Conclusions I and II, and we establish existence results for  solutions to ($D$) and ($P$).   
Indeed, the \textit{first} result, which is local, answers collectively all the open questions displayed above and generalizes all previously known results on  Pontryagin maximum principle in multiple ways. More specifically,  in Theorem \ref{pmp-ps} we derive  under {\it minimal} assumptions on the data,  a {\it complete} set of necessary conditions in the form of  {\it nonsmooth}  Pontryagin maximum principle  for a  \textit{strong} local minimizer $((\x,\y),\bar{u})$  of the Mayer problem ($P$) governed by the {\it coupled} sweeping system $(D)$ together with the  {\it joint} endpoints constraint set $S$.  The {\it moving} sweeping  sets $C(t)$ are  {\it general}, {\it nonsmooth}, {\it bounded} or {\it unbounded}, uniformly  {\it prox-regular},  and defined  as the intersection of a finite number of zero sub-level sets of  the generators $(\psi_i(t,\cdot))_{i=1}^r$.\\
The optimal  control problems studied in \cite{depinho3,nourzeidan3} are over \eqref{sweeping2} and not over the general system $(D)$. Noteworthy, unlike the result  derived in  \cite{depinho3}  where the sweeping sets are not only assumed to be  {\it bounded},  but satisfy {\it restrictive} assumptions  on their corners (obtuse angles) and on  the gradients of their generators (satisfying $\nabla_x\psi_i(t,\cdot)=0$ in a zone in $C(t)$  under the diagonal dominance),  {\it no} such {\it restrictive} assumptions are required in our result over ($D$), whether the {\it nonsmooth} moving sweeping sets $C(t)$  are {\it bounded or unbounded}.  While  when  $C(t)\equiv C$  is a {\it constant} set this corner assumption in \cite{depinho3} was removed in \cite{nourzeidan3},  its removal is far more intricate when $C(t)$ are  {\it moving sets} (see Section \ref{approx} and Theorem  \ref{invariance}).   In contrast of the result  in \cite{nourzeidan3}  established for a {\it restrictive} class of {\it constant unbounded}  sweeping sets,  our  result here is valid for {\it general unbounded, moving}, and  {\it prox-regular} sets that do not necessarily satisfy the restrictive assumption (A2.4) of \cite{nourzeidan3}.  In addition, the {\it convexity} assumption of the sets $f(t,x,U(t))$  in \cite{depinho3} and \cite{nourzeidan3} is now  discarded and  not only for the {\it separable}  endpoints  case treated therein, but also for  general {\it joint endpoints constraints}.  Furthermore,  as opposed to the {\it global} constraint qualification on the generators of the sweeping set $C$ in \cite{nourzeidan3} and of $C(t)$ in  \cite{depinho3},  our constraint qualifications  are required to  \textit{only}  hold at $\x(t)$ (see (A3.2) and (A3.3)), where (A3.3) is vacuous in the smooth case ($r=1$). 
Our {\it nontriviality} condition is simply $\lambda+\|p(T)\|=1$ and does not invoke  the measure corresponding to $x(t)\in {C(t)}$.  This is the {\it expected} form  in a Pontryagin  maximum principle for problems over controlled sweeping processes (see \cite{nourzeidan, nourzeidan3, depinho3}).   In our adjoint inclusion and transversality condition we employ  the recently introduced subdifferentials in \cite{nourzeidan}  that are strictly smaller than the Clarke and Mordukhovich subdifferentials.  \\
 The \textit{second} and \textit{third} results are global. The second  provides in Theorem \ref{cauchyglobal} the  existence and uniqueness of a {\it Lipschitz} solution for the {\it Cauchy problem} corresponding to our dynamic $(D)$ without requiring any {\it Lipschitz} behavior on the {\it nonsmooth} moving  sets $C(t)$. Instead,  we assume  Gr $C(\cdot)$ is bounded and  the gradients of the active generators are positively linear independent  $(A3.2)_G$. Note that this is the first result of its kind for  {\it general nonsmooth moving} sweeping sets, even for system \eqref{sweeping2}, that  is  based on the method of exponential penalty approximation. It is essential for developing a {\it numerical algorithm} to solve optimal control problems over such sweeping processes, which is the topic of our forthcoming paper.    
 The third is furnished in {Theorem \ref{existenceglobal}} on the existence of a global {\it optimal solution} for our problem $(P)$ over ($D$) with joint endpoints constraint set $S$ which  justifies the pursuit  of a Pontryagin maximum principle for an optimal solution of $(P)$. \\  
   
 \noindent {\textbf{Novelty of the methods employed.}}  There are {\it three} separate matters to tackle when establishing a Pontryagin maximum principle for a $\bar\delta$-minimizer $((\bar{x},\bar{y}),\bar{u})$  of our problem $(P)$.  The first matter is the  possible  {\it unboundedness} of the moving sweeping sets $C(t)$ and the joint endpoints set $S$, and the unboundedness of $\R^l$ (the sweeping set for the coupled controlled ODE).  The second, which is present {\it even}  if  $C(t)$ is {\it bounded} and/or the sweeping process is taken to be \eqref{sweeping2}  instead of $(D)$, is that the constraint qualification, (A3.2), on the active generators of $C(t)$ is  {\it only} valid at $\bar{x}$. The third is the {\it absence} of  a Pontryagin maximum principle in its expected form for a Mayer problem over \eqref{sweeping2}, where the  {\it nonsmooth} {\it moving} sweeping sets are {\it general} and {\it bounded}.\\
To address the  first and second matters, we first introduce a {\it new} dynamic of sweeping processes, $(\bar{D})$ (see  \eqref{Dtilda}),   obtained from $(D)$ through truncating the given {\it prox-regular} {\it nonsmooth} moving sweeping sets $C(t)$  and the space $\R^l$ by respectively the balls $\bar{B}_{\bar\varepsilon}(\x(t))$ and $\bar{B}_{\bar\delta}(\y(t))$. The radius   $\bar{\varepsilon}\in (0,\bar{\delta})$ is particularily chosen via (A3.2) so that  
the {\it bounded} truncated sets $C(t) \cap \bar{B}_{\bar{\varepsilon}}(\bar{x}(t))$ are {\it uniformly prox-regular}, and their $(r+1)$ generators, $(\psi_i(t,\cdot))_{i=1}^{r+1}$,  satisfy on the graph of $C(\cdot)\cap  \bar{B}_{\bar{\varepsilon}}(\bar{x}(\cdot))$ a  {\it uniform constraint qualification}  (see \eqref{conditionii}); here $\psi_{r+1}(t,x)$,  defined in \eqref{psi-r+1},  is the generator of the  ball $\bar{B}_{\bar{\varepsilon}}(\bar{x}(t)).$  Note that the system $(\bar{D})$ takes the form of \eqref{sweeping2}, where the state's dimension is $n+l$. Next,   in \eqref{Sdeltadelta},  we define for any $\delta \in (0,\bar\varepsilon)$, a compact truncation  $S_{ \delta,\delta}$  of $S$, and we show that the same $((\x,\y),\u)$ is a  $\delta$-{\it strong local minimizer} for the problem $(\bar{P}_{ \delta,\delta})$, whose  objective function is the same as $(P)$, but its dynamic is $(\bar{D})$ and  its  joint endpoint set  is $S_{ \delta,\delta}$ (see Remark \ref{optimality}).  At this point, our attention  {\it completely} shifts  away from  $((\x,\y),\u)$ being   $\bar\delta$-{\it strong local minimizer} for $(P)$ over $(D)$ and $S$, to  being a $\delta$-{\it strong local minimizer} for the problem $(\bar{P}_{ \delta,\delta})$ over $(\bar{D})$ and $S_{ \delta,\delta}$, where the moving sweeping sets    $\bar{\mathscr{N}}_{(\bar\varepsilon, \bar\delta)}(t):= [C(t) \cap \bar{B}_{\bar{\varepsilon}}(\bar{x}(t))]\times \bar{B}_{\bar{\delta}}(\bar{y}(t))$ are  {\it bounded} and  {\it nonsmooth}, and $S_{ \delta,\delta}$ is compact.  Observe that our truncation approach always leads to {\it moving} sweeping sets even when $C(t)\equiv C$ is constant,  and hence, it cannot be employed if we only admit  constant sweeping sets, as in \cite{nourzeidan3}. This  fact is behind the need to  impose in \cite{nourzeidan3} the unnatural assumption (A2.4), which is  not needed here due to our {\it form} of the truncation approach. \\
Our main objective is to obtain the \textit{nonsmooth} Pontryagin maximum principle for $(P)$ via that for $(\bar{P}_{ \delta,\delta})$, for one $\delta \in (0,\bar{\varepsilon})$. However,  in the literature, there is {\it no} Pontryagin maximum principle in its expected form that  applies  to  problems like  $(\bar{P}_{ \delta,\delta})$. This is so,  not only because of the presence of the joint endpoints constraint set, but  because such a result does not exist for  sweeping processes in the form of  \eqref{sweeping2}, where $C(t)$ is  a {\it general}, {\it nonsmooth},  {\it bounded},  and {\it moving} sweeping set (see Conclusion \ref{openproblemI}).  This is the third matter stated above.\\ 
\noindent  To resolve this matter, for a specific choice $\delta_o$ of $\delta$, we establish  a \textit{nonsmooth} Pontryagin maximum principle for  the  $\delta_o$-{\it strong local minimizer}  $((\x,\y),\u)$ of $(\bar{P}_{ \delta_o,\delta_o})$     that does {\it not} require any demanding conditions like the ones in \cite{depinho3} described prior to Conclusion \ref{openproblemI}.   As such, we derive this result  for problems over \eqref{sweeping2} with  {\it general} {\it nonsmooth}, {\it bounded}, and  {\it moving} sweeping sets,  including the {\it joint} state endpoints constraints. This is accomplished 
by developing  a {\it complete} new design of  the method of exponential penalty approximation that drastically differs from the one in \cite{depinho3} and generalizes  the one in \cite{nourzeidan3} to the {\it complex} setting of {\it time-dependent} sweeping sets and to {\it joint} endpoint constraints. Indeed, observe that in \cite{depinho3}, the stringent condition on the corners is imposed so that the exponential penalty method works for {\it nonsmooth} bounded moving sweeping sets in the {\it same} manner as it did  for  {\it smooth} sets in \cite{depinho2}, that is, by forcing  their sweeping set $C(t)$ itself to be {\it invariant} for their approximating dynamic. However, in this paper, we  produce a carefully crafted sequence of {\it smooth} sets, $(\bar{C}^{\gamma_k}(t) \times \bar{B}_{\bar{\rho}_k}(\bar{y}(t)))_k$ (see \eqref{tildeCgkkdef}), approximating   our {\it nonsmooth} moving sweeping set    $\bar{\mathscr{N}}_{(\bar\varepsilon, \bar\delta)}(t):= [C(t) \cap \bar{B}_{\bar{\varepsilon}}(\bar{x}(t))]\times \bar{B}_{\bar{\delta}}(\bar{y}(t))$  from its {\it interior},  and we show that this smooth sequence  (as opposed to the sweeping set itself) is actually {\it invariant} (see Remark \ref{invariant}) for our approximating dynamic $(\bar{D}_{\gk})$,  this is different than \cite{depinho3} and far more intricate than \cite{nourzeidan3}.  Furthermore, in the next step,  to obtain  the boundedness of the multipliers for the approximated normal cones  in their approximating dynamic,  the authors in \cite{depinho3} approximated $C(t)$ from the {\it outside} by a sequence of  invariant sets  (as  they did  for the smooth bounded sets in \cite{depinho2}). Instead,  we  construct another {\it smooth} approximation $\bar{\mathscr{A}}(t,k) $ of $\bar{\mathscr{N}}_{(\bar\varepsilon, \bar\delta)}(t)$, from its {\it interior}, where $\bar{\mathscr{A}}(t,k):=\bar{C}^{\gk}(t,k) \times \bar{B}_{\bar{\rho}_k}(\bar{y}(t)) \subset \inte \bar{C}_{\gk}(t) \times \bar{B}_{\bar{\rho}_k}(\bar{y}(t)) $. We then show that $\bar{\mathscr{A}}(t,k) $  is {\it invariant} for  {\it our} approximated dynamic, the multipliers therein  are {\it bounded}, and its solutions  {\it approximate} the solutions to the original dynamic. As our sweeping sets and their approximations are {\it time dependent}, the derivation of these results turns out to be quite involved and  challenging in comparison with the constant sweeping set $C$  treated in   \cite{nourzeidan3} (see section \ref{approx} and Theorem \ref{theoremrelationship}).\\
Next, to address  the presence of {\it joint} endpoint constraints,  we  pick carefully the value $\delta_o$ of $\delta$  so that we can successfully craft an approximation $S^{\gk}(k)$ of $S_{\delta_o,\delta_o}$ such that for $k$ large, any solution of $(\bar{D}_{\gk})$ with state endpoints  in $S^{\gk}(k)$ remains at all times in the invariant set $\bar{\mathscr{A}}(t,k).$ \\    
At this point, we are ready to design for $(\bar{P}_{ \delta_o,\delta_o})$ an approximating problem $(P^{\alpha,\beta}_{\gk})$, defined  over $(\bar{D}_{\gk}^{\beta})$  (closely related to $(\bar{D}_{\gk})$) and the joint endpoints constraint set $S^{\gk}(k),$ whose optimal solution converges to $((\bar{x},\bar{y}),\bar{u})$, see Proposition \ref{approximating}.   Then, we show  that  the corresponding adjoint variable sequence is {\it uniformly} of {\it bounded variation}, which is an essential requirement for the exponential penalty approximation method to work. We achieve this fact  by introducing  a {\it novel}  approach (see page 44 in the proof of Theorem \ref{pmp-ps}),  given that the  corresponding approaches  employed in \cite{depinho3, nourzeidan3} are not valid.
 A {\it careful} analysis of the limit  as $k\to \infty$ to the {\it standard} maximum principle for $(P^{\alpha,\beta}_{\gk})$ leads to our nonsmooth Pontryagin principle. To remove the convexity assumption on $(f,g)(t,x,y,U(t))$, we extend the {\it relaxation} technique  from \cite{verachadihassan}  to address: \text{(a)} strong local minimizers, \text{(b)} time-dependent sweeping sets, $C(t)$,  not necessarily moving in an absolutely continuous way, and \text{(c)} {\it{general joint state  endpoints}} constraint set $S$. \\
Using our  approach to the exponential penalty method  without truncating  $C(t)$, we prove in Section \ref{global} the {\it existence} and {\it uniqueness} of a Lipschitz solution to the Cauchy problem associated with $({D})$, Theorem \ref{cauchyglobal}. The {\it existence} of an {\it optimal}  solution for the problem $(P)$, Theorem \ref{existenceglobal},   established also  in Section \ref{global},  employs general results developed  in the Appendix. \\
\\
\textbf{Outline of the paper.}	The outline of the paper is as follows. In Section \ref{assumptions}, we display the problem ($P$), the underlying assumptions, and the local and global main results. In Section \ref{auxiliary}, we offer results for the dynamic $(\bar{D})$ with truncated sweeping sets, which are crucial when employing the penalty approximating method to prove the maximum principle. In Section \ref{global}, we ascertain the existence of a solution to the Cauchy problem governed by $(D)$ under global assumptions, and we prove the existence of optimal solution for the problem $(P)$. Section \ref{maximumprinciple} is dedicated to  proving the maximum principle for $(P)$. Finally, Section \ref{examplesec} offers an example demonstrating the importance of our initial model and the practical utility of the results. Section \ref{appendix} serves as an appendix. 
\subsection{Preliminaries}
In this subsection, we present the basic notations and concepts used in this article. \\
\textbf{General notations.} We denote by $\|\cdot\|$ and $\<\cdot,\cdot\>$  the Euclidean norm and the usual inner product, respectively. For $x \in \mathbb{R}^n$ and $a>0$, we denote, respectively, by $B_{a}(x)$ and $\bar{B}_{a}(x)$ the open and closed  ball centered at $x$ and of radius $a$. More particularly, $B$ and $\bar{B}$ represent the open unit ball and the closed unit ball, respectively. The effective domain and epigraph of an extended-real-valued function $f\colon\R^n\f\R\cup\{\infty\}$ are denoted by $\dom f$ and $\epi f$, respectively. A vector function $f=(f_1, \cdots, f_n): [0,T] \f \R^n$ is said to  be positive if $f_i$ is positive for each $i=1, \cdots, n$. We use $\mathscr{M}_{m\times n}[a, b]$ to indicate  the set of $m\times n$-matrix functions  on $[a,b]$. For $r\in\mathbb{N}$, we denote the identity matrix in $\mathscr{M}_{r\times r}$ by I$_{r\times r}$.\\
\textbf{Notations on sets.} The interior, boundary, closure, convex hull, and complement of a set $S \subset \R^n$ are represented by $\inte S$, $\bdry S$, {\text{cl}} ${S}$,  $\conv{S}$ and $S^c$, respectively. 	 For $S\subset\R^n$, we denote by $\sigma(\cdot,S)$ the support function of $S$. 
For $(S_k)_k$ a sequence of nonempty closed convex subsets of $\mathbb{R}^n$, we have, by \cite[Theorem 6]{SW}, that 
\begin{equation}\label{limitsetconvergence}	(S_k)_k\xrightarrow[k \to \infty]{\text{Hausdorff}}  S\;\;\iff\;\;
	\sigma (s^*, S_k)  \xrightarrow[k \to \infty]{\text{unif in} \;s^*}  \sigma (s^*, S), \;\;  \quad \forall  s^* \in \mathbb{R}^n: \|s^*\|\le 1.
	\end{equation}	
{For a set valued-map $S(\cdot): [0,T] \rightsquigarrow \mathbb{R}^n$, Gr $S(\cdot)$ denotes its graph.}\\		
\textbf{Notations on spaces.} For $S\subset\R^n$ compact, $\mathcal{C}(S;\R^n)$ denotes to the set of  continuous functions from  $S$ to $\R^n$. The Lebesgue space of $p$-integrable functions $f\colon [a,b]\f\R^n$ is denoted by $L^p([a,b];\R^n)$, where the norms in $L^p([a,b];\R^n)$ and $L^{\infty}([a,b];\R^n)$ (or $\mathcal{C}([a,b];\R^n)$) are written as $\|\cdot\|_p$ and $\|\cdot\|_{\infty}$, respectively.  The space  $W^{1,p}([a,b];\R^n)$ denotes the set of continuous functions $z\colon [a,b] \to \R^n$ having $\dot{z} \in L^{p}([a,b]; \mathbb{R}^n)$. The set of all functions $z\colon [a,b]\f\R^n$ of bounded variations is denoted by $BV([a,b];\R^n)$. The space $\mathcal{C}^*([a,b]; \R)$ denotes the dual of $\mathcal{C}([a,b];\R)$ equipped with the supremum norm. We denote by   $\|\cdot\|_{\textnormal{T.V.}}$ the induced norm on $\mathcal{C}^*([a,b]; \R)$. By Riesz representation theorem, each element in  $\mathcal{C}^*([a,b]; \R)$ can be interpreted as an element in the space of finite signed Radon measures on $[a,b]$ equipped with the weak* topology.  For  $\nu\in \mathcal{C}^*([a,b]; \R^n)$,  its  support is denoted by $\supp\{\nu\}.$ Denote  $\mathfrak{M}^1_+(S)$  to be the set of probability Radon measures on $S$.       \\
\textbf{Notations from nonsmooth analysis.} For standard references, see the monographs \cite{clarkeold,clsw,mordubook,Rockawets}. Let $S$ be a nonempty and closed subset of $\R^n$, and let $s\in S$. The {\it proximal}, the {\it Mordukhovich} ({also known as \it limiting}), and the {\it Clarke normal} cones to $S$ at $s$ are denoted by $N^P_S(s)$, $N_S^L(s)$, and $N_S(s)$, respectively. 
For  $\rho >0$, a set $S$ is  $\rho$-prox-regular whenever,  $\forall s\in S$, $\forall \xi \in N_S(s)$ with $\|\xi\|=1$, we have
$$ \langle \xi, z-s\rangle \le \frac{1}{2\rho}\|z-s\|^2, \forall z\in S.$$
In this case, $N^P_S(s)=N_S^L(s)= N_S(s)$, for all $s\in S$.
Given a {\it lower semicontinuous} function $f\colon\R^n\f\R\cup\{\infty\}$, and  $x \in \dom f$, the \textit{proximal}, the {\it Mordukhovich} (or \textit{limiting}), and the \textit{Clarke subdifferential}  of $f$ at  $x$ are denoted by  $\partial^P f (x)$, $\partial^L f(x)$, and $\partial f(x)$, respectively. Note that if  $x\in\inte (\dom{f})$ and $f$ is Lipschitz near $x$, \cite[Theorem 2.5.1]{clarkeold} yields that the Clarke subdifferential of $f$ at  $x$ coincides  with the {\it Clarke generalized gradient} of ${f}$ at $x$, also denoted here by $\partial {f}(x)$. If ${f}$ is $\CO^{1,1}$ near $x\in \inte(\dom f)$,  $\partial^2f(x)$ denotes the {\it Clarke generalized Hessian} of ${f}$ at $x$. For  $g\colon\R^n \f \R^n$ Lipschitz near $x\in\R^n$, $\partial{g}(x)$ denotes the {\it Clarke generalized Jacobian} of $g$ at $x$. Note that we will be using in this article {\it nonstandard} notions of subdifferentials that are strictly smaller than the Clarke and Mordukhovich subdifferentials, as seen in Section \ref{assumptions}. We  note that  $\nabla f$ of a function $f$  is taken here to be a column vector, that is, the {\it transpose} of the standard gradient vector. 

\section{Model, Assumptions, and Main results}  \label{assumptions}
The aim of this paper is to derive existence result for the Cauchy problem associated with $(D)$, global existence of optimal solutions and  necessary conditions in the form of a maximum principle for  the fixed time Mayer problem$(P)$, where $(D)$ and $(P)$ are given by the following: 
\begin{equation*}
	(P) \begin{cases}
		\text{minimize} \quad   {J}(x(0), y(0), x(T),y(T)) \\
		\text{over } ((x,y),u) \in W^{1,1}( [0,T], \mathbb{R}^n \times \mathbb{R}^l)\times \mathcal{U} \;\textnormal{ such that } \\
		\hspace{.2 in}(D)\begin{cases}
			\dot{x}(t) \in f(t,x(t),y(t), u(t)) - N_{C(t)}(x(t)), \text{ a.e. } t \in [0,T], \\
			\dot{y}(t) = g(t, x(t), y(t), u(t)), \text{ a.e. } t \in [0,T], \\
		\end{cases} \\
			\hspace{.2 in}(x(0),y(0),x(T),y(T)) \in S, \quad \textbf{ (B.C.)}\\
	\end{cases}
\end{equation*}
where $T>0$ is fixed, $J: \mathbb{R}^n \times \mathbb{R}^l \times \mathbb{R}^n  \times \mathbb{R}^l \longrightarrow \mathbb{R} \cup \{\infty\}, $ $ f: [0,T] \times \mathbb{R}^n \times \mathbb{R}^l \times \mathbb{R}^m  \longrightarrow \mathbb{R}^n,$ \\$ g: [0,T] \times \mathbb{R}^n \times \mathbb{R}^l \times \mathbb{R}^m \longrightarrow \mathbb{R}^l$,
 $C(t)$ is the intersection of the zero-sublevel sets of a finite sequence of functions $\psi_i(t,\cdot)$ where $ \psi_i: [0,T] \times \mathbb{R}^n \longrightarrow \mathbb{R}$, $i=1, \dots,r$,  $N_{C(t)}$ is the {\it Clarke} normal cone to $C(t)$,   $S  \subset C(0) \times \mathbb{R}^{l} \times \mathbb{R}^n \times \mathbb{R}^{l}$ is {\it closed}, $ U(\cdot): [0,T] \rightsquigarrow \mathbb{R}^m$ is nonempty, closed, and  Lebesgue- measurable set-valued map,  and  the set of control functions $\mathcal{U}$ is defined by  \begin{equation}\label{control}
	\mathcal{U}:= \{ u: [0,T] \longrightarrow \mathbb{R}^m: u \text{ is measurable and } u(t) \in U(t), \text{ a.e. } t \in [0,T]\}.
\end{equation}
A pair $((x,y),u)$ is {\it admissible} for $(P)$ if $((x,y),u) \in W^{1,1}( [0,T]; \mathbb{R}^n \times \mathbb{R}^l)\times \mathcal{U}$  satisfies the dynamic $(D)$ and the boundary conditions \textbf{(B.C.)}.  
An admissible pair  $((\bar{x},\bar{y}),\bar{u})$ is said to be a $\bar{\delta}$-{\it strong local minimizer} for $(P)$, for some $\bar{\delta} >0$, if for all $((x,y),u)$ admissible for $(P)$ and satisfying $ \|(x,y)-(\bar{x},\bar{y})\|_{\infty}  \le \bar{\delta}$, we have 
$$J(\bar{x}(0),\bar{y}(0),\bar{x}(T), \bar{y}(T)) \le J(x(0),y(0),x(T),y(T)).$$ 
A {\it combination} of these assumptions is used at different places of the paper.
\begin{enumerate}[label=(\textbf{A\arabic*})]
	\item \textbf{Assumption on $U(\cdot)$:} {The  measurable set-valued map $U(\cdot)$ has compact images}.
		\item \textbf{Assumption on $C(\cdot)$:}  For $t\in[0,T],$ the set $C(t)$  is {\it nonempty}, closed,  {\it uniformly} $\rho$-prox-regular, for some $\rho>0$, and  is given by 
	\begin{equation}\label{C}C(t) :=\bigcap_{i=1}^{r} C_i(t), \; \textnormal{ where } C_i(t):= \{ x \in \mathbb{R}^n: \psi_i(t,x) \le 0\} \subset \mathbb{R}^n, 
	\end{equation}
	where  $ (\psi_i)_{1\le i\le r}$  is a family of {\it continuous} functions $\psi_i: [0,T] \times \mathbb{R}^n \longrightarrow \mathbb{R}$.	\\
	\end{enumerate}
	In organizing this section, we separate the results into local and global results to address different sets of assumptions. Therein, we shall use the following notations. For  $x(\cdot)\in \mathcal{C}([0,T];\R^n)$ such that $ x(t) \in  C(t)$ for all $ t\in [0,T]$, and for $ (\tau, z) \in \text{Gr }C(\cdot)$, we define
\begin{eqnarray}
&&	\hspace{-.3 in} I_i^{\-}(x):=\{t\in [0,T]:  x(t)\in \inte C_i(t)\}\;\;\hbox{and}\;\;I_i^{0}(x):= [0,T]\setminus I_i^{\-}(x), \; \;  \forall i=1,\dots,r, \label{Ibtaui} \\
	&&	I^{\-}(x):=\bigcap\limits_{i=1}^r  I_i^{\-}(x)= \{t\in [0,T]:  x(t)\in \inte C(t)\}, \label{Ibtau}\\ 
&&I^0({{x}}):=\{t\in[0,T] : {x}(t)\in\bdry C(t)\}=[0,T]\setminus I^{\-}(x)=\{t : \I^0_{(t,{x}(t))}\not=\emptyset\}, \label{I0(x)} \\
&&\text{where }\;\; \I^0_{(\tau,z)}:=\{i\in\{1,\dots,r\} : \psi_i(\tau, z)=0\}.\label{Isub(t,x)}
	\end{eqnarray} 	

\subsection{Local assumptions and results}	
For a given pair $(\bar{x}(\cdot),\bar{y}(\cdot))\in \mathcal{C} ([0,T]; \mathbb{R}^n \times \mathbb{R}^{l })$  such that $\bar{x}(t)\in C(t) \; \forall t\in[0,T]$, and for a constant $\bar{\delta}>0$, we say that the following assumptions hold true at $((\bar{x}, \bar{y});\bar{\delta})$ if the corresponding conditions hold true. \\
	 \begin{enumerate}[resume*]
	\item \textbf{Local assumptions on the functions $\psi_i$ at $(\bar{x};\bar{\delta})$:}
	\begin{enumerate}[label=(\textbf{A3.\arabic*})]
	\item There exist   $\rho_o >0$ and $L_{\psi}>0$  such that, for each $i$,  
	$\nabla_x \psi_i(\cdot,\cdot)$ exists  on \\$\Gr\left(C(\cdot)\cap \bar{B}_{\bar{ \delta}}(\bar{x}(\cdot))\right) $ $+\{0\}\times \rho_o B$,  and $\psi_i(\cdot,\cdot)$  and $\nabla_x \psi_i(\cdot,\cdot)$ satisfy, \\
	for all $(t_1,x_1)$, $(t_2,x_2) \in$ $ \Gr \left(C(\cdot)\cap \bar{B}_{\bar{ \delta}}(\bar{x}(\cdot))\right) +\{0\}\times \cfrac{\rho_o}{2} \bar{B}$,
		\begin{eqnarray*}	\hspace{-.5 in}\max\left\{\abs{\psi_i(t_1,x_1)-\psi_i(t_2,x_2)},	\|\nabla_x \psi_i(t_1,x_1)-\nabla_x \psi_i(t_2,x_2) \| \right\} \le L_{ \psi} (\; \abs{t_1-t_2}+ \|x_1-x_2\|).
	\end{eqnarray*} 
\item For every $t \in I^0(\bar{x})$,  {the following constraint qualification at $\x(t)$ holds:}
$$\left[ \sum_{i\in\I^0_{(t,\bar{x}(t))}}\lambda_i\nabla_x\psi_i(t,\bar{x}(t)) =0, \text{ with each } \lambda_i \ge 0\right] \Longrightarrow \left[\lambda_i =0, \; \, \forall i \in  \I^0_{(t,\bar{x}(t))}\right].$$ 
\item There exists a positive Lipschitz function $\bar{\beta}(\cdot)=(\bar{\beta}_1(\cdot), \cdots, \bar{\beta}_r(\cdot)): [0,T] \f \mathbb{R}^r  $ such that  
	$$\hspace{-0.9 in}\sum_{\substack{j\in\I^0_{(t,\bar{x}(t))}\\ j\not=i}}\bar{\beta}_j(t)\abs{\< \nabla_x\psi_j(t,\bar{x}(t)),\nabla_x\psi_i(t,\bar{x}(t))\>} \; < \bar{\beta}_i(t) \|\nabla_x\psi_i(t,\bar{x}(t))\|^2,\;\,   \forall t\in I^0(\bar{x}), \;\; \forall i\in\I^0_{(t,\bar{x}(t))}. $$ 
	\end{enumerate}
\hspace{-1.3cm} 
For the given $\bar{\delta}$  and for any $a,b >0$, we introduce the following sets
\begin{eqnarray}
	&& \hspace{-.3 in} \mathbb{C}_{\bar{x} } := \bigcup\limits_{t \in [0,T]} \left[C(t)  \cap \bar{B}_{\bar{\delta}}(\bar{x}(t)) \right], \quad  \mathbb{B}_{\bar{y} }: =  \bigcup\limits_{t \in [0,T]}  \bar{B}_{\bar{\delta}}(\bar{y}(t)), \;\;\quad \mathbb{U}:=\bigcup\limits_{t \in [0,T]} U(t),\label{mathbbU}\\
	&&\bar{\mathscr{N}}_{(a,b)}(t):= \left[C(t)\cap \bar{B}_{a}(\bar{x}(t)) \right]\times \bar{B}_{b}(\bar{y}(t)), \quad \text{for }\; t\in [0,T], \label{scriptN}	\\
	&& \bar{\mathscr{B}}_{a}:= \bar{B}_{a}(\bar{x}(0))\times\bar{B}_{a}(\bar{y}(0))\times \bar{B}_{a}(\bar{x}(T))\times\bar{B}_{a}(\bar{y}(T)).\label{mathscrB}
		\end{eqnarray} 
\item  \textbf{Local assumptions on  $h(t,x,y,u):=(f,g)(t,x,y, u)$ at $((\bar{x}, \bar{y});\bar{\delta})$:}
		\begin{enumerate}[label=(\textbf{A4.\arabic*})]
\item For  $(x,y,u) \in \mathbb{C}_{\bar{x}}  \times \mathbb{B}_{\bar{y}} \times  \mathbb{U}$, $h(\cdot, x,y  ,u )$ is Lebesgue-measurable and, for a.e. $t \in [0,T] $,
			$ h(t,\cdot,\cdot,\cdot)$  is continuous on  $\bar{\mathscr{N}}_{(\bar{\delta},\bar{\delta})}(t) \times  U(t)$. There exist  $M_h>0$,  and  $L_h \in L^2([0,T]; \mathbb{R}^{+})$,   such that, for {a.e.} $t \in [0,T]$, for all $(x,y), \; (x',y') \in \bar{\mathscr{N}}_{(\bar{\delta},\bar{\delta})}(t)$ and  $u \in U(t)$,
		\begin{eqnarray*}
	\|h(t,x,y,u) \| \le M_h\;\;\text{and}\;\;\| h(t,x,y,u) -h(t,x',y',u)\| \le L_h(t) \|(x,y)-(x',y') \|.
		\end{eqnarray*}
	\item The set $h(t,x,y,U(t))$ is convex for all $(x,y) \in  \bar{\mathscr{N}}_{(\bar{\delta},\bar{\delta})}(t)$ and $t \in [0,T] \text{ a.e. }$\footnote{ This condition is only needed for the existence of an optimal solution,Theorem \ref{existenceglobal}, and not for Theorem \ref{pmp-ps} or  Theorem \ref{cauchyglobal}.}  	
	\end{enumerate}
\item \textbf{Local assumption on $J$ at $((\bar{x}, \bar{y});\bar{\delta})$:} 
 	There exist $\rho_1>0$ and $L_{J} >0$ such that $J$ is $L_{J}$-Lipschitz on $S(\bar{\delta})$, where 
	$$S(\bar{\delta}):= \left( \left[S \cap \bar{\mathscr{B}}_{\bar{\delta}} \right] + \rho_1 \bar{B}\right) \cap \left (\bar{\mathscr{N}}_{(\bar{\delta},\bar{\delta})}(0)  \times\bar{\mathscr{N}}_{(\bar{\delta},\bar{\delta})}(T) \right).$$
\end{enumerate}

The {\it first} main result of this paper provides necessary conditions, in the form of an extended Pontryagin's maximum principle, for a $\bar{\delta} $-strong local minimizer $((\bar{x},\bar{y}),\bar{u})$  for the problem $(P)$. 
We refer the reader to Section \ref{maximumprinciple} for the proof. First, we introduce the following {\it nonstandard} notions of subdifferentials that shall be used in Theorem \ref{pmp-ps}. \\
\begin{itemize}
	\item  $\partial^{(x,y)}_{\ell} h(t,\cdot,\cdot,u)$  denotes  the {\it extended Clarke generalized Jacobian} of $h(t,\cdot,\cdot,u)$  that extends from the interior to the boundary of ${ \bar{\mathscr{N}}_{(\bar{\delta},\bar{\delta})}(t):}= \left[C(t) \cap \bar{B}_{\bar{\delta}}(\bar{x}(t)) \right]\times \bar{B}_{\bar{\delta}}(\bar{y}(t))$ the notion of the Clarke generalized Jacobian (see \cite[Equation(11)] {nourzeidan}), 
	\item $\partial^{xx}_{\ell} \psi_i(t,\cdot)$ is the {\it Clarke generalized Hessian relative to}  ${\inte [C(t) \cap \bar{B}_{\bar\delta}(\x(t))]}$ of $\psi_i(t,\cdot)$ (see \cite[Equation(12)] {nourzeidan}),
	\item $\partial^L_{\ell} J(\cdot,\cdot,\cdot,\cdot)$ is the {\it limiting subdifferential of} $J(\cdot,\cdot,\cdot,\cdot)$ {\it relative to} $\inte S(\bar{\delta})$ (see \cite[Equation(8)] {nourzeidan}).\\
\end{itemize}
\begin{theorem}\label{pmp-ps}\textnormal{\textbf{[Generalized Pontryagin principle for $(P$)]}} \label{maxprincipleS}  Assume that $(A1)$-$(A2)$  are satisfied. Let $((\bar{x},\bar{y}),\bar{u})$ be a $\bar{\delta} $-strong local minimizer for $(P)$  such that $(A3.1)$, $(A3.3)$, {$(A4.1)$} and $(A5)$ are satisfied at $((\bar{x},\bar{y});\bar{\delta})$. Then, {whenever  $(A4.2)$ holds true, or if  sets $U(t)$ are uniformly bounded,} there exist an adjoint vector $p=(q, v)$ with $q \in BV([0,T]; \mathbb{R}^n)$ and $v \in W^{1,2}([0,T]; \mathbb{R}^l)$, finite signed Radon measures $(\nu^i)_{i=1}^r$ on $[0,T]$,  nonnegative functions $(\xi^i)_{i=1}^{r}$ in $L^{\infty}([0,T]; \mathbb{R}^{+})$,  $L^{2}$-measurable functions $\bar{A}(\cdot)$ in $\mathscr{M}_{n \times n}([0,T])$, $\bar{E}(\cdot)$ in $\mathscr{M}_{n \times l}([0,T])$, $\bar{\mathcal{A}}(\cdot)$  in $\mathscr{M}_{l \times n}([0,T])$, and $\bar{\mathcal{E}}(\cdot)$ in $\mathscr{M}_{l \times l}([0,T])$, $L^{\infty}$-measurable functions $(\vartheta^i(\cdot))_{i=1}^r$ in $\mathscr{M}_{n \times n}([0,T])$, and a scalar $\lambda \ge 0$, satisfying the following: 
	\begin{enumerate}[label=$({\roman*})$]
		\item  \textbf{Primal-dual admissible equation}
		\begin{eqnarray*}
			\begin{cases}
				\dot{\bar{x}}(t) = f(t, \bar{x}(t), \bar{y}(t), \bar{u}(t)) - \sum_{i=1}^{r}\xi^i(t) \nabla_x \psi_i(t,\bar{x}(t)) \text{ a.e. } t \in [0,T], \\
				\dot{\bar{y}}(t)= g(t, \bar{x}(t), \bar{y}(t), \bar{u}(t)) \text{ a.e. } t \in [0,T],\\
				\psi_i(t,\bar{x}(t)) \le 0, \; \forall t \in [0,T], \; \; \; \forall i \in \{1, \cdots, r \}.
			\end{cases} 
		\end{eqnarray*}
		\item \textbf{Non-triviality condition} 
		$$\lambda + \|p(T) \|=1.$$
		\item \textbf{Adjoint equations} \\
		For any $z(\cdot) \in \mathcal{C}([0,T], \mathbb{R}^n)$ 
		\begin{eqnarray*}
			&& \int_{[0,T]} \< z(t), dq(t)\> 
			= \int_{0}^{T}  \< z(t), -\bar{A}(t)^T q(t) \> dt +  \int_{0}^{T}  \< z(t),- \bar{\mathcal{A}}(t)^{T} v(t) \>dt \;\; \nonumber \\ &&+ \sum_{i=1}^{r} \left(\int_{0}^{T} \xi^i(t) \< z(t), {\vartheta^i(t)}q(t) \> dt +  \int_{0}^{T}  \< z(t),{\nabla}_x \psi_i(t, \bar{x}(t))\> d\nu^i(t)\right),\label{adjointS}  \\
			\nonumber \\
			&& \dot{v}(t)
			=  -\bar{E}(t)^Tq(t)- \bar{\mathcal{E}}(t)^{T} v(t),
		\end{eqnarray*}
		where  for all $t \in [0,T]$ a.e., \begin{eqnarray*}
			& &  (\bar{A}(t), \bar{E}(t)) \in \partial_{\ell}^{(x,y)} f(t, \bar{x}(t), \bar{y}(t), \bar{u}(t)), \;\;\quad (\bar{\mathcal{A}}(t), \bar{\mathcal{E}}(t)) \in \partial_{\ell}^{(x,y)} g(t, \bar{x}(t), \bar{y}(t), \bar{u}(t)),
			\\
			& & \vartheta^i(t) \in \partial_{\ell}^{xx} \psi_i(t,\bar{x}(t)), \; \text{ for } i=1, \cdots, r. 
		\end{eqnarray*}
		\item  \textbf{Maximization condition}  \begin{equation*}
			\max \limits_{u \in U(t)} \;\bigg \{ \langle q(t), f(t, \bar{x}(t), \bar{y}(t), u) \rangle + \langle v(t), g(t, \bar{x}(t), \bar{y}(t), u) \rangle  \bigg \} 
		\end{equation*}
		is attained at $u=\bar{u}(t)$ for a.e. $t \in [0,T]$. 
		\item \textbf{Complementary Slackness condition} 	For $i=1, \cdots, r$, we have: \\
		$\xi^i(t)=0\; \forall t \in I^{\-}_i(\bar{x}), $  \; and \; $\xi^i(t)  \<\nabla_x \psi_i(t,\bar{x}(t)), q(t) \> =0 \text{ a.e. } t \in [0,T].$			
		\item \textbf{Measures Properties} 	For $i=1, \cdots, r$, we have: \\ $\supp\{\nu^i\}  \subset I^{0}_i(\bar{x})$ \; and  the measure $ \< q(t),\nabla_x \psi_i(t,\bar{x}(t))\> d\nu^i(t)$  is nonnegative.		
		\item  \textbf{Transversality condition} \begin{eqnarray*}
			((q,v)(0), -(q, v)(T))   \in \lambda \; \partial^L_{\ell} J((\bar{x}, \bar{y})(0), (\bar{x}, \bar{y})(T)) +    N_{S}^L((\bar{x}, \bar{y})(0), (\bar{x}, \bar{y})(T)).\quad  
		\end{eqnarray*}
	\end{enumerate}
	In addition, if $S= C_0\times \mathbb{R}^{n+l} $, for a closed set $C_0 \subset C(0)\times\mathbb{R}^l$, then $\lambda=1$, {and the non-triviality condition is discarded}.   
\end{theorem}

\subsection{Global results}{\label{globalmain}}
We now introduce the following {\it global} versions of the previous assumptions that shall be used for the global results in our paper. 
	(\textbf{A3.1})$_G$ and (\textbf{A4})$_G$ are, respectively,  assumptions (A3.1) and (A4) when satisfied for $\bar{\delta}=\infty$ {(the same constants' labels are kept), that is, $\bar{x},\bar{y}$ and the balls around them are not involved therein, and	(\textbf{A3.2})$_G$	 is  the following global version of ({A3.2}) which will imply  the uniform prox-regularity of $C(t)$: }\\
\\
	(\textbf{A3.2})$_G$  \textbf{Global version of} (A3.2){\bf :} For every $(t,x)\in  \text{Gr}\; C(\cdot)$ with $\I^0_{(t,x)}\not=\emptyset$ we have 
	$$\left[ \sum_{i\in\I^0_{(t,x)}}\lambda_i\nabla_x\psi_i(t,x) =0, \text{ with each } \lambda_i \ge 0\right] \Longrightarrow \left[\lambda_i =0, \; \, \forall i \in  \I^0_{(t,x)}\right].$$ 
	
The {\it second}	main result of our paper consists of obtaining the existence and uniqueness of solutions for  the Cauchy problem corresponding to $(D)$ via penalty approximations. Its proof can be found in Section \ref{global}. 
\begin{theorem} \label{cauchyglobal} \textnormal{\textbf{[Existence \& uniqueness of Lipschitz solutions  for  $(D)$]}} 
	Assume that $\psi_i(\cdot,\cdot)$ continuous, and,  for $t\in [0,T]$, $C(t)$ is non-empty, closed,  and given by \eqref{C}. Assume that $(A3.1)$$_G$, $(A3.2)_G$ and $(A4.1)_G$ are satisfied, and that $\Gr C(\cdot)$ is bounded. Given $(x_0,y_0)\in C(0)\times \R^l$ and $u \in \mathcal{U}$, the Cauchy problem corresponding to $(D)$ and $((x(0),y(0));u) = ((x_0,y_0);u)$ has a unique solution $(x,y)$, which is Lipschitz and is the uniform limit of a subsequence (not relabeled) of $(x_{\gk},y_{\gk})_k$,   where  $(x_{\gk},y_{\gk})$ is the solution of a standard control system {corresponding to $u$ with $x_{\gk}(t)\in \inte C(t)$, for all $ t\in [0,T]$}.  \\
\end{theorem}
The {\it third} main result demonstrates the \textit{global} existence of an optimal solution for $(P)$ when the global assumptions are satisfied. The proof can also be found in Section \ref{global}.
\begin{theorem}\label{existenceglobal} \textnormal{\textbf{[Global existence of optimal solutions for $(P)$] }} 
	Assume that $(A1)$ holds, $\psi_i(\cdot,\cdot)$ continuous, and,  for $t\in [0,T]$, $C(t)$ is non-empty, closed,  and given by \eqref{C}. Assume that $(A3.1)$$_G$, $(A3.2)_G$ and $(A4)_G$ are satisfied, and that $\Gr C(\cdot)$ and $\pi_2(S)$ are bounded, where $\pi_2$ is the projection of $S$ into the second component. Let ${J}:\mathbb{R}^n \times \mathbb{R}^l \times \mathbb{R}^n \times \mathbb{R}^l \to \mathbb{R} \cup \{ \infty\}$ be merely lower semicontinuous.  Then, $(P)$ has a global optimal solution if and only if it has at least one admissible pair $(({x},{y}), {u}) $ with $({x}(0),{y}(0),{x}(T),{y}(T)) \in \dom J$. 
\end{theorem}

	\section{Truncated Sweeping Sets}\label{consequence-local}\label{auxiliary}
  	In this section, we provide results that are instrumental for Section \ref{maximumprinciple},  where the penalty approximation method shall be used to prove  the maximum principle for a $\bar{\delta}$-strong local minimizer $((\x,\y),\bar{u})$   (Theorem \ref{pmp-ps}).  To avoid imposing the {\it boundedness} of $\Gr C(\cdot)$ {and a {\it global} constraint qualification on the sweeping sets $C(t)$ of $(D)$, we shall truncate $C(t)$} by a ball around $\x(t)$ of a specific radius $\bar{\varepsilon}$ (that will be determined in Remark \ref{prox-tilde{C}}),   so that the  uniform prox-regularity of $C(t)\cap \bar{B}_{\bar{\varepsilon}}(\x(t))$ is ensured, its  {\it constraint qualification}  is satisfied, and its  normal cone explicit formula is valid (see Remarks \ref{prox-tilde{C}}- \ref{C-prop} and Lemma \ref{final-a3.2(ii)}). In other words, the so-truncated sweeping set is  the sub-level set of $\psi_1(t,\cdot),\cdots,\psi_r(t,\cdot)$, and $\psi_{r+1}(t,\cdot)$, where $\psi_{r+1}$ is given by 	\begin{equation}\label{psi-r+1}
		\psi_{r+1}(t,x)=\psi_{r+1}(t,x;\bar{x},\bar{\varepsilon}):= \dfrac{1}{2}[\|x-\bar{x}(t)\|^2-\bar{\varepsilon}^2].
	\end{equation}  Moreover, the differential equation in $(D)$ will now be replaced by a differential inclusion obtained by adding    $- N_{\bar{B}_{\bar{\delta}}(\bar{y}(t))}$ to its right hand side.  For this purpose,  we define $\varphi:[0,T]\times \R^l \longrightarrow \R$ as
	\begin{equation}\label{phi}
		\varphi(t,y)=\varphi(t,y;\bar{y}, \bar{\delta}):= \dfrac{1}{2}[\|y-\bar{y}(t)\|^2-\bar{\delta}^2],
	\end{equation}
	and we write \begin{equation} \label{ball-phi} \bar{B}_{\bar{\delta}}(\bar{y}(t)):=\{y\in \R^l:\varphi(t,y)\le 0\}.  \end{equation} 
	Denote by $(\bar{D})$  the aforementioned {\it truncated} system obtained from  $(D)$ by localizing $C(\cdot)$ around $\bar{x}$ and $\mathbb{R}^l$ around $\bar{y}$, that is, 
	\begin{equation} \label{Dtilda} 
		(\bar{D})\begin{cases}
			\dot{x}(t) \in f(t,x(t),y(t), u(t)) - N_{C(t) \cap \bar{B}_{\bar{\varepsilon}}(\bar{x}(t))}(x(t)), \text{ a.e. } t \in [0,T], \\
			\dot{y}(t) \in g(t, x(t), y(t), u(t)) -N_{\bar{B}_{\bar{\delta}}(\bar{y}(t))}(y(t)), \text{ a.e. } t \in [0,T]. 
					\end{cases}
	\end{equation}\label{dynamic-tildeD}
	Notice that the  sweeping set  for ``$x$'' in $(\bar{D})$ is 
		$$  C(t) \cap C_{r+1}(t):=C(t) \cap \bar{B}_{\bar{\varepsilon}}(\bar{x}(t)) = \bigcap\limits_{i=1}^{r+1} \{ x \in \mathbb{R}^n: \psi_i(t,x) \le 0\} =C(t) \cap \{ x \in \mathbb{R}^n: \psi_{r+1}(t,x) \le 0\},$$ 
	and hence, it is always generated by at least two functions, while the sweeping inclusion for ``$y$''  is generated by the single function $\varphi$. \\
		\begin{remark}\label{admissibility} 
		Any admissible pair $((x,y),u)$ for $(D)$  such that $(x(t),y(t))\in \bar{\mathscr{N}}_{(\bar{\varepsilon},  \bar{\delta})}(t)$ for all $ t\in [0,T],$ 
		is also  admissible  for $(\bar{D})$.  On the other hand, any admissible pair $((x,y),u)$ for $(\bar{D})$ such that $({x}(t),{y}(t))\in \bar{\mathscr{N}}_{(\delta_1,  \delta_2)}(t)$ with $\delta_1 < \bar{\varepsilon}$ and $\delta_2 < \bar{\delta}$ is also admissible for $(D)$. 
		This is due to the fact that  $({x}(t),{y}(t))\in  {B}_{\bar{\varepsilon}}(\bar{x}(t))\times {B}_{\bar{\delta}}(\bar{y}(t)), \forall t \in [0,T]$, and hence, using the local property of the proximal normal cone, we have
		$N_{C(t)}^P({x}(t))=N_{C(t) \cap \bar{B}_{\bar{\varepsilon}}(\bar{x}(t)) }^P({x}(t))$ and $\{0\}= N^P_{\bar{B}_{\bar{\delta}}(\bar{y}(t))}({y}(t))$. In particular, $((\bar{x}, \bar{y}), \bar{u})$ solves $(D)$ if and only if it solves $(\bar{D})$. \\
	\end{remark}
For  $x(\cdot)\in \mathcal{C}( [0,T]; \mathbb{R}^n)$ such that $x(t)\in C(t) \cap \bar{B}_{\bar{\varepsilon}}(\bar{x}(t))$ $\forall t \in [0,T]$, and  $(\tau,z) \in $ Gr $\left (C(\cdot) \cap \bar{B}_{\bar{\varepsilon}}(\bar{x}(\cdot)) \right)$, we define the following  sets obtained through adding to  those in  \eqref{Ibtaui}-\eqref{Isub(t,x)}   the extra constraint produced by $\psi_{r+1}$:
\begin{eqnarray}
	&&	I_{r+1}^{\-}(x):=\{t\in [0,T]:  x(t)\in B_{\bar{\varepsilon}}(\bar{x}(t))\}\;\;\hbox{and}\;\;I_{r+1}^{0}(x):= [0,T]\setminus I_{r+1}^{\-}(x), \nonumber\; \;  \\
	&&	\bar{I}^{\-}(x):=\bigcap\limits_{i=1}^{r+1}  I_i^{\-}(x)= \{t\in [0,T]:  x(t)\in \inte C(t) \cap B_{\bar{\varepsilon}}(\bar{x}(t))\}={I}^{\-}(x) \cap \{t\in [0,T]:  x(t)\in B_{\bar{\varepsilon}}(\bar{x}(t))\}, \nonumber \\
	&&	\bar{I}^{0}(x)=[0,T]\setminus \bar{I}^{\-}(x) = {I}^{0}(x) \cup \{ t\in [0,T]:  \|x(t)-\bar{x}(t)\| = \bar{\varepsilon} \}	=\{t\in [0,T]: \bar{\I}^0_{(t,x(t))}\neq\emptyset\}, \nonumber\\
&&\text{where }\bar{\I}^0_{(\tau,z)}:=\{i\in\{1,\dots,r, r+1\} : \psi_i(\tau, z)=0\}  \label{newI}.	
\end{eqnarray} 	
Since  $\bar{x}(t)\in B_{\bar{\varepsilon}}(\bar{x}(t))$,  then $\psi_{r+1}(t,\bar{x}(t))<0$ and hence, 
$\bar{I}^0(\bar{x})=I^0{(\bar{x})} $ and, for $t\in \bar{I}^0(\bar{x})$, $\bar{\I}^0_{(t,\bar{x}(t))}=\I^0_{(t,\bar{x}(t))}.$
	\subsection{Preparatory results  for  truncated sweeping sets}\label{truncated}
	In this subsection, we present some properties pertaining to both $C(t)$ and $C(t) \cap \bar{B}_{\bar{\varepsilon}}(\bar{x}(t))$ as well as  the  sweeping processes $(D)$ and $(\bar{D})$. 	\\
	 The following lemma provides an equivalent condition to (A3.2) which allows to obtain the formula for  the normal cone  to  $C(t)$  at points $x$ in $C(t)$ near $\bar{x}(t)$ (Remark \ref{normal-C}).   
	\begin{lemma}\label{final-a3.2(i)} \textnormal{\textbf{[Assumption (A3.2)]}} Let $C(\cdot)$ satisfying $(A2)$ for $\rho>0$. Consider  $\bar{x}(\cdot)\in \mathcal{C}([0,T];\mathbb{R}^n)$ with $\bar{x}(t)\in C(t)$ for all $ t \in [0,T]$, and $\bar{\delta}>0$  such that $(A3.1)$ holds at $(\bar{x};\bar{\delta})$. Then, the validity of assumption $(A3.2)$  at $\bar{x}$ is equivalent to the existence of $0<\varepsilon_o<\bar{ \delta}$ and  $\eta_o >0$  such that  
			\begin{equation}\label{3.1(i)} \left\|\sum_{i\in\I^0_{(t,c)}}\lambda_i\nabla_x\psi_i(t,c)\right\|>2\eta_o, \;\;  \forall (t,c)\in 
				\left\{(\tau,x)\in \text{Gr} \,\left(C(\cdot) \cap\bar{B}_{\varepsilon_o}(\bar{x}(\cdot)\right): \I^0_{(\tau,x)}\not=\emptyset\right\},	\end{equation} 
			where $\I^0_{(\tau,x)}$ is defined in \eqref{Isub(t,x)} and $(\l_i)_{i\in\I^0_{(t,c)}}$ is any sequence of nonnegative numbers satisfying $\sum_{i\in\I^0_{(t,c)}}\l_i=1. $
	\end{lemma}
	
	\begin{proof} It suffices to show that (A3.2) implies $\eqref{3.1(i)}$. If not, then there exist sequences $t_n\in [0,T]$, $c_n\in C(t_n)\cap \bar{B}_{\frac{1}{n}}(\bar{x}(t_n))$   with $\I^0_{(t_n,c_n)} \neq\emptyset$, and  $(\l_i^n)_{i\in\I^0_{(t_n,c_n)}}$ with $\sum_{i\in\I^0_{(t_n,c_n)}}\l_i^n=1$ and $\l_i^n\ge 0$, for all $i\in \I^0_{(t_n,c_n)}$ and $n\in \N$, such that
		$\left\|\sum_{i\in\I^0_{(t_n,c_n)}}\lambda_i^n\nabla_x\psi_i(t_n,c_n)\right\|\le \frac{2}{n}, \forall n\in\mathbb{N}.$ As up to a subsequence,   $(t_n, c_n)\rightarrow (t_o,c_o):=(t_o, \bar{x}(t_o))$, Lemma \ref{lemmaaux1}  yields the existence of $ \emptyset\not=\J_o\subset \{1,\dots,r\}$ and a subsequence of  $(t_n, c_n)_n$  we do not relabel, such that  $\I^{0}_{(t_n,c_n)}= \J_o\subset \I^{0}_{(t_o,c_o)}\;\,\hbox{for all}\;\,n\in\N.$ It follows that
		\begin{equation}\label{v1}
			\left\|\sum_{i\in\J_o}\lambda_i^n\nabla_x\psi_i(t_n,c_n)\right\|\le \frac{2}{n} \text{,}\;\; \sum_{i\in\J_o}\l_i^n=1 \; (\forall n \in \N) ,\;\;  \text{and,}\;\; \l_i^n\ge 0 \;(\forall i\in \J_o,  \forall\; n\in\N). \end{equation}
		Hence, after going to a subsequence if necessary, it follows  that for all $i\in \J_o$, $\l_i^n\rightarrow \l_i^o\ge0$ with $\sum_{i\in\J_o}\l_i^o=1$. Upon taking the limit as $n\rightarrow\infty$ in \eqref{v1} and by defining $\lambda_i^0=0$ for all $i\in \I^0_{(t_o,c_o)} \setminus \mathcal{J}_o$, (A3.1) implies $\sum_{i\in\I^0_{(t_o,c_o)}}\lambda_i^o\nabla_x\psi_i(t_o,c_o)=0$, which contradicts (A3.2).
\end{proof}
	
		\begin{remark}\label{normal-C}
		Let $C(\cdot)$ satisfying $(A2)$ for $\rho>0$. Consider  $\bar{x}(\cdot)\in \mathcal{C}([0,T];\mathbb{R}^n)$ with $\bar{x}(t)\in C(t)$ for all $ t \in [0,T]$, and $\bar{\delta}>0$  such that $(A3.1)$ and $(A3.2)$ hold at $(\bar{x};\bar{\delta})$. Let  $\varepsilon_o$   be the constant from Lemma \ref{final-a3.2(i)}. Then,   the prox-regularity assumption of $C(t)$, equation \eqref{3.1(i)}, \cite[Corollary 4.15]{prox}, and \cite[Corollary 10.44]{clarkebook},  yield  that\\
		$$ N_{C(t)}(x)= N_{C(t)}^P(x)= N_{C(t)}^L(x),\quad \forall x\in C(t),$$
		and,  for all  $(t,x)\in \text{Gr} \,\left(C(\cdot) \cap\bar{B}_{\varepsilon_o}(\bar{x}(\cdot)\right)$,			
		\begin{equation} \label{normalcone-C}
			N_{C(t)}(x) =
			\begin{cases} \Bigg\{\sum_{i\in\mathcal{I}^0_{(t,x)}}\lambda_i\nabla_x \psi_i(t,x): \lambda_i \ge 0\Bigg\} \neq\{0\} \; & \textnormal{ if } x \in \textnormal{bdry }C(t) \\
				\{0 \} \; & \textnormal{ if } x \in \textnormal{int }C(t).
			\end{cases}
		\end{equation}
		
	\end{remark}	
	 Next we show that,  in addition to providing  in Remark \ref{normal-C} the formula  for the normal cone to $C(t)$,  Lemma  \ref{final-a3.2(i)} is also essential in proving a key result, namely,    the closed graph property of $N_{C(\cdot)}(\cdot)$ in the domain where \eqref{normalcone-C} is valid.	 	 
	 	\begin{lemma}\label{closed} Let $C(\cdot)$ satisfying $(A2)$ for $\rho>0$. Consider  $\bar{x}(\cdot)\in \mathcal{C}([0,T];\mathbb{R}^n)$ with $\bar{x}(t)\in C(t)$ for all $ t \in [0,T]$, and $\bar{\delta}>0$  such that $(A3.1)$ and $(A3.2)$ hold at $(\bar{x};\bar{\delta})$. Then,  for $\varepsilon_o$ obtained in Lemma \ref{final-a3.2(i)}, the set-valued map $(t,y) \to N_{C(t)}(y)$  
	has closed graph on the set   Gr\,$\left(C(\cdot)\cap\bar{B}_{\varepsilon_o}(\bar{x}(\cdot))\right)$.
\end{lemma}
\begin{proof}
	Let $v_n \in N_{C(t_n)}(y_n)$ such that $v_n \to v_o$ and $(t_n, y_n) \to (t_o, y_o)$ in Gr\,$\left(C(\cdot)\cap\bar{B}_{\varepsilon_o}(\bar{x}(\cdot))\right)$. We shall prove that $v_o \in N_{C(t_o)}(y_o)$.  If $v_o=0$ then obviously $v_o \in N_{C(t_o)}(y_o)$.  Now, let $v_o\neq 0$, then for $n$ large enough, $v_n\neq 0$, and hence,  equation (\ref{normalcone-C}) implies that  $ y_n \in \bdry C(t_n)$  and 
 $v_n=\sum_{i\in\I^0_{(t_n,y_n)}}\lambda_i^n\nabla_x\psi_i(t_n,y_n)$ for some $ (\lambda_i^n)_i \ge 0 $. By Lemma \ref{lemmaaux1}, we deduce the existence of $ \emptyset\not=\J_o\subset \{1,\dots,r\}$ and a subsequence of  $(t_n, y_n)_n$  we do not relabel, such that we have
	$\I^{0}_{(t_n,y_n)}= \J_o\subset \I^{0}_{(t_o,y_o)}\;\,\hbox{for all}\;\,n\in\N.$
	Hence, for $n$ large enough, $v_n=\sum_{i\in \J_o}\lambda_i^n\nabla_x\psi_i(t_n,y_n)$ and  $\sum_{i\in \J_o}\lambda_i^n >0$.	
	Define, for each $i\in\J_o$, the bounded sequence $(\beta_{i}^{n})_n$, where $\beta_{i}^{n}:=\frac{\lambda_{i}^{n}}{\sum_{j \in \J_o} \lambda_j^{n} }\ge 0$. Since also $\sum_{i \in \J_o}\beta_{i}^{n}=1$ for all $n$, then for each $i\in\J_o$, along a subsequence (we do not relabel), $\beta_{i}^{n} \to \beta_{i}\ge 0$ with $\sum_{i \in \J_o}\beta_{i}=1$. Using $(A3.1)$ and Lemma \ref{final-a3.2(i)}, we have  $0\neq \sum_{i\in \J_o}\beta_i^n\nabla_x\psi_i(t_n,y_n) \to \sum_{i\in \J_o}\beta_i\nabla_x\psi_i(t_o,y_o) \neq 0$.  By writing $$v_n= \big(\sum_{j \in \J_o} \lambda_j^n \big)  \big (\sum_{i\in \J_o}\beta_i^n\nabla_x\psi_i(t_n,y_n) \big),$$ and using the fact that $v_n \to v_o\neq 0$, we deduce that $ \sum_{j \in \J_o} \lambda_j^n $ is convergent to a limit $\beta_o> 0$. Hence, $v_o=  \sum_{i\in \J_o}\beta_o\beta_i\nabla_x\psi_i(t_o,y_o). $ Now, define $$\alpha_i:=\begin{cases}
		\beta_o \beta_i & \text{if } i \in \J_o\\
		0 & \text{if } i \in \I^{0}_{(t_o,y_o)} \setminus \J_o.
	\end{cases}$$ 
	Then, $v_o=\sum_{i\in \I^{0}_{(t_o,y_o)}}\alpha_i\nabla_x\psi_i(t_o,y_o) \in N_{C(t_o)}(y_o). $ 	
\end{proof}
 {Combining Lemma \ref{closed} with Lemma \ref{proxregular}  immediately produces a range for  $\bar{\varepsilon} >0$  ensuring  the uniform prox-regularity of the truncated sets $C(t)\cap \bar{B}_{{\bar{\varepsilon}}}(\bar{x}(t))$.  
 
 \begin{remark}\label{prox-tilde{C}}Let $C(\cdot)$ satisfying $(A2)$ for $\rho>0$. Consider  $\bar{x}(\cdot)\in \mathcal{C}([0,T];\mathbb{R}^n)$ with $\bar{x}(t)\in C(t)$ for all $ t \in [0,T]$, and $\bar{\delta}>0$  such that $(A3.1)$ and $(A3.2)$ hold at $(\bar{x};\bar{\delta})$.	  Then, for $\bar\varepsilon\in(0, \rho)\cap (0,\varepsilon_o]$, where $\varepsilon_o$ is given   in Lemma \ref{final-a3.2(i)},  there exists $\rho_{{\bar\varepsilon}}>0$, obtained from Lemma  \ref{proxregular},  such that for all $t\in[0,T], \;\; C(t) \cap \bar{B}_{\bar\varepsilon}(\bar{x}(t))$  is 
 ${\rho}_{\bar\varepsilon}$-prox-regular. 
\end{remark}}
Another  important consequence of Lemma \ref{final-a3.2(i)}  and Remark \ref{normal-C} is manifested in the following result that establishes the Lipschitz continuity and the uniqueness of the solutions near $(\bar{x}, \bar{y})$ for the Cauchy problem of $(D)$ via its equivalent form. We note that, under global assumptions, the existence of a  solution for the Cauchy problem of $(D)$ is given in Theorem \ref{cauchyglobal}, which will be established  in Section \ref{global}.  First, define  $\mu$  to be 
	\begin{equation}\label{def-mu}
		{\mu}:= 	L_\psi (1+M_h).
	\end{equation}
	\begin{lemma}\label{lipschitz-local}  Let $C(\cdot)$ satisfying $(A2)$ for $\rho>0$.  Consider  $(\bar{x}(\cdot),\bar{y}(\cdot))\in \mathcal{C}([0,T];\mathbb{R}^n \times \mathbb{R}^l)$ with $\bar{x}(t)\in C(t)$ for all $ t \in [0,T]$, and $\bar{\delta}>0$  such that $(A3.1)$ and $(A3.2)$ hold at $(\bar{x};\bar{\delta})$, and $(A4.1)$ is satisfied by $(f,g)$ at $((\bar{x},\bar{y});\bar{\delta})$. Let $u\in \mathcal{U}$ and $(x_0,y_0)\in  \bar{\mathscr{N}}_{(\varepsilon_0, \bar{\delta})}(0)$   be fixed.  Then, a pair $(x(\cdot),y(\cdot)) \in  W^{1,1}([0,T];\R^{n+l})$ such that $(x(t),y(t)) \in \bar{\mathscr{N}}_{(\varepsilon_0, \bar{\delta})}(t) \; \forall t \in [0,T]$ is a solution of  $(D)$ corresponding to $((x_0,y_0),u)$   if and only if  there exist measurable  functions $(\lambda_1, \cdots, \lambda_{r})$ such that, for all $i=1,\cdots,r$,   $\lambda_i(t)= 0$  for $ t\in I_i^{\-}(x)$,  and $((x,y),u)$ together with  $(\lambda_1,\cdots,\lambda_r)$ satisfies	
	\begin{eqnarray}\label{lipschitz-D}	\begin{cases}					
			& \dot{x}(t)=f(t,x(t),y(t),u(t))-\sum_{i\in \I^0_{(t,{ x}(t))}} \lambda_i(t) \nabla_x\psi_i(t,x(t))\;\;\textnormal{a.e.}\; t\in[0,T], \\
			&\dot{y}(t) = g(t, x(t), y(t), u(t)), \text{ a.e. } t \in [0,T],\\
			& (x(0),y(0))=(x_0,y_0).
		\end{cases}
	\end{eqnarray}
	Furthermore, we have the following bounds \begin{eqnarray}\label{bound} \begin{cases}
			& \|\lambda_i\|_{\infty}\le \|\sum_i ^r\lambda_i\|_{\infty} \le \frac{\mu}{4 \eta_0^2},\;\; \forall i=1,\cdots,r,\\
			&	\|\dot{x}\|_{\infty} \le M_h + \frac{\mu}{4 \eta_0^2}  L_{\psi}, \;\; \|\dot{y}\|_{\infty}\le  M_h.  \end{cases}
	\end{eqnarray}
	{Consequently,   the pair $(x,y)$  is the unique solution of $(D)$ in $\bar{\mathscr{N}}_{({\varepsilon_0},\bar{\delta})}(\cdot)$ corresponding to  $((x_0, y_0), u)$.   In particular,  if $((\x,\y),\bar{u})$ solves $(D)$, then $(\bar{x}(\cdot),\bar{y}(\cdot))$ is Lipschitz and  is the unique solution of  $(D)$ corresponding to  $ ((\x(0),\y(0)), \u)$.} 
		\end{lemma} 	
	\begin{proof}
		The equivalence in the first part of this lemma follows  immediately from Filippov selection theorem and the normal cone formula in \eqref{normalcone-C}. 
		Now, we proceed to prove the bounds in \eqref{bound}. Since  for all $i=1,\cdots,r$,   $\psi_i(\cdot,x(\cdot))\in W^{1,1}$, then $\frac{d}{dt} \psi_i(t, x(t))$ exists for almost all $t\in[0,T]$.   Using assumption $(A3.1)$ and \cite[equation (3.1)]{zeidan}, we deduce that,   $\forall i=1,\cdots,r$, 
		\begin{eqnarray*} \label{inclusion-vera} \frac{d}{dt} \psi_i(t,x(t)) \subset \partial_{(t,x)} \psi_i(t,x(t)). (1, \dot{x}(t)) =  \hat{\partial}_t \psi_i(t,x(t)) + \langle\nabla_x \psi_i(t,x(t)), \dot{x}(t) \rangle,  \; t\in[0,T]\; \text{a.e.,}
		\end{eqnarray*}
where,  for $(t,z)\in \text{Gr} (C(\cdot)\cap \bar{B}_{\bar{\delta}}(\x(\cdot)))$, 
\begin{equation}\label{deltahat} \hat{\partial}_t \psi_i(t,z):= \text{conv}\{\lim_{j\to\infty}\nabla_t\psi_i(t_j,z_j): (t_j,z_j)\to (t, z)\}.\end{equation}
Thus,  there exist measurable  $\theta_i(\cdot) \in \hat{\partial}_t \psi_i(\cdot,x(\cdot))$ a.e., such that 
	\begin{equation}\label{dpsi-r}\frac{d}{dt} \psi_i(t,x(t)) = \theta_i(t) + \langle \nabla_x \psi_i(t,x(t)), \dot{x}(t) \rangle\;\; \; t\in[0,T]\; \text{a.e.,}\; \;\forall i=1,\cdots,r.		\end{equation}
Note that,  by  $($A3.1$)$, we have,   for   $t\in [0,T]$ a.e., and $\forall  \theta_i(t)\in \hat{\partial}_t \psi_i(t,x(t))$, 
		\begin{eqnarray} \label{mu}
\hspace{-.5 in} &&\hspace{-.5in} \left|\theta_i(t) + \langle \nabla_x \psi_i(t,x(t)), {f}(t,x(t), y(t), u(t)) \rangle \right|\le   L_\psi(1+     \|f(t,x(t), y(t), u(t))\|), \; \forall i=1,\cdots,r. \end{eqnarray}
	Define in $[0,T]$ the set	of full measure:  
	\begin{equation}\label{scriptT}\mathscr{T}:= \{t\in (0,T):  \dot{x}(t) \; \text{and}\;\frac{d}{dt} \psi_i(t, x(t))  \; \text{exist}, \; \forall i=1,\cdots,r\}.\end{equation} 
Let $t \in I^{\-}({{x}})\cap \mathscr{T}$.  Then, $\I^0_{(t,{ x}(t))} = \emptyset$, and hence,   $\forall i=1,\cdots,r, \; \lambda_i(t)=0$. This implies that $\dot{x}(t)=f(t,x(t),y(t),u(t))$, and hence $\| \dot{x}(t)   \| \le M_h.$ \\
\\		
Let $t\in I^0({{x}})\cap \mathscr{T}$ with $ \sum_{i\in  \I^0_{(t,x(t))}} \lambda_i(t)  \ne 0$; otherwise we join the conclusion of the previous case.  Since for all $ i \in  \I^0_{(t,x(t))}$, we have $\psi_i(t,x(t)) =0$ and $x(s)\in C(s)\; \forall s\in [0,T]$, it follows that  
		$ \frac{d}{dt} \psi_i(t, x(t))=0$, for all $ i \in  \I^0_{(t,x(t))}$. Hence, for the
		finite sequence $(\theta_i)_{i=1}^r$ in \eqref{dpsi-r},  we have 
		\begin{equation} \label{deripsi22}
			0= \theta_i(t) + \langle \nabla_x \psi_i(t,x(t)), \dot{x}(t) \rangle.\;\; 
		\end{equation}
		Multiplying \eqref{deripsi22} by $\lambda_i(t)$, and using the fact that $x(\cdot)$ satisfies the first equation of \eqref{lipschitz-D}, we get that 
		\begin{eqnarray}
		\hspace{-.5 in}	0 	=   \lambda_i(t) \theta_i(t) + \lambda_i(t) \left\< \nabla_x \psi_i(t,x(t)), f(t,x(t),y(t),u(t)) - \sum_{j\in\mathcal{I}^0_{(t,x(t))}}\lambda_j(t)\nabla_x \psi_j(t,x(t)) \right\>. \label{lambdai}
		\end{eqnarray}
		Summing  \eqref{lambdai} over all $ i \in \I^0_{(t,x(t))}$ and using \eqref{mu}, we deduce that	
				\begin{eqnarray*}
		\hspace{-.2 in}	\|   \sum_{i\in\mathcal{I}^0_{(t,x(t))}}\lambda_i(t)\nabla_x \psi_i(t,x(t))  \|^2	&=&  \sum_{i\in  \I^0_{(t,x(t))}} \lambda_i(t) \left (\theta_i(t) +  \left\< \nabla_x \psi_i(t,x(t)), f(t,x(t),y(t),u(t))  \right\> \right)\nonumber \\
		\hspace{-.2 in}	& \le& L_{\psi}(1+  \|f(t,x(t),y(t),u(t))\|)  \sum_{i\in  \I^0_{(t,x(t))}} \lambda_i(t).  		
		\end{eqnarray*}
		Hence, utilizing  \eqref{3.1(i)} on the term on the left hand side, and then dividing by  $ \sum_{i\in  \I^0_{(t,x(t))}} \lambda_i(t) \ne 0$ the last inequality, we deduce from  \eqref{def-mu} that
		\begin{eqnarray}
			&& \sum_{i\in  \I^0_{(t,x(t))}} \lambda_i(t) 
			\le \frac{ L_{\psi}}{4\eta_0^2}(1+  \|f(t,x(t),y(t),u(t))\|)\overset{(A4.1)} {\le} \frac{\mu}{4 \eta_0^2}.  \label{lambda-bound}		\end{eqnarray}
		Therefore,  $\| \sum_{i=1}^r \lambda_i \|_{\infty} \le \frac{\mu}{4 \eta_0^2}$.  Finally,   employing  (A4.1) for  $f$  and $g$, along with \eqref{lipschitz-D},  the bounds on $\| \dot{x}\|_{\infty}  $ and  $\|\dot{y}\|_{\infty}$ follow.
		\\
		For the uniqueness, let $X:=(x, y),$ $\tilde{X}:=(\tilde{x}, \tilde{y})$  in $\bar{\mathscr{N}}_{(\varepsilon_0,\bar\delta)}(\cdot)$ be two solutions of $(D)$ corresponding to $((x_0, y_0), u)$,  and let    $(\lambda_i)_{i=1}^{r}, (\tilde{\lambda}_i)_{i=1}^{r}$ be their corresponding multipliers satisfying \eqref{lipschitz-D}. 
		Using the hypomonoticity of the normal cone to the $\rho$-prox-regular sets $C(t)$,  the $L_h$-Lipschitz property of $h(t,\cdot,\cdot,u(t))$, and the bounds in \eqref{bound} for the multipliers,  we deduce that
		\begin{eqnarray}
			&&	\frac{1}{2} \frac{d}{dt}(\|X(t)-\tilde{X}(t)\|^2)  =  \langle \dot{X}(t)-\dot{\tilde{X}}(t), X(t)-\tilde{X}(t) \rangle \nonumber\\
			& \le &{({L}_h(t) + {\frac{{\mu}}{4{\rho}{\eta}_o^2}}L_{\psi})} \|X(t)-\tilde{X}(t)\|^2:=\kappa(t)\|X(t)-\tilde{X}(t)\|^2.\label{gron}
		\end{eqnarray}
		Hence using Gronwall's lemma, we deduce that 
		$$\|X(t)-\tilde{X}(t)\|^2 \le e^{2\int_{0}^{t}\kappa(s)ds} \|X(0)-\tilde{X}(0)\|^2 =0.$$ 
		Then, $X(t)=\tilde{X}(t) \quad \forall t \in [0,T],$ and the uniqueness is proved. 	\end{proof}
	
	{The following lemma provides  a second condition, \eqref{conditionii}, equivalent to (A3.2) which, unlike \eqref{3.1(i)},   validates  the formula for  the normal cone to  the uniform prox-regular  truncated sets $C(t)\cap \bar{B}_{{\bar\varepsilon}}(\bar{x}(t)),$  obtained in Remark \ref{prox-tilde{C}}, (see  Remark \ref{C-prop} stated below).}  Note that since $\psi_{r+1}$, given by \eqref{psi-r+1}, is a function of $\bar{\varepsilon}$, this lemma is of a different nature than Lemma \ref{final-a3.2(i)}. Observe that,  for any given $\bar{\varepsilon}>0$, we have that $\psi_{r+1}(t,x)$ and $\nabla_x\psi_{r+1}(t,x):=x-\x(t)$ exist and continuous everywhere.
	\begin{lemma}\label{final-a3.2(ii)} \textnormal{\textbf{[Assumption (A3.2)]}} Let $C(\cdot)$ satisfying $(A2)$ for $\rho>0$. Consider  $\bar{x}(\cdot)\in \mathcal{C}([0,T];\mathbb{R}^n)$ with $\bar{x}(t)\in C(t)$ for all $ t \in [0,T]$, and $\bar{\delta}>0$  such that $(A3.1)$ holds at $(\bar{x};\bar{\delta})$. Then, $(A3.2)$ is satisfied at $\bar{x}$ if and only if for  $\bar{\varepsilon}\in (0,\rho)\cap (0,\varepsilon_o]$ and its corresponding $\psi_{r+1}$ given by \eqref{psi-r+1}, there exists $\bar{\eta} \in (0, \eta_0)$  (without loss of generality  $\bar{\eta}\le \frac{\bar{\varepsilon}}{2} $) such that
			\begin{equation}{\label{conditionii}}      \hspace{-.2 in}\left\|\sum_{i\in\bar\I^0_{(t,c)}}\lambda_i\nabla_x\psi_i(t,c)\right\|>2\bar\eta, \;\;  \forall (t,c)\in 
				\{(\tau,x)\in \text{Gr} \,\left(C(\cdot) \cap\bar{B}_{\bar\varepsilon}(\bar{x}(\cdot)\right): \bar\I^0_{(\tau,x)}\not=\emptyset\},\end{equation}
			where  $(\l_i)_{i\in\bar\I^0_{(t,c)}}$ is any sequence of nonnegative numbers satisfying $\sum_{i\in\bar\I^0_{(t,c)}}\l_i=1$, and $\bar\I^0_{(\tau,x)}$ is given by \eqref{newI}. 
	
	\end{lemma}
	
	\begin{proof} 
		 We only need to show that $(A3.2)$ yields  \eqref{conditionii}.  For this, assume $(A3.2)$ is valid and let $\bar\varepsilon\in(0,\rho)$ with    $\bar\varepsilon\le\varepsilon_o$. From Lemma \ref{final-a3.2(i)}, it follows that,  for  any $\bar{\eta} \in (0,\eta_o)$,  \eqref{conditionii} holds   for all  $(t,c)\in \{(\tau,x)\in \text{Gr} \,\left(C(\cdot) \cap\bar{B}_{\bar\varepsilon}(\bar{x}(\cdot)\right): \bar\I^0_{(\tau,x)}\not=\emptyset\}$ such that $(r+1)\notin \bar\I^0_{(t,c)}$. It remains to prove that \eqref{conditionii}  is valid  for all  $(t,c)$ such that  $(r+1)$  is necessarily in $\bar\I^0_{(t,c)}$, that is, when $\bar\I^0_{(t,c)}= \I^0_{(t,c)}\cup \{r+1\}$ and   $\l_{r+1}\not= 0$. Arguing  by contradiction, then there exist sequences $t_n\in [0,T]$,  $c_n\in C(t_n)$ with $\|c_n-\bar{x}(t_n)\|=\bar{\varepsilon}$, and $(\l_i^n)_{i\in \bar\I^0_{(t_n,c_n)}}$ with 
		$\l_i^n\ge 0$, for all $i\in\I^0_{(t_n,c_n)}$,  $\l_{r+1}^n >0$, and
		\begin{equation}\label{lambda}
		(\sum_{i\in\I^0_{(t_n,c_n)}}\l_i^n) +\l_{r+1}^n=1,
		\end{equation}
		such that  $\left\|\sum_{i\in\I^0_{(t_n,c_n)}}\lambda_i^n\nabla_x\psi_i(t_n,c_n) +\l_{r+1}^n (c_n-\bar{x}(t_n))\right\|\le \frac{2}{n},\quad \forall n\in\mathbb{N}$. Using the compactness of $[0,T]$, (A3.1),  and the continuity of $\bar{x}$, it follows that up to subsequences, $t_n\rightarrow t_o\in [0,T]$ and $c_n\rightarrow c_o\in C(t_o)$ with $\|c_o-\bar{x}(t_o)\|=\bar{\varepsilon}$. Note that $\I^0_{(t_n,c_n)} \not=\emptyset$, since otherwise, \eqref{lambda} yields $\l_{r+1}^n=1$, and in this case the above inequality becomes  $\|c_n-\bar{x}(t_n)\|\le \frac{2}{n}$, which is  invalid  for $n$ large.  Thus, by Lemma \ref{lemmaaux1},  for some $ \emptyset\not=\J_o\subset \{1,\dots,r\}$,   $\I^{0}_{(t_n,c_n)}= \J_o\subset \I^{0}_{(t_o,c_o)}$, for $n$ large. This implies that, for  $n$ large enough,
		\begin{eqnarray}\label{v2}
			\left\|\sum_{i\in\J_o}\lambda_i^n\nabla_x\psi_i(t_n,c_n) +\l_{r+1}^n(c_n-\bar{x}(t_n))\right\|\le \frac{2}{n}, \\
			\sum_{i\in\J_o}\l_i^n +\l_{r+1}^n=1, \;\;  \l_{r+1}^n > 0, \; \;\;  \text{and}\; \l_i^n\ge 0 \;\; \forall i\in\J_o. \nonumber \end{eqnarray}
		Hence, up to a subsequence,   $\l_i^n\rightarrow \l_i^o\ge0$ for all $i\in \J_o$, and $\l_{r+1}^n\rightarrow \l_{r+1}^o \ge 0$.  
		Upon taking the limit as $n\rightarrow\infty$ in \eqref{v2},   (A3.1) yields that    
		\begin{equation}\label{v3}
			\sum_{i\in\J_o}\lambda_i^o\nabla_x\psi_i(t_o,c_o)+ \l_{r+1}^o(c_o-\bar{x}(t_o))=0,\quad \sum_{i\in\J_o}\l_i^o +\l_{r+1}^o=1, \;\;\l_i^o\ge 0 \;\; \forall i\in\J_o\cup  \{r+1\}.
		\end{equation}
		From \eqref{v3} and  Lemma \ref{final-a3.2(i)} we get that $\l_{r+1}^o >0$. As $\|c_o-\bar{x}(t_o)\|=\bar{\varepsilon}$,   \eqref{v3} is translated to saying
		$$0\not= v:= \sum_{i\in\J_o}\lambda_i^o\nabla_x\psi_i(t_o,c_o)=-\l_{r+1}^o(c_o-\bar{x}(t_o)),$$
		and hence, per \eqref{normalcone-C}, $0\not= v \in N^P_{C(t_o)}(c_o) \cap -N^P_{\bar{B}_{\bar\varepsilon}(\bar{x}(t_o))}(c_o)$. As $\bar{\varepsilon}\in (0,\rho)$, then,  this  inclusion contradicts \cite[Equation (14)]{nourzeidan}, where we take  $S:=C(t_o)$. Note that  this equation can be easily verified to hold without the compactness  of the $\rho-$prox-regular  set $``S"$, considered therein.
	\end{proof} 
The  following result,  which shall be used in subsection \ref{proofpmp}, is  an immediate consequence of  Lemma \ref{final-a3.2(ii)} obtained via a simple argument by contradiction  and  the continuity in (A3.1) of $(\psi_i)_{1\le i\le r}$  and $(\nabla_x \psi_i)_{1\le i\le r}$ on the compact set   {Gr}\,$\left(C(\cdot)\cap \bar{B}_{\bar{\varepsilon}}(\bar{x}(\cdot)\right)$.
	\begin{remark}\label{pisiepse}
	Let $C(\cdot)$ satisfying $(A2)$ for some $\rho>0$. Consider  $\bar{x}(\cdot)\in \mathcal{C}([0,T];\mathbb{R}^n)$ with $\bar{x}(t)\in C(t)$ for all $ t \in [0,T]$, and $\bar{\delta}>0$  such that $(A3.1)$ and $(A3.2)$ hold at $(\bar{x};\bar{\delta})$.  Then,   for  $\bar\varepsilon\in (0,\rho)\cap (0,\varepsilon_o]$, and its corresponding $\psi_{r+1}$ and $\bar{\eta}$ from  Lemma \ref{final-a3.2(ii)}, there exists  $a_o>0$ such that for all $i\in\{1,\dots,r+1\}$ we have 
	\begin{equation} \label{pisieps} \big[(t,x) \in \text{Gr}\,\left(C(\cdot)\cap \bar{B}_{\bar{\varepsilon}}(\bar{x}(\cdot)\right)\;\hbox{and}\; \|\nabla_x\psi_i(t,x)\|\leq \bar\eta\big]\implies \psi_i(t,x)<-a_o.\end{equation}
\end{remark}
{Important consequences of Lemma \ref{final-a3.2(ii)} are the following explicit formulae for the normal cone to the truncated sets  $C(t) \cap \bar{B}_{\bar{\varepsilon}}(\bar{x}(t))$ and  for their prox-regularity constant, which shall replace $\rho_{\bar{\varepsilon}}$. Assume WLOG that $L_{\psi}\ge \frac{4\bar{\eta}}{\rho_o}$,  where $\rho_o$ is the constant from (A3.1).  
\begin{remark}\label{C-prop} Under the assumption of Remark \ref{pisiepse}, let $\bar{\varepsilon}\in (0,\rho)\cap (0,\varepsilon_o]$ with its corresponding $\psi_{r+1}$, given by \eqref{psi-r+1}, and $\bar{\eta}$ from Lemma \ref{final-a3.2(ii)}.
Let  {$\rho_{\bar{\varepsilon}}$} be the uniform prox-regular constant  of $C(t) \cap \bar{B}_{\bar\varepsilon}(\bar{x}(t))$ obtained from Remark \ref {prox-tilde{C}}.  Using Lemma \ref{final-a3.2(ii)},  \cite[Corollary 4.15]{prox}, and \cite[Corollary 10.44]{clarkebook}, it follows that, for all  $(t,x)\in \text{Gr}\,\left(C(\cdot) \cap\bar{B}_{\bar\varepsilon}(\bar{x}(\cdot)\right)$, 
\\
	$$ N_{C(t) \cap \bar{B}_{\bar\varepsilon}(\bar{x}(t))}(x)= N_{C(t) \cap \bar{B}_{\bar\varepsilon}(\bar{x}(t))}^P(x)= N_{C(t) \cap \bar{B}_{\bar\varepsilon}(\bar{x}(t))}^L(x),$$
	and 			
	\begin{equation} \label{normalcone}
	\hspace{-.2 in}	N_{C(t) \cap \bar{B}_{\bar\varepsilon}(\bar{x}(t))}(x) =
		\begin{cases} \displaystyle\{\sum_{i\in\mathcal{\bar{I}}^0_{(t,x)}}\lambda_i\nabla_x \psi_i(t,x): \lambda_i \ge 0\} \neq\{0\}&\textnormal{ if } x \in\textnormal{bdry}(C(t) \cap \bar{B}_{\bar\varepsilon}(\bar{x}(t))) \\
			\{0 \}&\textnormal{ if } x \in \textnormal{int}(C(t) \cap \bar{B}_{\bar\varepsilon}(\bar{x}(t))).
		\end{cases}
	\end{equation}
Furthermore,  $C(t) \cap \bar{B}_{\bar\varepsilon}(\bar{x}(t))$ is uniformly $\frac{2\bar{\eta}}{L_{\psi}}$-prox-regular and  epi-Lipschitz at every $x \in C(t) \cap \bar{B}_{\bar{\varepsilon}}(\bar{x}(t)) $,   with \begin{equation}\label{clint}
	\text{ cl} \left(\inte \left (C(t) \cap \bar{B}_{\bar{\varepsilon}}(\bar{x}(t))  \right) \right) = C(t) \cap \bar{B}_{\bar{\varepsilon}}(\bar{x}(t)). 
\end{equation} 
Indeed, for the prox-regularity constant,  given that $\nabla_x\psi_{r+1}(t,x)= x-\x(t)$  and condition $(A3.1)$ is satisfied, then  \cite[Theorem  9.1]{adly}$(i)$ holds for $m:=r+1$,  $g_i:=\psi_i$, and   $\rho:=\frac{\rho_o}{2}$, and the inequality in \cite[Theorem 9.1]{adly}$(ii)$ is satisfied by $\psi_1,\cdots, \psi_{r+1}$   for $\gamma:=L_{\psi}$. The last condition of \cite[Theorem  9.1]{adly} is obtained by translating the second equation of \cite[Lemma 6.1]{nourzeidan3} to our setting.  Hence,  the conclusion follows from  the proof of that theorem, and by using  Lemma \ref{final-a3.2(ii)} and $L_{\psi}\ge \frac{4\bar{\eta}}{\rho_o}$. To prove that $C(t) \cap \bar{B}_{\bar{\varepsilon}}(\bar{x}(t)) $ is epi-Lipschitzian for every $x \in C(t) \cap \bar{B}_{\bar{\varepsilon}}(\bar{x}(t)) $ and that \eqref{clint} is satisfied, we use \cite[Exercise 6.5 and  Corollary 6.10]{clsw}, and equations \eqref{conditionii}-\eqref{normalcone}.\\
\end{remark}	
Parallel to  Lemma \ref{lipschitz-local},  and based on Lemma \ref{final-a3.2(ii)} and Remark \ref{C-prop}, we shall  obtain here the Lipschitz continuity and the uniqueness of the solutions of the Cauchy problem corresponding to the \textit{truncated} system $(\bar{D})$, defined in \eqref{Dtilda}. We note that the existence of a solution for this more general Cauchy problem is obtained in Corollary \ref{corollary2}. 

\begin{lemma}\label{lipschitz-local-2} Consider $C(\cdot)$ satisfying (A2) for $\rho>0$. Consider  $(\bar{x}(\cdot),\bar{y}(\cdot))\in \mathcal{C}([0,T];\mathbb{R}^n \times \mathbb{R}^l)$ with 
$\bar{x}(t)\in C(t)$ for all $ t \in [0,T]$ and  $(\bar{x}(\cdot),\bar{y}(\cdot))$ is $L_{(\bar{x}, \bar{y})}$-Lipschitz on $[0,T]$ for some $L_{(\bar{x}, \bar{y})}\ge 1$. Let  $\bar{\delta}>0$  such that $(A3.1)$ and $(A3.2)$ hold at $(\bar{x};\bar{\delta})$ and $(A4.1)$ is satisfied by $(f,g)$ at $((\bar{x},\bar{y});\bar{\delta})$, and let $\bar{\varepsilon}\in (0,\rho)\cap (0,\varepsilon_o]$ with its corresponding $\psi_{r+1}$ given by \eqref{psi-r+1}.  Fix $u\in \mathcal{U}$ and $(x_0,y_0) \in \mathscr{N}_{(\bar\varepsilon, \bar{\delta})}(0)$.  Then, a pair  $(x(\cdot),y(\cdot))\in W^{1,1}([0,T];\R^{n+l})$ such that $(x(t),y(t)) \in  \mathscr{N}_{(\bar\varepsilon, \bar{\delta})}(t) \;\forall t \in [0,T]$  solves the system $(\bar{D})$  associated with $((x_0,y_0),u)$ 
if and only if  there exist measurable functions $(\lambda_1, \cdots, \lambda_{r+1})$  and  $\zeta$ such that, $\forall i=1,\cdots,r+1$,    $\lambda_i(t)= 0$   $ \forall t\in I_i^{\-}(x)$,   $\zeta(t) \varphi(t,y(t))=0$  $\forall t \in [0,T]$,  and $((x,y),u)$, $(\lambda_i)_{i=1}^{r+1}$, and $\zeta$ satisfy
\begin{eqnarray}	\begin{cases}					
		 &\dot{x}(t)=f(t,x(t),y(t),u(t))-\sum_{i\in \bar{\I}^0_{(t,{ x}(t))}} \lambda_i(t) \nabla_x\psi_i(t,x(t))\;\;\textnormal{a.e.}\; t\in[0,T], \\
		&\dot{y}(t) = g(t, x(t), y(t), u(t))- \zeta(t)\nabla_y\varphi(t,y(t)), \text{ a.e. } t \in [0,T],\\
&(x(0),y(0))=(x_0,y_0). \end{cases} \label{lipschitz-xy1}
\end{eqnarray}	
 Furthermore, we have the following bounds \begin{eqnarray}
	 &\max\{\|\sum_{i=1}^{r+1} \lambda_i\|_{\infty}, \; \|\zeta\|_{\infty}\} \le \frac{\bar{\mu}}{4 \bar{\eta}^2}, \quad 
	\max\{ \|\dot{x}\|_{\infty}, \|\dot{y}\|_{\infty}\} \le M_h + \frac{\bar{\mu}}{4 \bar{\eta}^2} \bar{L},\label{lipschitz-xy}\\
	& \hspace{-.7 in} \max\{\|\dot{x}(t)-f(t,x(t),y(t),u(t))\|, \|\dot{y}(t)-g(t,x(t),y(t),u(t))\|\} \le \frac{\bar{\mu}}{4 \bar{\eta}^2}\bar{L}, \;\; t\in [0,T] \,\text{a.e.},\label{dotx-f}
	\end{eqnarray} 
where
\begin{eqnarray}
	&\bar{L}:= \max \{ L_{\psi}, \bar{\delta} L_{(\bar{x},\bar{y})}\}  \ge \bar{\delta}, \quad  
	\bar{\mu}:=\bar{L} (1+M_h) \ge \mu.\label{tilde-mu}\end{eqnarray}
	Consequently, the pair $(x,y)$ is the unique  solution of \eqref{Dtilda} corresponding to $((x_0,y_0),u)$. 
\end{lemma}	
\begin{proof}
The equivalence follows from Filippov Selection theorem, the normal cone formula in \eqref{normalcone},  and the fact that    $N_{\bar{B}_{\bar{\delta}}(\bar{y}(t))}(y)$ equals $\{0\}$ if $\varphi(t,y)<0$, and equals $\{\lambda(y-\bar{y}(t)):\lambda\ge 0\}$ if $\varphi(t, y)=0.$\\
For the bounds  pertaining  $\|\sum_{i=1}^{r+1}\lambda_i \|_{\infty}$ and $\|\dot{\x}\|_{\infty}$, we follow the same steps as in proof of Lemma \ref{lipschitz-local}, with the main difference here is that we add an extra constraint to $C(t)$, namely, $\psi_{r+1}(t,x)\le 0$. For this reason,  it suffices to show  that \eqref{dpsi-r} and \eqref{mu}, where ${L}_{\psi}$ is enlarged to $\bar{L}$, are also valid for $i=r+1$, and  
that the set $\mathscr{T}$ can be modified to take into account   the addition of $\psi_{r+1}$. Once these goals are achieved,  the proof follows from that of  Lemma \ref{lipschitz-local}, where   $\bar{\I}^0_{(t,x(t))}$, Lemma \ref{final-a3.2(ii)}, $\bar\eta$, and $\bar\mu$  are used  instead of $\I^0_{(t,x(t))}$,  Lemma \ref{final-a3.2(i)}, $\eta_0$, and $\mu$, respectively.   \\
\\
Note that, by \eqref{psi-r+1},   $\nabla_x\psi_{r+1}(t,z)=z-\bar{x}(t)$ exists for all $(t,z)\in [0,T]\times \R^n$. Furthermore,  as $\x(\cdot)$ is $L_{(\x,\y)}$-Lipschitz, we have,  on $\text{Gr } C(\cdot)\cap\bar{B}_{\bar{\varepsilon}}(\bar{x}(\cdot))$, that  $\psi_{r+1}(\cdot,\cdot)$  is $\bar{\varepsilon} L_{(\bar{x},\bar{y})}$-Lipschitz,   $\nabla_x\psi_{r+1}(\cdot,\cdot)$ is  bounded by $\bar\varepsilon\le \bar{L}$, and  $\hat{\partial}_t\psi_{r+1}(\cdot,\cdot)$, defined via  \eqref{deltahat} for $i=r+1$, satisfies
\begin{equation}\label{deltahat-r+1}\hat{\partial}_t\psi_{r+1}(t,z)= \langle z-\x(t),-\partial\x(t)\rangle=\partial_t\psi_{r+1}(t,z), \;\;\forall (t,z)\in \text{Gr } C(\cdot)\cap\bar{B}_{\bar{\varepsilon}}(\bar{x}(\cdot)), \end{equation}
and hence,  $\forall \theta_{r+1} \in \hat{\partial}_t\psi_{r+1}(t,z), |\theta_{r+1}|\le \bar\varepsilon L_{(\x,\y)}\le \bar{L}$. Thus, for $t\in [0,T]$ a.e.,  and for all $  \theta_{r+1}(t)\in \hat{\partial}_t \psi_{r+1}(t,x(t))$,	we have
	\begin{eqnarray} \label{mutilde}
&&\hspace{-.2 in}| \theta_{r+1}(t) + \langle \nabla_x \psi_{r+1}(t,x(t)), {f}(t,x(t), y(t), u(t)) \rangle \le \bar{L} (1+\|{f}(t,x(t), y(t), u(t))\|). 		
		\end{eqnarray}
Therefore,  \eqref{mu} and \eqref{mutilde} yield  that  \eqref{mu} holds up to $i=r+1$, that is, for $t\in [0,T]$ a.e., for all $ \theta_i(t)\in \hat\partial\psi_i(t,x(t))$,	we have \begin{eqnarray} \label{mut}
	\hspace{-.9in}&& \hspace{-.2 in}|\theta_{i}(t) + \langle \nabla_x \psi_{i}(t,x(t)), {f}(t,x(t), y(t), u(t)) \rangle |\le \bar{L} (1+ \|{f}(t,x(t), y(t), u(t))\| ), \;  \forall i=1,\cdots,r+1. 
		\end{eqnarray}
On the other hand, from \eqref{psi-r+1}, \eqref{deltahat-r+1}, and the fact that $\dot{\x}(t)\in \partial\x(t)$  a.e.,  we have	
\begin{eqnarray}
\frac{d}{dt} \psi_{r+1}(t, x(t))&=& \langle x(t)-\bar{x}(t),  -\dot{\bar{x}}(t)+\dot{x}(t) \rangle, \; t\in[0,T]\;  \text{a.e.,}\color{red}\label{dpsi}\\
&=& \theta_{r+1}(t)+ \langle \nabla_x\psi_{r+1}(t,x(t)), \dot{x}(t) \rangle, \; t\in[0,T]\;  \text{a.e.,}\label{dpsi-theta}\end{eqnarray}
where $\theta_{r+1}(t)= \langle x(t)-\x(t), -\dot{\x}(t)\rangle \in \hat{\partial}_t\psi_{r+1}(t,x(t))$ a.e. \\
Therefore, \eqref{dpsi-r} holds up to $i=r+1$, that is, $\forall i$, there is measurable $\theta_i(\cdot) \in \hat\partial_t\psi_i(\cdot,x(\cdot))$ a.e.,  with 
\begin{equation}\label{dpsi-all}
\frac{d}{dt} \psi_{i}(t, x(t))= \theta_i(t)+	\langle \nabla_x\psi_{i}(t,x(t)), \dot{x}(t) \rangle, \;  \text{a.e.,}\; \; \forall i=1,\cdots, r+1.
\end{equation}
Instead of the set  $\mathscr{T}$ given in \eqref{scriptT} in the proof of Lemma  \ref{lipschitz-local}, we use  the following modified set $\bar{\mathscr{T}}$ 
 that involves $\dot{\bar{x}}$ and on which $\frac{d}{dt} \psi_{r+1}(\cdot, x(\cdot))$ readily exists, 
	$$\bar{\mathscr{T}}:= \{t\in (0,T):  \dot{x}(t),  \dot{\bar{x}}(t),  \; \text{and}\;\frac{d}{dt} \psi_i(t, x(t))  \; \text{exist}, \; \forall i=1,\cdots,r\}.$$ 
Therefore, similarly to \eqref{lambda-bound} we obtain
\begin{eqnarray}
			&& \sum_{i\in  \bar{\I}^0_{(t,x(t))}} \lambda_i(t) 
			\le \frac{ \bar{L}}{4\bar\eta^2}(1+  \|f(t,x(t),y(t),u(t))\|)\overset{(A4.1)} {\le} \frac{\bar\mu}{4 \bar\eta^2},   \label{lambda+1-bound}		\end{eqnarray} 
implying, via \eqref{lipschitz-xy1}(i) and $(A4.1)$, the required bound  in \eqref{lipschitz-xy} for $\|\dot{x}\|_{\infty}$ and the first bound in \eqref{dotx-f}.  	
	\\
For the bounds of $\zeta$ and $\dot{y}$ in \eqref{lipschitz-xy}, we use the full- measure set  $\bar{\mathscr{A}}:=\{t\in(0,T): \dot{\bar{y}}(t)\;\text{and}\;\dot{{y}}(t)\;\text{exist}\}$. If $t\in \bar{\mathscr{A}}$ and $\varphi(t,y(t))< 0$, then $\zeta(t)=0$ and the bound  on $\dot{y}$ follows using $(A4.1)$. If $t\in \bar{\mathscr{A}}$ and $\varphi(t,y(t))=0$, then $\|y(t)-\bar{y}(t)\|=\bar{\delta}$ and, since $\varphi(\cdot,y(\cdot))\le 0$, $\frac{d}{dt}\varphi(t,y(t))=0$. Hence, as \eqref{phi}  implies that 
\begin{equation}\label{dphi}\frac{d}{dt} \varphi(t, y(t))= \langle y(t)-\bar{y}(t),  -\dot{\bar{y}}(t)+\dot{y}(t) \rangle, \; t\in[0,T]\;  \text{a.e.,}\end{equation}
 then,  using  $\bar{\eta}<\frac{\bar{\varepsilon}}{2}<\frac{\bar{\delta}}{2}$ (by Lemma \ref{final-a3.2(ii)}),  \eqref{lipschitz-xy1}$(ii)$,    $L_{(\bar{x},\bar{y})}\ge 1$, and $\bar\delta L_{(\x,\y)}\le \bar{L}$ (by  \eqref{tilde-mu}), we get that for $t\in [0,T]$ a.e., 
\begin{equation}\label{zeta-bound}4\bar{\eta}^2 \zeta(t)\le { \bar{\delta}}^2\zeta(t)= \langle y(t)-\bar{y}(t) ,g(t,x(t), y(t), u(t))-\dot{\bar{y}}(t)\rangle \le \bar{L} (1+ \|g(t,x(t), y(t), u(t))\|).\end{equation}  
Therefore, by $(A4.1)$  we have,  $\|\zeta\|_{\infty}\le \frac{\bar{\mu}}{4\bar{\eta}^2}$, which  when combined with the second equation of \eqref{lipschitz-xy1},  yields the bound on $\|\dot{{y}}\|_{\infty}$ in \eqref{lipschitz-xy} and the second bound in \eqref{dotx-f}.	\\
The uniqueness proof of $(x,y)$  is similar to that in Lemma \ref{lipschitz-local}, where system $(D)$ is replaced  by  $(\bar{D})$,  the $\rho$-prox-regularily of $C(t)$ is replaced by the $\frac{2\bar{\eta}}{L_{\psi}}$-prox-regularity of $C(t)\cap \bar{B}_{\bar{\varepsilon}}(\x(t))$ obtained in Remark \ref{C-prop}, and  \eqref{lipschitz-D}-\eqref{bound}, $\mu$,  $\eta_o$, and $L_{\psi}$,  are  replaced by \eqref{lipschitz-xy1}-\eqref{lipschitz-xy}, $\bar\mu$,  $\bar\eta$, and $\bar{L}$,  respectively.  The $\infty$-prox-regularity of $\bar{B}_{\bar\delta}(\y(t))$  keeps the inequality in \eqref{gron}  valid.
\end{proof}

\subsection{Exponential penalty approximation  for the system $(\bar{D})$ }\label{approx} 
This subsection aims to establish the relationship between $(\bar{D})$ and  its approximating  standard control system} $(\bar{D}_{\gk})$, as well as the existence and uniqueness of Lipschitz solutions to the Cauchy problem associated with $(\bar{D})$. Throughout this whole subsection, we assume 
$C(\cdot)$ satisfying (A2) for $\rho>0$, and   $(\bar{x}(\cdot),\bar{y}(\cdot))\in \mathcal{C}([0,T];\mathbb{R}^n \times \mathbb{R}^l)$ with 
	$\bar{x}(t)\in C(t)$ for all $t\in [0,T]$  and  $(\bar{x}(\cdot),\bar{y}(\cdot))$ is $L_{(\bar{x}, \bar{y})}$-Lipschitz on $[0,T]$ for some $L_{(\bar{x}, \bar{y})}\ge 1$. Let  $\bar{\delta}>0$ be such that $(A3.1)$ and $(A3.2)$ hold at $(\bar{x};\bar{\delta})$ and $(A4.1)$ is satisfied by $(f,g)$ at $((\bar{x},\bar{y});\bar{\delta})$. Fix  $0<\bar{\varepsilon} < \bar{\delta}$, its corresponding $\psi_{r+1}$ given by \eqref{psi-r+1}, and $\bar{\eta}\in (0,\frac{\bar{\varepsilon}}{2}),$ such that $\bar\varepsilon, \psi_{r+1}$, and $\bar\eta$ satisfy Lemma \ref{final-a3.2(ii)}.  Assuming that $L_{\psi}\ge \frac{4\bar{\eta}}{\rho_o}$, set $\bar{\rho}:=\frac{2\bar{\eta}}{L_\psi}$, the prox-regular constant for the sets $C(t)\cap\bar{B}_{\bar{\varepsilon}}(\bar{x}(t))$ from Remark \ref{C-prop}. \\
  \\
We start by extending the function $h(t,x,\cdot,u)$ {from $\bar{B}_{\bar{\delta}}(\y(t))$} to $\mathbb{R}^l$ so that this extension satisfies  for all $ y \in \mathbb{R}^l,$  $(A4.1),$ {and  also  (A4.2)  whenever it is satisfied by $h$.} This extension shall be later used in Theorem \ref{invariance}.
\begin{remark} \textnormal{\textbf{[Extension]}}\label{extension} For  $t\in[0,T]$ a.e.,  {$x\in \left[C(t) \cap \bar{B}_{\bar\delta}(\x(t))\right],$} and for $u\in U(t)$,  it is possible to extend the function $h(t,{x} ,\cdot,u):=(f,g)(t,{x},\cdot,u)$ so that,  whenever $h$  satisfies $(A4)$ $($including $(A4.2)$$)$, its extension also satisfies $(A4)$  for all $y\in  \mathbb{R}^l$. Indeed, the convexity for all $t\in [0,T]$ of  $\bar{B}_{\bar{\delta}}(\bar{y}(t)) $  yields that $\pi(t,\cdot):=\pi_{ \bar{B}_{{\bar\delta}}(\bar{y}(t))}(\cdot)$	is well-defined and $1$-Lipschitz on $ \mathbb{R}^l$.
		\\	 
 Define  for $t\in[0,T]$ a.e., and  $(x,y,u)\in \left[C(t) \cap \bar{B}_{\bar\delta}(\x(t))\right] \times \mathbb{R}^l \times U(t)$, $\bar{h}(t,x,y,u):= h(t,x,\pi(t,y), u).$	 
{Whenever} $h$ satisfies $(A4)$ at $((\bar{x},\bar{y}),\bar{\delta})$, arguments similar to those in  {\cite[Remark 4.1] {nourzeidan}} show that   $\bar{h}$ (whose name we  keep  as $h$) also satisfies $(A4)$, where $\bar{\mathscr{N}}_{(\bar{\delta},\bar{\delta})}(t)$, which is $\left[C(t) \cap \bar{B}_{\bar{\delta}}(\bar{x}(t)) \right]\times \bar{B}_{\bar{\delta}}(\bar{y}(t)),$ is now replaced   by  $\left[C(t) \cap \bar{B}_{\bar{\delta}}(\bar{x}(t)) \right] \times \mathbb{R}^l$.
\end{remark}
The following notations, which depend on $(\bar{x}; \bar{\varepsilon})$ and $(\bar{y};\bar{\delta})$,  will be used in the proofs of the results that follow as well as the proof of Theorem \ref{pmp-ps}. They are instrumental in constructing  a  dynamic  $(\bar{D}_{\gk})$ that approximates  $(\bar{D})$  and has rich properties. 
	\begin{itemize}	
		\item 	Let $\bar{L}$ and $\bar{\mu}$ be the constants given in  \eqref{tilde-mu}. Define a sequence $(\gk)_k$ such that,  for all  $k \in \mathbb{N}$,
	$\gk >\dfrac{2 \bar{\mu}}{\bar{\eta}^2}e \;(> \dfrac{e}{\bar{\delta}}) $  \text{ and } $\gk \to \infty\text{ as } k \longrightarrow \infty,$
		and the real sequences $(\bar{\alpha}_k)_k, (\bar{\sigma}_k)_k,$ and $(\bar{\rho}_k)_k$ by  \begin{equation} \label{sigma}
		\hspace{-.4 in}	\bar{\alpha}_k := \frac{1}{\gk} \ln\left(\frac{{\bar{\eta}}^2  \gk}{2 \bar{\mu}}\right);\;  \bar{\sigma}_k:=\frac{(r+1)\bar{L}}{2\bar{\eta}^2}\left(\frac{\ln(r+1)}{\gk}+\bar{\alpha}_k\right);\;\bar{\rho}_k:= \sqrt{\bar{\delta}^2-2\bar{\alpha}_{\gk}}. \end{equation}
		Our choice of $\gk$ with  the fact that $\bar{\mu} > \bar{\delta}>\bar{\eta}$ yield that  ${\bar{\delta}}^2>\dfrac{2\ln({\gk\bar{\delta}})}{\gk}>2\bar{\alpha}_{\gk},$ and  
		\begin{equation}\label{alpha-mu}
			\gk e^{-\gk\bar{\a}_{k}}=\frac{2\bar{\mu}}{\bar{\eta}^2},\;\;(\bar{\alpha}_k, \bar{\sigma}_k, \bar{\rho}_k) >0\;\;\forall \;k\in\mathbb{N},\;\; \bar{\alpha}_k\searrow 0, \;\bar{\sigma}_k\searrow 0\;\;\hbox{and}\;\;\bar{\rho}_k\nearrow \bar{\delta}. 
		\end{equation}
		\item For each $t \in [0,T]$ and $k\in\N$, we define the compact sets
		\begin{eqnarray}
				\label{tildeCgkdef} &&\hspace{-.4 in}\bar{C}^{\gk}(t):=\bigg\{x\in\R^n : \sum_{i=1}^{r+1}e^{\gk \psi_i(t,x)}\leq 1\bigg\} \subset  \inte C(t) \cap {B}_{\bar{\varepsilon}}(\bar{x}(t)), \\   
			 &&\hspace{-.4 in}\bar{C}^{\gk}(t,k):=\bigg\{x\in\R^n: \sum_{i=1}^{r+1}e^{\gk \psi_i(t,x)}\leq \frac{2\bar{\mu}}{\bar{\eta}^2\gk}=e^{-\gk\bar{\a}_{k}}\bigg\}
			\subset \inte \bar{C}^{\gk}(t). \label{tildeCgkkdef}			
			\end{eqnarray}		
				\item {For $u\in \mathcal{U}$}, the approximation dynamic $(\bar{D}_{\gk})$ of $(\bar{D})$ is defined by
		\begin{equation}\label{Dgk1} \hspace{-2 cm}({\bar{D}}_{\gk}) \begin{cases} \dot{x}(t)=f(t,x(t),y(t),u(t))-\sum\limits_{i=1}^{r+1} \gk e^{\gk\psi_i(t,x(t))} \nabla_x\psi_i(t,x(t)),\;\;\textnormal{a.e.}\; t\in[0,T],\\ 
\dot{y}(t)=g(t,x(t),y(t),u(t))- \gk e^{\gk\varphi(t,y(t))}\nabla_y \varphi(t,y(t)), \text{ a.e. } t \in [0,T]. \end{cases}
		\end{equation}	
		\end{itemize}
	Using Lemma \ref{final-a3.2(ii)}, a translation of \cite[equation (8)]{nourzeidan3},  and arguments parallel to   those used in the proofs of \cite[Propositions 4.4 \& 4.6]{nourzeidan3}  and \cite[Proposition {5.3}]{nourzeidan4},   it is not difficult to derive the following properties for our sets $\bar{C}^{\gk}(t)$ and $\bar{C}^{\gk}(t,k)$, { knowing  that the sets $C(t)\cap \bar{B}_{\bar{\varepsilon}}(\x(t))$  are $\frac{2\bar{\eta}}{L_{\psi}}$- prox-regular  and are epi-Lipschitzian.  Notice,  from  \eqref{tildeCgkdef} and \eqref{tildeCgkkdef}, that  these sets here are  time-dependent,  uniformly localized  near $\bar{x}(t)$, and   are defined not only via $\psi_1,\cdots,\psi_r$ but also via the extra  function $\psi_{r+1}$.}
\begin{proposition}\label{propcgk(k)} The following holds true. 		\begin{enumerate}[$(i)$]
			\item  There exist $k_1\in\N$ and $r_1\in (0, { \frac{{\rho}_o}{2}}]$, such that  $\forall k\geq k_1$,   $\forall (t, x)\in\{(t,x)\in [0,T]\times\R^n:\sum_{i=1}^{r+1} e^{\gamma_k \psi_i(t,x)}=1\}$, and  $\forall \,(\tau,z)\in B_{2r_1}(t,x)$, we have 
			\begin{equation}\label{psiuniform}
				\left\|\sum_{i=1}^{r+1} e^{\gamma_k \psi_i(\tau,z)} \nabla_x \psi_i(\tau,z)\right\|> 2\bar{\eta} \sum_{i=1}^{r+1} e^{\gamma_k \psi_i(\tau,z)}.
			\end{equation} 
			\item There exists $k_2\geq k_1$ and $\bar{\epsilon}_{o}>0$ such that for all $k\geq k_2 $ we have  
\begin{eqnarray} 
\hspace{-.8 in}\left[x\in \bar{C}^{\gk}(t) \;\; \&\right.\left. \|\sum_{i=1}^{r+1} e^{\gamma_k \psi_i(t,x)} \nabla_x \psi_i(t,x)\| \le \bar{\eta} \sum_{i=1}^{r+1} e^{\gamma_k \psi_i(t,x)} \right] \hspace{-.1in} \implies \hspace{-.1in}\left[\sum_{i=1}^{r+1} e^{\gamma_k \psi_i(t,x)}<e^{-\bar{\epsilon}_{o} \gamma_k}\right].\label{epsilonpsigk}
			\end{eqnarray}
			\item   For all $ t \in [0,T]$,  for all $ k$,  $\bar{C}^{\gk}(t) \subset \inte \left( C(t) \cap \bar{B}_{\bar{\varepsilon}}(\bar{x}(t)) \right)$ and  $\bar{C}^{\gk}(t,k)\subset \inte \bar{C}^{\gk}(t)$, and these sets  are uniformly compact. Moreover, there exists $ k_3 \in \mathbb{N}$ such that for $k\geq k_3$, these sets are the closure of their interiors, their boundaries  and interiors are non-empty,  and  the formulae for their respective boundaries and interiors are obtained from  their own definitions in \eqref{tildeCgkdef} and \eqref{tildeCgkkdef} by replacing the inequalities therein  by equalities and strict inequalities, respectively. Furthermore, 
				$\bar{C}^{\gk}(t)$  and $\bar{C}^{\gk}(t,k)$ are amenable, epi-Lipschitz, and are respectively {$\frac{\bar\eta}{ L_{\psi}}$-  and $\frac{\bar\eta}{2 L_{\psi}}$}-prox-regular.
			\item  For every $t \in [0,T]$,  $(\bar{C}^{\gk}(t))_k$ and $(\bar{C}^{\gk}(t,k))_k$ are  nondecreasing sequences  whose  Painlev\'e-Kuratowski limit is $C(t) \cap \bar{B}_{\bar{\varepsilon}}(\bar{x}(t))$ and satisfy
			\begin{eqnarray}\hspace{-.7 in}{\small \inte\left(C(t) \cap \bar{B}_{\bar{\varepsilon}}(\bar{x}(t)) \right)=\bigcup_{k\in \N} \inte \bar{C}^{\gk}(t)=\bigcup_{k\in \N} \bar{C}^{\gk}(t) =\bigcup_{k\in \N} \inte \bar{C}^{\gk}(t,k)=\bigcup_{k\in \N} \bar{C}^{\gk}(t,k). }\label{union.Cgk(k)}			\end{eqnarray}
			\item  For $c\in  \bdry \left (C(0) \cap \bar{B}_{\bar{\varepsilon}}(\bar{x}(0)) \right) $, there exist $k_c\geq k_3$,  ${r}_{\c}>0$, and a vector $d_c\not=0$ such that  
			\begin{equation*}\label{lastideahope}
			\left(\left[\left(C(0)\cap \bar{B}_{\bar{\varepsilon}}(\bar{x}(0))\right) \cap \bar{B}_{{r}_{c}}({c})\right]+\bar{\sigma}_k\frac{d_c}{\|d_c\|}\right)\subset \inte \bar{C}^{\gk}(0,k),\;\;\forall k\geq k_c. 
			\end{equation*}
	In particular,  for $k\geq k_c$ we have 
	\begin{equation}\label{boundarypoint(k)}
				\left({c}+\bar{\sigma}_k\frac{d_c}{\|d_c\|}\right)\in \inte \bar{C}^{\gk}(0,k).
				 \end{equation}				
		\end{enumerate}
	\end{proposition}

	\begin{remark} \label{intelemma(k)} We deduce, from Proposition \ref{propcgk(k)}, that for any $c\in C(0) \cap \bar{B}_{\bar{\varepsilon}}(\bar{x}(0))$, there exists a sequence $({c}_{\gk})_k$ such that, for $k$ large enough, ${c}_{\gk}\in \inte \bar{C}^{\gk}(0,k)\subset \inte \bar{C}^{\gk}(0)$, and ${c}_{\gk}\f c$. Indeed: \begin{enumerate}[$(i)$]
			\item For $c\in  \bdry \left (C(0) \cap \bar{B}_{\bar{\varepsilon}}(\bar{x}(0)) \right) $,  we choose $c_{\gk} := c+\bar{\sigma}_k\frac{d_c}{\|d_c\|}$ for all $k$.  For $k\geq k_c$, we have from \eqref{boundarypoint(k)} that $c_{\gk}\in \inte \bar{C}^{\gk}(0,k)$. Moreover, since $\bar{\sigma}_k\f 0$  we have $c_{\gk}\f c$.
			\item For $c\in \inte \left(C(0) \cap \bar{B}_{\bar{\varepsilon}}(\bar{x}(0)) \right)$, {Proposition \ref{propcgk(k)}$(iv)$} 
yields the existence of $\hat{k}_c \in \N$, such that $c\in \inte \bar{C}^{\gk}(0,k)$ for all $k\ge \hat{k}_c$.
			Hence, there exists  $ \hat{r}_{c}>0 $  satisfying  \begin{equation*} \label{approxxbar}c\in \bar{B}_{\hat{r}_{c}}({c})\subset  \inte \bar{C}^{\gk}(0,k),\;\;\forall k\ge \hat{k}_c.
			\end{equation*}
			In this case,  we  take   the sequence ${c}_{\gk}\equiv c\in\inte \bar{C}^{\gk}(0,k)$ that converges to $c$. 
	\end{enumerate} \end{remark}
On the other hand, for the ball $\bar{B}_{\bar{\delta}}(\bar{y}(0))$ generated by the single function ${\varphi(0,\cdot)}$ in \eqref{phi}-\eqref{ball-phi}, we have the following property.
\begin{remark}\label{movingto-rhok0}
There exists  $k_o\in\N$  such that 
\begin{equation}\label{interior-rhok}
	\bar{B}_{\bar{\delta}}(\bar{y}(0))\cap \bar{B}_{\frac{\bar{\delta}}{4}}({\bf d}) - \frac{2\bar{\alpha}_k}{\bar{\delta}} \mathscr{V}({\bf d})\in {B}_{\bar{\rho}_k}(\bar{y}(0)),\;\;  \forall k\ge k_o\;\;\text{and}\; \; \;\forall {\bf d} \in \bdry \bar{B}_{\bar{\delta}}(\bar{y}(0)),
\end{equation}
where $\mathscr{V}({\bf d}):=\frac{\nabla_y\varphi(0,{\bf d})}{\|\nabla_y\varphi(0,{\bf {\bf d}})\|}=\frac{{\bf d}-\bar{y}(0)}{\bar{\delta}}$. This  property follows by applying  Theorem  3.1$(iii)$ of \cite{nourzeidan}  to $C:= \bar{B}_{\bar{\delta}}(\bar{y}(0))$, $r_o:=\frac{\bar{\delta}}{4}$, and $\eta:=\frac{\bar{\delta}}{2},$ and by noting the triangle inequality with  
$\| \nabla_y\varphi(0,{\bf d})\|=\bar{\delta}$   gives 
$$ \|\nabla_y\varphi(0,z)\| >\frac{\bar{\delta}}{2} \;\;\text{and}\;\;  \;\langle \nabla_y\varphi(0,z), \mathscr{V}({\bf d})\rangle > \frac{\bar{\delta}}{2},\, \;\; \forall {\bf d}\in \bdry \bar{B}_{\bar{\delta}}(\bar{y}(0)) \; \;\text{and}\;\; \forall z\in B_{\frac{\bar{\delta}}{2}}({\bf d}).$$
\end{remark}
Parallel to Remark \ref{intelemma(k)} and using Remark \ref{movingto-rhok0}, we deduce the following.
	\begin{remark}\label{movingto-rhok} For any ${\bf{d}} \in  \bar{B}_{\bar{\delta}}(\bar{y}(0))$, there exists a sequence $({d}_{\gk})_k$ such that, for $k$ large enough, ${d}_{\gk}\in \inte \bar{B}_{\bar{\rho}_k}(\bar{y}(0)) $, and ${d}_{\gk}\f {\bf d}$. Indeed: \begin{enumerate}[$(i)$]
			\item As $\bar{\rho}_k\nearrow \bar{\delta}$, we deduce from \eqref{interior-rhok} that for any ${\bf d}\in {\bdry} \bar{B}_{\bar{\delta}}(\bar{y}(0))$, there exists a sequence $({d}_{\gk})_k$ such that,{ for $k$ large enough, ${d}_{\gk}\in  B_{\bar{\rho}_k}(\bar{y}(0))\subset B_{\bar{\delta}}(\bar{y}(0))$, }and ${d}_{\gk}\f {\bf d}$.		
 			\item 	 For ${\bf{d}} \in \inte \bar{B}_{\bar{\delta}}(\bar{y}(0))  $, there exists $\textbf{k}_{\textbf{d}} \in \N$,  such that $\textbf{d}\in \inte \bar{B}_{\bar{\rho}_k}(\bar{y}(0))$ for all $k\ge \textbf{k}_{\textbf{d}}$.
 			Hence, there exists  $ \textbf{r}_{\textbf{d}}>0 $  satisfying  \begin{equation*} \label{approxxbar}\textbf{d}\in \bar{B}_{\textbf{r}_{\textbf{d}}}({\textbf{d}})\subset  \inte \bar{B}_{\bar{\rho}_k}(\bar{y}(0)),\;\;\forall k\ge \textbf{k}_{\textbf{d}}.
 			\end{equation*}
 			In this case,  we  take   the sequence ${d}_{\gk}\equiv {\bf d}\in\inte \bar{B}_{\bar{\rho}_k}(\bar{y}(0))$ that converges to ${\bf{d}}$. \\
\end{enumerate}
	\end{remark}
The next theorem is fundamental for the paper, as it illustrates two key ideas. First, it highlights the invariance for  $(\bar{D}_{\gk})$ of $\bar{C}^{\gk}(\cdot,k)\times \bar{B}_{\bar{\rho}_k}(\bar{y}(\cdot))\subset \inte \bar{C}^{\gk}(\cdot)\times {B}_{\bar{\delta}}(\bar{y}(\cdot)). $  More precisely, for $k$ large, if   the initial condition is  in $\bar{C}^{\gk}(0,k)\times \bar{B}_{\bar{\rho}_k}(\bar{y}(0))$,   then   $(\bar{D}_{\gk})$  has a  unique solution which is uniformly Lipschitz and remains in $\bar{C}^{\gk}(t,k)\times \bar{B}_{\bar{\rho}_k}(\bar{y}(t)) \;  \forall t \in [0,T]$. This result extends that in \cite{nourzeidan,nourzeidan3} in two directions:  (i)  when the original problem has {\it coupled} sweeping processes, and (ii) when the sweeping set is {\it time-dependent} and {\it localized} near $(\bar{x}, \bar{y})$.  Second, it  shows that   the solution of $(\bar{D}_{\gk})$ uniformly approximates  that of $(\bar{D})$. 
			
\begin{theorem}\label{invariance} \label{theoremrelationship}
			Let $(c_{\gk}, d_{\gk})_k$ be  such that $(c_{\gk}, d_{\gk}) \in    {\bar{C}}^{\gk}(0,k) \times  \bar{B}_{\bar{\rho}_k}(\bar{y}(0))  $ for every $k$, and $(c_{\gk}, d_{\gk}) \longrightarrow (x_0,y_0) \in \bar{\mathscr{N}}_{(\bar{\varepsilon}, \bar{\delta})} (0).$ Let  $u_{\gk}$ be a given sequence in $\mathcal{U}$. The following results hold:\\
\textnormal{\bf(I).} \textnormal{\bf[Existence of solution to $(\bar{D}_{\gk})$ and Invariance]} For $k$ large enough, the Cauchy problem of the system  $(\bar{D}_{\gk})$  corresponding to $(x(0),y(0))=(c_{\gk}, d_{\gk})$, and $u=u_{\gk}$, has  a unique solution $(x_{\gk},y_{\gk}) \in {W^{1,\infty}  }([0,T];\mathbb{R}^{n}\times \R^l)$   such that 
\begin{eqnarray}
& & (x_{\gk}(t),y_{\gk}(t)) \in \bar{C}^{\gamma_k}(t,k) \times  \bar{B}_{\bar{\rho}_k}(\bar{y}(t)) \quad \forall t \in [0,T], \label{boundxy} \\
& & \max\{ \|\xi_{\gk} \|_{\infty},  \| \zeta_{\gk} \|_{\infty}\} \le \frac{2\bar{\mu}}{\bar{\eta}^2}, \quad \max\{\|\dot{x}_{\gk}\|_{\infty},  \|\dot{y}_{\gk}\|_{\infty}\}     \le M_h + \frac{2\bar{\mu}}{\bar{\eta}^2} \bar{L}, \label{boundxi} 
			\end{eqnarray}
	where  $\xi_{\gk}(\cdot)$ and  $\zeta_{\gk}(\cdot)$ are the  positive continuous functions on $[0,T]$ corresponding respectively to the solutions  $x_{\gk}$ and $y_{\gk}$ via the formulae 
			\begin{equation} \label{defxi} \xi_{\gk}(\cdot):= \sum_{i=1}^{r+1}\xi_{\gk}^i(\cdot); \;\;\xi_{\gk}^i(\cdot):=\gk e^{\gk\psi_i(\cdot, x_{\gk}(\cdot))}  \; (i=1,\dots,r+1);\;\text{and}\; \;\zeta_{\gk}(\cdot):=  \gk e^{\gk\varphi(\cdot, y_{\gk}(\cdot))}. 
			\end{equation}
\textnormal{\bf(II).} \textnormal{\bf[Solution of (${\bar{D}}_{\gk}$) converges to a unique solution of $(\bar{D})$]} \\There exist  $(x, y) \in W^{1,\infty}([0,T];\mathbb{R}^n\times \R^l)$   and $(\xi^1, \cdots, \xi^r, \xi^{r+1}, \zeta)	\in L^{\infty}([0,T]; \mathbb{R}_+^{r+2})$ 		
			such that a subsequence of  $((x_{\gk}, y_{\gk}), (\xi^1_{\gk}, \cdots, \xi^r_{\gk},\xi^{r+1}_{\gk}, \zeta_{\gk}))$  $($we do not relabel$)$ satisfies \begin{eqnarray}
				& &\hspace{-.45 in}	(x_{\gk}, y_{\gk}) \xrightarrow[]{unif} (x, y), \quad 
					(\dot{x}_{\gk}, \dot{y}_{\gk}) \xrightarrow[in \; L^{\infty}]{w*} (\dot{x}, \dot{y}), \;\;\;\xi^i_{\gk}  \xrightarrow[in \; L^{\infty}]{w*} \xi^i \;(\forall i), \;\zeta_{\gk}  \xrightarrow[in \; L^{\infty}]{w*} \zeta, \label{converging}
			\end{eqnarray} 
		 and   $\xi_{\gk}$ converges weakly* in $L^{\infty}([0,T];\R_+)$ to $\xi:= \sum_{i=1}^{r+1} \xi^i $.  Moreover,
		 \begin{eqnarray} 						
&&(x(t), y(t)) \in {\bar{\mathscr{N}}_{(\bar\varepsilon, \bar{\delta})}(t) := (C(t) \cap \bar{B}_{\bar{\varepsilon}}(\bar{x}(t))) \times \bar{B}_{\bar{\delta}}(\bar{y}(t))}\; \;\;\forall t \in [0,T],   \label{boundofxyl} \\
&& \max\{\|{\xi}\|_{\infty},  \|{\zeta}\|_{\infty}\}  \le\; \; \frac{2\bar{\mu}}{\bar{\eta}^2},\quad \max\{\|\dot{x}\|_{\infty},  \|\dot{y}\|_{\infty}\}  \le M_h + \frac{2\bar{\mu}}{\bar{\eta}^2}  \bar{L}, \label{boundofxil}\\
	 					&&\begin{cases}
	\xi^i(t)=0 \textnormal{ for all } t \in I_i^{\-}(x), \;\; i \in \{1, \cdots, r, r+1\}, \\
			\xi(t)=0 \textnormal{ for all } t \in \bar{I}^{\-}(x), \;\; \textnormal{and} \;\;	 \zeta(t)=0 \textnormal{ for all } t \text{ such that } \varphi(t,y(t))<0.\label{xi-support}	\end{cases}\end{eqnarray}
If  $u_{\gk}$ admits a subsequence that converges a.e. to some $u \in \mathcal{U},$ or  if $(A1)$ and  $(A4.2)$ hold, then there exists $ u \in \mathcal{U}$ such that $(x, y)$ is the  unique solution of $(\bar{D})$ corresponding to $((x_0,y_0), u)$, and, for almost all $t\in [0,T]$,   
\begin{eqnarray}						
&& \dot{x}(t)=f(t,x(t),y(t),u(t))-\sum\limits_{i=1}^{r+1} \xi^i(t) \nabla_x\psi_i(t,x(t)),\label{dual1}\\
&& \hspace{.3 in}=f(t,x(t),y(t),u(t))-\sum_{i\in \bar{\I}{^0}_{(t,x(t))}} \xi^i(t) \nabla_x\psi_i(t,x(t)),\label{dual2}\\
&&\dot{y}(t)=g(t,x(t),y(t),u(t))- \zeta(t)\nabla_y\varphi(t,y(t)).\label{dual3}
\end{eqnarray}						
		\end{theorem}
		\begin{proof} \underline{\bf Part (I).}
			 Recall that  in Remark \ref{extension}, for  $t\in[0,T]$ a.e.,  {$x\in \left[C(t) \cap \bar{B}_{\bar\varepsilon}(\x(t))\right]$}, and for $u\in U(t)$, we extended $h(t,x,\cdot,u):=(f,g)(t,x,\cdot,u)$ so that $(A4)$ holds true  for all $y \in \mathbb{R}^l$. Hence,  for fixed $k\in \N$  and  for $u:=u_{\gk}$,  the system $(\bar{D}_{\gk})$ is well defined on the set 		
 \begin{equation} \label{Dlabel}
\mathscr{D}:=\{ (t,x,y) \in [0,T] \times \mathbb{R}^n \times \mathbb{R}^l: x \in 
\inte\left( C(t) \cap \bar{B}_{\bar{\varepsilon}}(\bar{x}(t))\right), \;\;y \in {B}_{2 \bar{\delta}}(\bar{y}(t)) \}.
 \end{equation}
 As $(0,c_{\gk}, d_{\gk})\in \mathscr{D}$,   standard local existence and uniqueness results from ordinary differential equations  (e.g., \cite[Theorem {{5.3}}]{hale}) confirm that 
 for some $T_1 \in (0,T]$, the Cauchy problem $(\bar{D}_{\gk})$ with $(x(0),y(0))=(c_{\gk}, d_{\gk})$  has a unique solution $(x_{\gk}, y_{\gk}) \in W^{1,1}([0,T_1];\mathbb{R}^n \times \mathbb{R}^l )$ such that $ (s,x_{\gk}(s), y_{\gk}(s)) \in  \mathscr{D} $ for all $s \in [0,T_1]$.	Set 
 \begin{eqnarray} \hat{T} := \sup \{T_1 : &(x, y) \text{ solves} \; ({\bar{D}}_{\gk}) \; \text{on} \;[0,T_1] \; \text{with}\; (x(0),y(0))=(c_{\gk}, d_{\gk}) \nonumber\\&\text{ and } (t,x(t),y(t)) \in \mathscr{D} \; \; \forall t \in [0,T_1] \}. \label{Tbar} \end{eqnarray}
The uniqueness  of the solution  yields that a solution $(x_{\gk},y_{\gk})$ of $({\bar{D}}_{\gk})$ with $(x(0),y(0))=(c_{\gk}, d_{\gk})$ exists on the interval $[0,\hat{T})$, and we have $(t,x_{\gk}(t),y_{\gk}(t)) \in \mathscr{D}$, $\forall t\in [0,\hat{T})$. \\
\\
\textbf{Step I.1. On ${[0,\hat{T}]},$ $(x_{\gk}(t),y_{\gk}(t)) \in\bar{C}^{\gk}(t) \times {\bar{B}_{\bar\delta}(\bar{y}(t))},$ and  $\hat{T}=T$. }  \\

Notice that  $x_{\gk}(0) = c_{\gk}\in \bar{C}^{\gk}(0,k)  \subset \inte \bar{C}^{\gk}(0)$ implies  that the function $\Delta(\cdot)$ given by $\Delta(\tau):= \sum_{i=1}^{r+1}e^{\gk \psi_i(\tau,x_{\gk}(\tau))}-1$ has $\Delta(0)<0$. If for some  $t_1 \in (0, \hat{T})$,   $ \Delta(t_1)=0$,  let $t>t_1$ close enough to $t_1$ so that $t \in (0, \hat{T})$.  Then, from   \eqref{dpsi-all} and  \eqref{mut},   we deduce that  for $i=1,\cdots, r+1$, there exists $\theta^i_{\gk}(\cdot) \in \hat{\partial}_s \psi_i(\cdot,x_{\gk}(\cdot))$ a.e., such that 
\begin{eqnarray} 
&&\frac{d}{ds} \psi_i(s,x_{\gk}(s)) = \theta^i_{\gk}(s) + \langle \nabla_x \psi_i(s,x_{\gk}(s)), \dot{x}_{\gk}(s) \rangle, \;\; s \in [t_1, t] \; \text{ a.e., }  \label{deripsi2}\\		&&	\hspace{-.6 in}	 \abs{\theta^i_{\gk}(s) + \langle \nabla_x \psi_i(s,x_{\gk}(s)), {f}(s,x_{\gk}(s), y_{\gk}(s), u_{\gk}(s)) \rangle}  \le \bar{L} (1+M_h) ={\bar{\mu}},\;\; s \in [t_1, t]  \text{ a.e. }\;\;    \label{thetai}
	\end{eqnarray} 
Then, using the first equation of $(\bar{D}_{\gk})$, we obtain 
				\begin{eqnarray}
					\Delta(t)-\Delta(t_1) &=&
					 \sum_{i=1}^{r+1} \int_{t_1}^{t}  \gk e^{\gamma_k \psi_i(s,x_{\gk}(s))} \frac{d}{ds} \psi_i(s,x_{\gk}(s)) ds\nonumber\\
					&\overset{\eqref{deripsi2}}{=}&  \int_{t_1}^{t}  \left( \sum_{i=1}^{r+1} \gk e^{\gamma_k \psi_i(s,x_{\gk}(s))} \left ( \theta^i_{\gk}(s) + \langle \nabla_x \psi_i(s,x_{\gk}(s)), {f}(s,x_{\gk}(s), y_{\gk}(s), u_{\gk}(s)) \rangle \right) \right. \nonumber \\
					& & \left. -  \langle \sum_{i=1}^{r+1} \gk e^{\gamma_k \psi_i(s,x_{\gk}(s))}\nabla_x \psi_i(s,x_{\gk}(s)), \sum_{j=1}^{r+1} \gk  e^{\gamma_k \psi_j(s,x_{\gk}(s))} \nabla_x \psi_j(s,x_{\gk}(s))   \rangle \right)  ds \nonumber\\
								&\overset{\eqref{thetai}}{\le}&  \int_{t_1}^{t}  \left( \sum_{i=1}^{r+1} \gk e^{\gamma_k \psi_i(s,x_{\gk}(s))}  {\bar{\mu}}   -  \left\| \sum_{i=1}^{r+1} \gk e^{\gamma_k \psi_i(s,x_{\gk}(s))}\nabla_x \psi_i(s,x_{\gk}(s))\right\|^2 \right)  ds\label{similarto}\\
										&\overset{\eqref{psiuniform}}{\le}&  \int_{t_1}^{t} \left(  \sum_{i=1}^{r+1} \gk e^{\gamma_k \psi_i(s,x_{\gk}(s))} \left({\bar{\mu}}  - 4 \bar{\eta}^2  \gk \sum_{i=1}^{r+1}   e^{\gamma_k \psi_i(s,x_{\gk}(s))}  \right) \right) ds \nonumber \\
					&\le&  \int_{t_1}^{t}  \sum_{i=1}^{r+1} \gk e^{\gamma_k \psi_i(s,x_{\gk}(s))} ( {\bar{\mu}}  - 2\bar{\eta}^2 \gk )  ds \;\; <0, \nonumber
				\end{eqnarray}
				the third and the second to last inequality are due to the fact that we can choose $t$ close enough to $t_1$ so that,  for $s\in [t_1,t]$, $x_{\gk}(s) \in B_{2r_1}(t_1,x_{\gk}(t_1))$ (needed to apply  \eqref{psiuniform}) and  $\sum_{j=1}^{r+1} \gk e^{\gamma_k \psi_j(s,x_{\gk}(s))} >\frac{1}{2}$, and the last inequality follows from $\gk > \frac{ {\bar{\mu}}}{2\bar{\eta}^2}$. This shows that, $\forall  t_1\in (0,\hat{T})$ with $\Delta(t_1)=0,$  $ \;\Delta(t)<0$ for all $ t>t_1$  close enough to $t_1$. Whence,  the continuity of $\Delta(\cdot)$  on $[0,\hat{T})$ and $\Delta(0) <0$ yield that $\Delta(t)\le 0$ for all $ t \in [0,\hat{T})$, that is,  $x_{\gk}(t) \in  \bar{C}^{\gk}(t) \subset \inte \left(C(t) \cap \bar{B}_{\bar{\varepsilon}}(\bar{x}(t))\right) \forall t \in [0,\hat{T})$.\\
				\\
		On the other hand, as $y_{\gk}(0)=d_{\gk}\in \bar{B}_{\bar{\rho}_k}(\bar{y}(0)) \subset B_{\bar{\delta}}(\bar{y}(0))$, we have $\varphi(0,{y}_{\gk}(0))<0$.  If for some  $t_1 \in (0, \hat{T})$,   $ \varphi(t_1,y_{\gk}(t_1))=0$, that is,  $\|y_{\gk}(t_1)-\bar{y}(t_1)\|=\bar\delta$, choose $t>t_1$ close enough  to $t_1$ so that,    $\forall  s\in[t_1,t]$,  $e^{\gamma_k \varphi(s,y_{\gk}(s))} >\frac{1}{2}$  and 
		$ \|y_{\gk}(s)-\bar{y}(s)\|>  \frac{\bar{\delta}}{2}$.  Hence,  \eqref{dphi}, $(\bar{D}_{\gk})$,  (A4.1),  \eqref{tilde-mu}, $y_{\gk}(\cdot)\in B_{2\bar{\delta}}(\y(\cdot))$, and   ${\bar{\eta}}< \frac{\bar{\delta}}{2}$ (by Lemma \ref{final-a3.2(ii)}), yield that, for $s \in [t_1,t]$ a.e.,
		 \begin{eqnarray} \label{dphi2}
		\hspace{-.8in}\frac{d}{ds} \varphi(s,y_{\gk}(s)) &=&\langle y_{\gk}(s)-\bar{y}(s) ,-\dot{\bar{y}}(s)+{g}(s,x_{\gk}(s), y_{\gk}(s), u_{\gk}(s))\rangle  \nonumber \\&&  - \gk  e^{\gamma_k \varphi(s,y_{\gk}(s))}  \| y_{\gk}(s)-\bar{y}(s)\|^2  \nonumber\\
&\le& \|y_{\gk}(s)-\bar{y}(s)\|\;L_{(\bar{x},\bar{y})}(1+M_h) - \gk  e^{\gamma_k \varphi(s,y_{\gk}(s))}  \| y_{\gk}(s)-\bar{y}(s)\|^2\label{eq}\\		
	&<& 2\bar{\mu}-\frac{\gk}{2} \bar{\eta}^2 <0,\nonumber
		\end{eqnarray}
		the last inequality follows from $\gk > \frac{ 2{\bar{\mu}}}{\bar{\eta}^2} e$.
		Hence, for all $t$ close enough to $t_1$, we have $$ 	\varphi(t,y_{\gk}(t)) =	\varphi(t,y_{\gk}(t)) 	-\varphi(t_1,y_{\gk}(t_1))  =  \int_{t_1}^{t}  \frac{d}{ds}\varphi(s,y_{\gk}(s)) ds <0.$$   
		This shows that $y_{\gk}(t) \in \bar{B}_{\bar{\delta}}(\bar{y}(t))$ for all $t\in [0,\hat{T})$. \\
							\\
				Since for  $t\in [0,\hat{T})$,   $(t,x_{\gk}(t), y_{\gk}(t))$ remains in the compact  set Gr$\left(\bar{C}^{\gk}(\cdot)\times \bar{B}_{\bar{\delta}}(\bar{y}(\cdot))\right)$ 
	then it is possible to extend in this compact set the solution $(x_{\gk},y_{\gk})$ to the whole interval $[0, \hat{T}]$. If $\hat{T}<T$,  the local existence of a solution starting  at $\hat{T}$ contradicts the definition of $\hat{T}$,  proving that 				 $\hat{T}=  T$. This completes Step {\bf I.1}.\\
	\\
	\textbf{Step I.2. Invariance of $\bar{C}^{\gk}(t,k) \times  \bar{B}_{\bar{\rho}_k}(\bar{y}(t))$, i.e.,  { \eqref{boundxy}  is valid.}}\\				As $c_{\gk} \in \bar{C}^{\gk}(0,k)$, we have $\sum_{i=1}^{r+1}  e^{\gamma_k \psi_i(0,c_{\gk})}\le \frac{2{\bar{\mu}}}{\bar{\eta}^2 \gk}$.  
				Since $\gamma_k \to \infty$, there exists $k_4 \in \mathbb{N}$ large enough such that for all $k \ge k_4$, we have that  
\begin{equation} \label{max}
	\frac{2{\bar{\mu}}}{\bar{\eta}^2 \gk} \ge \max \{ e^{-\frac{\gk\bar{\epsilon}_o}{2}}, \sum_{i=1}^{r+1}  e^{\gamma_k \psi_i(0,c_{\gk})}\},
				\end{equation}
				where $\bar{\epsilon}_o$ the constant from \eqref{epsilonpsigk}. Fix $k\ge k_4$. Let $t_1  \in [0,T)$ such that $\sum_{i=1}^{r+1}  e^{\gamma_k \psi_i(t_1,x_{\gk}(t_1))} =\frac{2\bar{\mu}}{\bar{\eta}^2 \gk}$. Let $\bar{\epsilon}_k = \min \{ \frac{\bar{\epsilon}_o}{2}, \frac{\ln {2}}{ 2\gk}\}$. Using the continuity of $x_{\gk}$ and $\psi_i(\cdot, \cdot)$, we can choose $t$ close enough to $t_1$ such that for all $s \in [t_1,t]$, 		
				\begin{eqnarray}
					\sum_{i=1}^{r+1} e^{\gamma_k \psi_i(s,x_{\gk}(s))}  \ge \sum_{i=1}^{r+1} e^{\gamma_k \psi_i(t_1,x_{\gk}(t_1))} e^{-\gk \bar{\epsilon}_k}{{=}} \frac{2\bar{\mu}}{\gk\bar{\eta}^2 } e^{-\gk \bar{\epsilon}_k} \label{psi-bound}\\
					\overset{\eqref{max}}{\ge} e^{-\frac{\gk\bar{\epsilon}_o}{2}} e^{-\gk \bar{\epsilon}_k} \ge e^{-\frac{\gk\bar{\epsilon}_o}{2}} e^{-\frac{\gk\bar{\epsilon}_o}{2}} =e^{-\gk\bar{\epsilon}_o}.\nonumber
				\end{eqnarray}  
				Hence, by Proposition $\ref{propcgk(k)}(ii)  $, and the fact that $x_{\gk}(\tau) \in \bar{C}^{\gk}(\tau)$  for all $\tau \in [0,T]$, we have 
				\begin{equation}\label{bound-grad}\|\sum_{i=1}^{r+1} e^{\gamma_k \psi_i(s,x_{\gk}(s))} \nabla_x \psi_i(s,x_{\gk}(s))\| > \bar{\eta} \sum_{i=1}^{r+1} e^{\gamma_k \psi_i(s,x_{\gk}(s))}, \quad\forall s\in[t_1,t].\end{equation} 
Thus, for $\bar{\Delta}(\cdot):= \sum_{i=1}^{r+1} e^{\gamma_k \psi_i(\cdot,x_{\gk}(\cdot))}-\frac{2\bar{\mu}}{\bar{\eta}^2 \gk}$, we have				\begin{eqnarray*}
					\bar{\Delta}(t)-\bar{\Delta}(t_1)&=&\sum_{i=1}^{r+1} e^{\gamma_k \psi_i(t,x_{\gk}(t))} 	-\sum_{i=1}^{r+1} e^{\gamma_k \psi_i(t_1,x_{\gk}(t_1))}\\ 
&\overset{(\refeq{similarto})}{\le}& \int_{t_1}^{t}  \left( \sum_{i=1}^{r+1} \gk e^{\gamma_k \psi_i(s,x_{\gk}(s))}  {\bar{\mu}}   -  \left\| \sum_{i=1}^{r+1} \gk e^{\gamma_k \psi_i(s,x_{\gk}(s))}\nabla_x \psi_i(s,x_{\gk}(s))\right\|^2 \right)  ds\nonumber\\
&\overset{(\refeq{bound-grad})}{\le}&\int_{t_1}^{t} \left(  \sum_{i=1}^{r+1} \gk e^{\gamma_k \psi_i(s,x_{\gk}(s))} \left({\bar{\mu}}  -  \bar{\eta}^2  \gk \sum_{i=1}^{r+1}   e^{\gamma_k \psi_i(s,x_{\gk}(s))}  \right) \right) ds\nonumber\\					
					&\overset{(\refeq{psi-bound})}{\le}&  \int_{t_1}^{t}  \sum_{i=1}^{r+1} \gk e^{\gamma_k \psi_i(s,x_{\gk}(s))} \bar{\mu}\left( 1  -  2     e^{-\gk \bar{\epsilon}_k}  \right)  ds <0,
					\end{eqnarray*}
				the last  inequality follows from  the definition of $\bar{\epsilon}_k$. This proves that $x_{\gk}(t) \in \bar{C}^{\gk}(t,k)$ for all $ t>t_1$ close enough to $t_1$. Whence, similarly to Step 1, the continuity of  
$\bar{\Delta}(\cdot)$ and $\bar{\Delta}(0)\le 0$ imply   that $x_{\gk}(t) \in\bar{C}^{\gk}(t,k), \; \forall t\in [0,T]$.  \\ 
\\
On the other hand, having $y_{\gk}(0)=d_{\gk}\in \bar{B}_{\bar{\rho}_k}(\bar{y}(0))$, where $\bar{\rho}_k$ is given in \eqref{sigma},  means  that $\varphi(0,y_{\gk}(0)) \le -\bar{\alpha}_k$. Since $\bar{\alpha}_k \to 0$, and $\bar{\alpha}_k>0$ for all $k$, then we can find  $k_5\ge k_4$ such that $$\bar{\alpha}_k \le \min \left\{ \frac{\bar{\delta}^2}{4} , - \varphi(0,d_{\gk})\right\} \text{ for all } k\ge k_5.$$
Define $ \hat{\epsilon}_k:= \min \{\frac{\bar{\delta}^2}{8}, \frac{\ln{2}}{{ 2}\gk} \}$.   
If for some  $t_1 \in [0, T)$,   $ \varphi(t_1,y_{\gk}(t_1))=-\bar{\alpha}_k$,  let $t>t_1$ close enough to $t_1$ such that 
\begin{equation} \varphi(s,y_{\gk}(s)) \ge -\bar{\alpha}_k  - \hat{\epsilon}_k, \quad \forall s \in [t_1,t]. \label{eq3} \end{equation}
Then, for all $s \in [t_1,t]$, we have \begin{equation} \|y_{\gk}(s)-\bar{y}(s)\|^2 \ge \bar{\delta}^2-2\bar{\alpha}_k - 2\hat{\epsilon}_k \ge \frac{\bar{\delta}^2}{4} \ge \bar{\eta}^2 .  \label{eq2} \end{equation} 
Hence, using, respectively, \eqref{eq}, $y_{\gk}(\cdot)\in \bar{B}_{\bar{\delta}}(\bar{y}(\cdot))$,  \eqref{tilde-mu}, \eqref{eq3}, \eqref{alpha-mu}$(i)$,  and \eqref{eq2}, we deduce that 
	\begin{eqnarray*}
	\varphi(t,y_{\gk}(t)) 	-\varphi(t_1,y_{\gk}(t_1))  &=&  \int_{t_1}^{t}  \frac{d}{ds}\varphi(s,y_{\gk}(s)) ds \\
	&{\le}& \int_{t_1}^{t} \left ( \bar{\mu}  
	- \gk  e^{-\gamma_k \bar{\alpha}_k} e^{-\gk \hat{\epsilon}_k}   \|y_{\gk}(s) -\bar{y}(s)\|^2 \right)  ds \nonumber\\ 
	&{\le}& \int_{t_1}^{t} \left ( \bar{\mu}  
	- \frac{2 \bar{\mu}}{\bar{\eta}^2} e^{-\gk \hat{\epsilon}_k}   \bar{\eta}^2 \right)  ds \nonumber\\ 
	&{=}& \int_{t_1}^{t} \bar{\mu} \left (  1
	- 2 e^{-\gk \hat{\epsilon}_k}    \right)  ds<0 \nonumber 	
\end{eqnarray*}
proving that $\varphi(t,y_{\gk}(t))< -\bar\alpha_{\gk}$. Thus, the continuity of $\varphi(\cdot,y_{\gk}(\cdot))$  yields,  $y_{\gk}(t) \in \bar{B}_{\bar{\rho}_k}(\bar{y}(t))$ $ \forall t\in [0,T]$.\\
\\
\textbf{Step I.3.	 $(x_{\gk},y_{\gk},\xi_{\gk}, \zeta_{\gk})$  satisfy equation {(\ref{boundxi}).}}\\	
So far, we proved that  a solution  $(x_{\gk}, y_{\gk})$ of the Cauchy problem of $(\bar{D}_{\gk})$ exists and  satisfies (\ref{boundxy}). Hence,     the definitions of $\bar{C}^{\gk}(t, k)$ and $\xi_{\gk}$ given in \eqref{tildeCgkkdef} and \eqref{defxi}, respectively,  yield that $\|\xi_{\gk}\|_{\infty}\le \frac{2\bar{\mu}}{\bar{\eta}^2}$. On the other hand, the definition of $\bar{B}_{\bar{\rho}_k}(\bar{y}(t))$ yields that $\varphi(t,y_{\gk}(t))\le -\bar{\alpha}_k,$ and thus,  the same bound is immediately obtained for the norm of   $\zeta_{\gk},$ defined  in \eqref{defxi}. Whence, the first inequality in  \eqref{boundxi} is satisfied.  
Employing  this latter  in $(\bar{D}_{\gk})$ and then calling on the  definition of $\bar{L}$ in \eqref{tilde-mu},  we obtain that the second inequality in  \eqref{boundxi} is valid. \\
  \\				
\underline{\bf Part (II).}
\textbf{Step II.1. Existence of $(\xi^1,\cdots, \xi^{r+1}, \zeta)$  and $(x,y)$ satisfying   \eqref{converging}-\eqref{xi-support}.}\\
Using \eqref{boundxy}-\eqref{boundxi}, it follows  that \eqref{boundxyxi-n} holds for $\mathcal{R}:=r+1$ and  $(x_k,y_k, \xi_k^i,\zeta_k):=(x_{\gk},y_{\gk}, \xi_{\gk}^i,\zeta_{\gk})$. 
 Hence, by Lemma \ref{convergence}$(i)$,  there is a subsequence (not relabeled)
of   $(x_{\gk}, y_{\gk})$, $  (\xi_{\gk}^1, \cdots, \xi_{\gk}^{r+1}, \zeta_{\gk})$, that  converges,  respectively,  to some $(x,y)\in W^{1,\infty}([0,T];\mathbb{R}^{n+l})$,    $( \xi^1, \cdots, \xi^{r+1},\zeta)\in L^{\infty}([0,T]; \mathbb{R}^{r+2}_{+}),$ such that  \eqref{converging} and \eqref{boundofxil} are satisfied. Moreover,  \eqref{boundofxyl} follows from \eqref{boundxy}, \eqref{alpha-mu}, and  Proposition \ref{propcgk(k)} $(iv)$.\\
\\
Now, we show that \eqref{xi-support} holds.  Let $ i\in \{1, \cdots, r+1\}$ and $ t \in I^{\-}_i(x),$ that is, $\psi_i(t,x(t)) <0$.  Then,  by $(A3.1)$ and the uniform convergence of $x_{\gk}$ to $x$, there exist  $ k_t \in \mathbb{N}$, $\alpha_t >0$, and  $\tau_t >0 $ such that   $\forall k \ge k_t$, we have $$\psi_i(s,x_{\gk}(s)) < -\frac{\alpha_t}{2}, \;\; \forall s \in (t-\tau_t, t+\tau_t )\cap [0,T].$$ 
Hence, $\xi_{\gk}^i(s) < \gk e^{-\gk \frac{\alpha_t}{2}} \xrightarrow[k \to \infty]{} 0,  \textnormal{ uniformly on } (t-\tau_t, t+\tau_t ) \cap [0,T]\text { and } \xi^i(t)=0 . $ Let $t \in \bar{I}^{\-}(x),$ then $ t \in I^{\-}_i(x)$  $ \forall i\in \{1, \cdots, r+1\}$, and hence, $\xi^i(t)=0$  $\forall i \in \{1, \cdots, r+1\}$,    implying that also $\xi(t)=0$. Similarly, let $t \in [0,T]$ such that $\varphi(t,y(t)) <0$. The same arguments  now applied to $\varphi(t,y(t))$ yield the existence of  $\hat{k}_t \in \mathbb{N}$, $\hat{\alpha}_t>0$ and $\hat{\tau}_t >0$ such that, $\forall s\in (t-\hat{\tau}_t, t+\hat{\tau}_t ) \cap [0,T],$
$$ \zeta_{\gk}(s):= \gamma_k e^{\gk \varphi(s,y_{\gk}(s))} < \gk e^{ \frac{-\gk\hat{\alpha}_t}{2}} \xrightarrow[k \to \infty]{\textnormal{uniformly}} 0, \;  \text { and hence,} \; \zeta(t)=0. $$
\textbf{Step II.2. Existence of $u\in \mathcal{U}$: $((x,y),u)$ $\&$ $(\xi^i,\zeta)$ satisfy $(\bar{D})  $ and  \eqref{dual1}-\eqref{dual3}, $(x,y)$ unique.}\\ Whether  $u_{\gk}$ admits a subsequence that converges   to some $u \in \mathcal{U}$, for $t\in[0,T]$ a.e., or  {assumptions (A1) and} (A4.2) are satisfied,   apply in each of the two cases the corresponding result in Lemma \ref{convergence} $(ii)$ to  $Q(\cdot):=C(\cdot) \cap \bar{B}_{\bar\varepsilon}(\bar{x}(\cdot))$, $\r:=r+1$,  $q_i(\cdot,\cdot):=\psi_i(\cdot,\cdot)$, and to the sequences  $((x_k,y_k), u_k):=  ((x_{\gk}, y_{\gk}),u_{\gk})$,  $\xi_k^i:= \xi^i_{\gk}$, and $\zeta_k:=\zeta_{\gk}$  in \eqref{defxi}, and  their respective limits  $(x, y)$, $\xi^i$ and $\zeta$. Then, there exists   $u(\cdot)$  such that $((x,y),u)$, $\xi^i$ $(i=1,\cdots,r+1)$ and $\zeta$ satisfy  (\ref{dual1})-\eqref{dual3}. 
 The facts that  $(x,y)$ is a solution of  $(\bar{D})$ corresponding to $((x_0,y_0),u)$ and is unique follow now directly from Lemma \ref{lipschitz-local-2}. This completes the proof of this Theorem.
 \end{proof}			
\begin{remark} \label{invariant} Similar arguments to steps \textbf{I.1-2} in the proof of Theorem \ref{theoremrelationship} also show the invariance of  the larger sets $\bar{C}^{\gk}(t)\times \bar{B}_{ \bar{\rho}_k}(\bar{y}(t));$  this means that if $(c_{\gk}, d_{\gk})$ is taken in $\bar{C}^{\gk}(0)\times \bar{B}_{\bar{\rho}_k}(\bar{y}(0))$, then for all $t \in [0,T],$ $(x_{\gk}(t),y_{\gk}(t)) \in \bar{C}^{\gk}(t) \times \bar{B}_{\bar{\rho}_k}(\bar{y}(t))$.\\
\end{remark}		
		
	The following corollary is an immediate consequence of Theorem \ref{theoremrelationship}, { in which we take  $u_{\gk}= u$ for all $k$, and hence, neither (A1) nor (A4.2)  is required.}  It also consists  of a {\it Lipschitz-existence and uniqueness} result for the Cauchy problem of  $(\bar{D})$ via  the solution of the  Cauchy problem of  $(\bar{D}_{\gk})$, whose initial condition is carefully chosen. 
\noindent\begin{corollary}\label{corollary2}  For given $(x_0,y_0) \in {\bar{\mathscr{N}}_{(\bar{\varepsilon}, \bar{\delta})} (0)}$ and $u \in \mathcal{U}$, the system  $(\bar{D})$ corresponding to $((x_0,y_0);u)$ has a unique solution $(x,y)$,  and hence it is Lipschitz and satisfies \eqref{lipschitz-xy1}-\eqref{dotx-f}.  This solution is the uniform limit of a subsequence $($not relabeled $)$ of $(x_{\gk},y_{\gk})_k,$ which is obtained via Theorem \ref{theoremrelationship} as the solution of $(\bar{D}_{\gk})$ with $((x(0),y(0));u)=((c_{\gk},d_{\gk});u)$, where ${c}_{\gk}$ and $d_{\gk}$ are the sequences from Remarks \ref{intelemma(k)} and \ref{movingto-rhok}  corresponding to $c=x_0$ and ${\bf{d}}=y_0$, respectively. Hence, for $k$ sufficiently large, we have that $(x_{\gk}(t),y_{\gk}(t)) \in \bar{C}^{\gk}(t,k) \times \bar{B}_{\bar{\rho}_k}(\bar{y}(t)) \; \forall t \in [0,T]$, $(x_{\gk},y_{\gk})_k$ is uniformly lipschitz, and all conclusions of Theorem $\ref{theoremrelationship}$ hold. 
		\end{corollary}

\section{{ Proof of }Global results}\label{global}
The goal of this section is to prove Theorems \ref{cauchyglobal} -- \ref{existenceglobal}, whose statements  require  the global assumptions, presented in Subsection \ref{globalmain},  which  are obtained from the local ones by considering  $\bar\delta=\infty$,  in addition to assuming the boundedness of Gr $C(\cdot)$. We first start by generalizing some of the results presented in Section \ref{auxiliary}. \\
\\
The compactness of  \text{Gr} $C(\cdot)$ assumed in the following lemma allows us to easily imitate  the proof of  Lemma \ref{final-a3.2(i)}  and  produce  the following equivalence 
between (A3.2)$_{G}$ and a global version of condition {\eqref{3.1(i)}}, namely, \eqref{eta-global},  in which $\bar{x}$ and the localization around it are absent.
\begin{lemma}\label{global-a3.2}
Assume that $\psi_i(\cdot,\cdot)$ is continuous and,  for all $t \in [0,T]$, the set $C(t)$ is nonempty, closed, and given by \eqref{C}. Assume that $({A3.1})_G$ holds and that \text{Gr} $C(\cdot)$ is compact. Then $({A3.2})_G$  is equivalent to {the existence of a constant $\eta>0$} such that 
	\begin{equation}\label{eta-global} \left\|\sum_{i\in\I^0_{(t,c)}}\lambda_i\nabla_x\psi_i(t,c)\right\|>2\eta, \;\;  \forall (t,c)\in 
		\{(\tau,x)\in \text{Gr} \; C(\cdot) : \I^0_{(\tau,x)}\not=\emptyset\},
	\end{equation}
	where $\I^0_{(\tau,x)}$ is defined in \eqref{Isub(t,x)} and $(\l_i)_{i\in\I^0_{(t,c)}}$ is any sequence of nonnegative numbers satisfying $\sum_{i\in\I^0_{(t,c)}}\l_i=1$.\end{lemma} 
	{As a consequence of Lemma \ref{global-a3.2}, we obtain the uniform prox-regularity of $C(t)$, as well as a formula for the normal cone to $C(t)$}. For  ${L}_{\psi}$ the common  bound of $\{\|\nabla_x \psi_i(\cdot,\cdot)\|\}_{i=1}^{r}$ and {the common Lipschitz constant of} $\{\nabla_x \psi_i(\cdot,\cdot)\}_{i=1}^r$ on  the compact set Gr $C(\cdot) + \{0\} \times \frac{\rho_o}{2} \bar{B}$, {we assume without loss of generality that ${L_{\psi}} \ge \frac{8\eta}{\rho_o}$. }
\begin{remark}\label{global-prox}	Under the assumptions of Lemma \ref{global-a3.2} and $(A3.2)$$_G$, condition \eqref{eta-global}  	implies via  \cite[Proposition 4.1$(i)$]{nourzeidan3} and \cite[Theorem 9.1]{adly} that  for all $t \in [0,T]$,  $C(t)$ is amenable (in the sense of \cite{Rockawets}), epi-lipschitzian,  $C(t)= \clo(\inte C(t))$, and is {\it uniformly} {$\frac{2\eta}{L_{\psi}}$-prox-regular}.  In this global setting, the normal cone  formula \eqref{normalcone-C} is now valid for all $(t,x) \in \text{Gr } C(\cdot).$ In particular,   
	\begin{equation} \label{normalconeg}
	 \hspace{-.8 in}	N_{C(t)}(x)= N_{C(t)}^P(x)= N_{C(t)}^L(x) = 
		\left\{\sum_{i\in\I^0_{(t,x)}}\lambda_i\nabla_x\psi_i(t,x): \lambda_i \ge 0\right\}  \neq\{0\},\quad {\forall x\in\bdry C(t)}. \end{equation}	
\end{remark}	
The following lemma, which requires $\Gr C(\cdot)$ bounded,  is  a global version of Lemma \ref{lipschitz-local}. 
\begin{lemma}\label{lipschitz-global} 	Let the assumptions of Lemma \ref{global-a3.2}, $(A3.2)$$_G$ and $(A4.1)$$_G$ be satisfied. Let $u\in \mathcal{U}$,  $(x_0,y_0)\in C(0)\times \mathbb{R}^l$ be fixed, and  $({x}(\cdot),{y}(\cdot))\in W^{1,1}([0,T];\mathbb{R}^n\times \mathbb{R}^l)$ with $(x(0),y(0))=(x_0,y_0)$ {and $x(t)\in C(t),\;\forall t\in[0,T]$}.  Then, $(x,y)$ solves $(D)$ {corresponding to $((x_0,y_0),u)$ } if and only if  there exist measurable  functions $(\lambda_1(\cdot), \cdots, \lambda_{r}(\cdot) )$ such that, for all $i=1,\cdots,r$,   $\lambda_i(t)= 0$  for $ t\in I_i^{\-}(x)$,  and {$((x,y),u)$} together with  $(\lambda_1,\cdots,\lambda_r)$ satisfies \eqref{lipschitz-D}.
	Furthermore, the bounds in \eqref{bound} are satisfied with $\eta_0$ being replaced by $\eta$, the constant from Lemma \refeq{global-a3.2}. {Furthermore,} $(x,y)$ is the unique solution of $(D)$ corresponding to $((x_0, y_0), u)$. 	
\end{lemma}
\begin{proof} 
	We follow the same proof of Lemma \ref{lipschitz-local}  using the normal cone formula in \eqref{normalconeg} instead of \eqref{normalcone-C}, Lemma \ref{global-a3.2} instead of Lemma \ref{final-a3.2(i)}, and the prox-regularity constant  {$\frac{2\eta}{L_{\psi}}$} provided by Remark \ref{global-prox}   instead of $\rho$. 
\end{proof}
We now introduce the following notations that are going to be used in our proofs. 
\begin{itemize}
		\item Recall  from \eqref{def-mu}  that $\mu:=L_{\psi}(1+M_h)$. Define the sequences $\gamma_k > \frac{2 \mu}{\eta^2}e$, $\alpha_k$, $\sigma_k$ as in \eqref{sigma}, where $\bar{\eta}$, $\bar{\mu}$, $\bar{L}$, and $r+1$, therein, are replaced  by $\eta$, $\mu$, $L_{\psi}$ and $r$, respectively. 
	\item  For each $t \in [0,T]$ and $k\in\N$, we define
	\begin{eqnarray}
		\label{Cgkdef} &&\hspace{-0.6 in} C^{\gk}(t):=\bigg\{x\in\R^n : \sum_{i=1}^{r}e^{\gk \psi_i(t,x)}\leq 1\bigg\} \subset \inte C(t) \; \text{for}\; r>1, \;\&\; C^{\gk}(t):=C(t)\,\;\text{for}\; r=1,\\
		\label{Cgkkdef}&& \hspace{-.6 in}C^{\gk}(t,k):=\bigg\{x\in\R^n : \sum_{i=1}^{r}e^{\gk \psi_i(t,x)}\leq \frac{2\mu}{{\eta}^2\gk}=e^{-\a_{k}\gk} \bigg\} \subset \inte C^{\gk}(t).
	\end{eqnarray}		
	\end{itemize}
\begin{remark}\label{global-results}
Under the assumptions of Remark \ref{global-prox}, proposition \ref{propcgk(k)}
and Remark \ref{intelemma(k)} hold true in a global setting,  that is,  the ball around $\bar{x}$ is now omitted in those  statements. More precisely, the summations therein are only up to $r$ instead of $r+1$ terms, with $C^{\gk}(t,k)$, $C^{\gk}(t)$, and $C(t)$ replacing $\bar{C}^{\gk}(t,k)$, $\bar{C}^{\gk}(t)$, and $C(t) \cap \bar{B}_{\bar{\varepsilon}}(\bar{x}(t))$, respectively. 	\end{remark} 

We now prove Theorem \ref{cauchyglobal}, which says that under global assumptions, the Cauchy problem corresponding to $(D)$  has a unique solution that is Lipschitz. Similar to the proof of Corollary \ref{corollary2}, the proof of Theorem \ref{cauchyglobal}  follows closely the arguments  used to prove Theorem \ref{theoremrelationship} after  removal of the truncation on $C(t)$ and $\R^l$.  However, doing so requires important modifications. For instance,  removing the truncation on $C(t)$ in the set $\mathscr{D}$ defined in  \eqref{Dlabel},  makes it   unsuitable for the global setting, and hence, it will have to be redefined (see \eqref{Dlabel2}). This discrepancy is due to having $C(t)\cap \bar{B}_{\bar{\varepsilon}}(\x(t))$  always generated by at least {\it two} functions, and hence,  $\bar{C}^{\gk}(t) \subset\inte C(t)$ is always valid.    While in the global setting,  for the case  $r=1$, \eqref{Cgkdef} yields that $C^{\gk}(t)=C(t)$ and hence, $\mathscr{D}$ must be modified to include Gr $C(\cdot)$.

\begin{proof}[\textbf{Proof of Theorem \ref{cauchyglobal}}]
 {We denote by $M_C$ the bound of Gr $C(\cdot)$}. Consider the Cauchy problem $(D)$ corresponding to $((x(0),y(0));u) = ((x_0,y_0);u)\in  (C(0)\times \R^l)\times \mathcal{U}$. The existence of a solution that is Lipschitz and unique, will be shown by approximating $(D)$ with $(D_{\gk})$, {defined below as} the global version of $(\bar{D}_{\gk})$. Let ${c}_{\gk}{\in  C^{\gk}(0,k)}$ be the sequence from the global version of Remark \ref{intelemma(k)} corresponding (and converging) to $c=x_0$ (see Remark \ref{global-results}). We now proceed with the proof by imitating the same steps of the proof of Theorem \ref{theoremrelationship},  in which we employ $(c_{\gk},d_{\gk}):=(c_{\gk},y_0) $ and  we make the following notable modifications. Inspired by the idea at the beginning  of  \cite[Section 4] {nourzeidan}, the uniform  {$\frac{2\eta}{L_{\psi}}$}-prox-regularity of $C(t)$ {and the $2$-Lipschitz property of $\pi_{C(t)}(\cdot)$ produce an extension of $h=(f,g)(t,\cdot,y,u)$ from $C(t)$ to { $C(t) + \frac{\eta}{ L_{\psi}} {B}$} satisfying}  (A4)$_G$,  where $C(t)$ is now replaced by $C(t) + \frac{\eta}{{ L_{\psi}}} {B}$, and $L_h(t)$ is now replaced by $2L_h(t)$. On the other hand, we note that since $\frac{\eta}{L_{\psi}} < \frac{{\rho}_o}{2}$, then $\psi_1, \cdots, \psi_r$ satisfy (A3.1)$_G$ on $\Gr \left(C(\cdot)\right) +\{0\}\times \frac{\eta}{L_{\psi}} \bar{B}$.  For fixed $k\in \N$ {large enough}, we consider the system $({D}_{\gk})$ corresponding to $((c_{\gk},y_0),u)$ to be 	\begin{equation}\label{Dgk2} \hspace{-1.6 cm}({{D}}_{\gk}) \begin{cases} \dot{x}(t)=f(t,x(t),y(t),u(t))-\sum\limits_{i=1}^{r} \gk e^{\gk\psi_i(t,x(t))} \nabla_x\psi_i(t,x(t)),\;\;\textnormal{a.e.}\; t\in[0,T],\\ 
			\dot{y}(t)=g(t,x(t),y(t),u(t)), \text{ a.e. } t \in [0,T]. \end{cases}
	\end{equation}
	This system is well defined on the following modified version of the set $\mathscr{D}$, given in \eqref{Dlabel},  		
	\begin{equation} \label{Dlabel2}
		\mathscr{D}_G:=\{ (t,x,y) \in [0,T] \times \mathbb{R}^n \times \mathbb{R}^l: \;\;x \in C(t) + \frac{\eta}{{ L_{\psi}}} B  \}.
	\end{equation}
	As $(0,c_{\gk}, y_0)\in \mathscr{D}_G$,  we follow steps similar to the ones used to reach \eqref{Tbar}, and we deduce that  a solution $(x_{\gk},y_{\gk})$ of $({D}_{\gk})$ with $(x(0),y(0))=(c_{\gk}, y_0)$ exists on the interval $[0,\hat{T}_G)\subset [0,T]$, and  $(t,x_{\gk}(t),y_{\gk}(t)) \in \mathscr{D}_G$, $\forall t\in [0,\hat{T}_G)$, where  $\hat{T}_G$ is $\hat{T}$ defined in \eqref{Tbar}, in which $\mathscr{D}$ is replaced by $\mathscr{D}_G$.
	{Simlarily to  Step I.1 in the proof of Theorem \ref{theoremrelationship}}, we conclude  that $x_{\gk}(t) \in C^{\gk}(t)$ for all $ t\in[0,\hat{T}_G)$.
	On the other hand,  since $\varphi$ is absent in $(D_{\gk})$ and $g$ is bounded by $M_h$ (from A4.1), we immediately  obtain  that,  for $ t\in [0,\hat{T}_G)$, $y_{\gk}(t) \in M_0\bar{B}$, where ${M_0:=\|y_0\|+M_hT}$. The boundedness of Gr $C(\cdot)$ by $M_C$  guarantees that the solution remains in the bounded set $M_C \bar{B} \times M_0 \bar{B}$ and hence,   $\hat{T}_G= T$. By mimicking  Step I.2.{ of  the proof of Theorem \ref{theoremrelationship}}, we obtain the invariance of $C^{\gk}(t,k)$ in $x$ for $(D_{\gk})$, and hence, our solution   $(x_{\gk}, y_{\gk})$ for  the Cauchy problem $({D}_{\gk})$ with $((c_{\gk}, y_0),u)$, also satisfies  $(x_{\gk}(t),y_{\gk}(t)) \in C^{\gk}(t,k) \times M_0\bar{B}$ for all $ t\in[0,T]$. Thus,  the definition of ${C}^{\gk}(t, k)$ given in \eqref{Cgkkdef}, $\xi_{\gk}^i(\cdot)=\gk e^{\gk\psi_i(\cdot, x_{\gk}(\cdot))}$ ($i=1,\cdots,r$) and $\xi_{\gk}(\cdot):= \sum_{i=1}^{r}\xi_{\gk}^i(\cdot)$,  yield that  $\|\xi_{\gk}\|_{\infty}\le \frac{2{\mu}}{{\eta}^2}$.  Employing this bound of $\xi_{\gk}$ in $({D}_{\gk})$, we obtain that $\max\{\|\dot{x}_{\gk}\|_{\infty},  \|\dot{y}_{\gk}\|_{\infty}\}  \le M_h + \frac{2{\mu}}{{\eta^2}} {L}_{\psi}$.
	It follows  that \eqref{boundxyxi-n} holds for $\mathcal{R}:=r$ and  $(x_k,y_k, \xi_k^i,\zeta_k):=(x_{\gk},y_{\gk}, \xi_{\gk}^i,0)$. 
	Whence, Lemma \ref{convergence}$(i)$ together with the global version of  Proposition \ref{propcgk(k)} $(iv)$ (see Remark \ref{global-results})  implies that a subsequence
	of $(x_{\gk}, y_{\gk})$, and $  \xi_{\gk}^i$  $(\forall i=1,\cdots,{r})$ converge  respectively   to some $(x,y)\in W^{1,\infty}([0,T];\mathbb{R}^n \times  \mathbb{R}^l)$,  and $( \xi^1, \cdots, \xi^{r})\in L^{\infty}([0,T]; \mathbb{R}^{r}_{+}),$ satisfying 
	\begin{eqnarray*}
		& &	(x_{\gk}, y_{\gk}) \xrightarrow[]{unif} (x, y), \quad 
		(\dot{x}_{\gk}, \dot{y}_{\gk}) \xrightarrow[in \; L^{\infty}]{w*} (\dot{x}, \dot{y}), \;\;\;\xi^i_{\gk}  \xrightarrow[in \; L^{\infty}]{w*} \xi^i \;(\forall i=1,\cdots,r), \label{converging2}\\
&&(x(t), y(t)) \in C(t)  \times M_0\bar{B}\; \;\;\forall t \in [0,T]; \;\;  \max\{\|\dot{x}\|_{\infty},  \|\dot{y}\|_{\infty}\}  \le M_h + \frac{2{\mu}}{{\eta^2}} {L}_{\psi};\; \; \|\sum_{i=1}^r\xi^i\|_{\infty}\le \frac{2{\mu}}{{\eta}^2},\label{boundofxy2l} \\
		&&
			\xi^i(t)=0 \textnormal{ for all } t \in I_i^{\-}(x), \;\; i \in \{1, \cdots, r\}, \quad
			\xi(t):=\sum_{i=1}^r\xi^i=0 \textnormal{ for all } t \in {I}^{\-}(x), \;\;  
	\end{eqnarray*}
	the validation of the last equations is  similar to that  for \eqref{xi-support}. 
	We now apply the dominated convergence theorem to {$(D_{\gk})$ at $((x_{\gk},y_{\gk}),u_{\gk}:=u)$  (as done in the proof of Case 1 in Lemma \ref{convergence} $(ii)$), and we deduce that $((x,y),u)$ and $\lambda_i=\xi^i$ satisfy \eqref{lipschitz-D}.  By means of Lemma \ref{lipschitz-global},  we conclude }that $(x,y)$ is the {\it unique} solution of  $({D})$ corresponding to $((x_0,y_0),u)$ and is {\it Lipschitz}. 
	\end{proof} 
We now prove the third main result of our article, namely, the  \textit{global} existence of optimal solutions {for our problem $(P)$} under global assumptions.
\begin{proof}  [\textbf{Proof of Theorem \ref{existenceglobal}}]
Let $M_C$ and $M_{\pi_2(S)}$ be the bounds of Gr $C(\cdot)$ and $\pi_2(S)$, respectively.	Observe that  $(({x},{y}),{u})$ is admissible  for $(P)$ is {\it equivalent} to $((x,y),u)$ solving $(D)$ and  \begin{equation*}\label{SG}
		({x}(0), {y}(0), {x}(T), {y}(T)) \in  S_G:= S \cap (C(0) \times M_{\pi_2(S)} \times C(T) \times M \bar{B} ),
	\end{equation*}
where $M:=M_{\pi_2(S)}  + T M_h$. Since $({P})$ has an admissible solution, $J$ is lower semicontinuous,   and  $S_G$ is compact, then the infimum of $J$  over {$((x,y),u)$ satisfying $({D})$ and  ({\bf{B.C.}})} exists. Let $((x_n, y_n), u_n)$ be a minimizing sequence for $({P})$. {Then,  for each $n\in\N,  ((x_n,y_n),u_n)$ satisfies  $(D)$, starting at $(x_0,y_0):=(x_n(0),y_n(0))$, and ({\bf{B.C.}}). Hence, by Lemma \ref{lipschitz-global},}   there exists a  sequence of $(\lambda_i^n)_{i=1}^r$ such that, for all $i=1,\cdots,r$,   $\lambda_i^n \in L^{\infty}([0,T];\R_+)$, $\lambda_i^n=0$ on $I^\-_i(x_n)$, and, $(\lambda_i^n)_{i=1}^r$ with $((x_n,y_n), u_n)$   satisfy \eqref{lipschitz-D}  and the bounds in \eqref{bound},  where $\eta_0$ is replaced by $\eta$.  Apply lemma \ref{convergence}$(i)$  to $\mathcal{R}:=r$, $(x_n, y_n)$,  $\xi^i_n:=\lambda_i^n,$ $\zeta_n:=0$, $M_1:= \max\{M_C,M\}$, $M_2:=M_h +\frac{\mu}{4\eta^2} L_{\psi}$, and $M_3:=\frac{\mu}{4{\eta}^2}$, 
	we obtain  $(\hat{x},\hat{y}) \in W^{1,\infty}$, $(\lambda_1,\cdots,\lambda_r)\in L^{\infty}([0,T];\R^r_+)$,  $\zeta:=0$, and  subsequences (not relabeled) of $(x_n,y_n)_n$ and $(\lambda_i^n)_n$ such that 
	$(x_{n}, y_{n}) \xrightarrow[]{unif}(\hat{x},\hat{y}),  (\dot{x}_{n}, \dot{y}_{n})  \xrightarrow[in \; L^{\infty}]{w*} (\dot{\hat{x}}, \dot{\hat{y}}),  \lambda_i^n \xrightarrow[in \; L^{\infty} ]{w*} \lambda_i $, for all $i=1,\cdots,r$, and $(\dot{\hat{x}}, \dot{\hat{y}})$ and $(\lambda_1,\cdots, \lambda_r)$ satisfy the bounds in \eqref{boundxyxi}. 
	Furthermore,  we have $(\hat{x}(0),\hat{y}(0),\hat{x}(T),\hat{y}(T)) \in {S}$.\\
	On the other hand, as $((x_n,y_n), u_n)$ and $(\lambda_i^n)_{i=1}^r$   satisfy  \eqref{lipschitz-D}, this means that  they solve \eqref{equiv} for $\zeta:=0$, 
	$q_i:=\psi_i,$ and $ Q(t):=C(t)$. Noting that $(A1)$ and $ {(A4.2)_G}$ hold,  then by applying the global version of Lemma \ref{convergence} $(ii)$ (see Remark \ref{global-lemma3}), we obtain $\hat{u}\in \mathcal{U}$ such that $((\hat{x},\hat{y}),\hat{u})$ and $(\lambda_i)_{i=1}^r$ also satisfy \eqref{equiv},  which is \eqref{lipschitz-D}.  Thus, to prove that $((\hat{x},\hat{y}),\hat{u})$ is admissible for $(D)$,  it suffices by the equivalence in Lemma \ref{lipschitz-global} to show that  for all $i=1,\cdots,r$,  $\lambda_i$ is supported on $I_i^0(\hat{x})$, knowing  that   for all $i=1,\cdots,r$,    $\lambda_i^n (t)=0 $ for $t \in I^{\-}_i(x_n)$. 	Fix $t \in I^{\-}_i(\hat{x})$,  then,  $\psi_i(t,\hat{x}(t))<0$. Since $x_n $ converges uniformly to $\hat{x}$ and $\psi_i(\cdot,\cdot)$ is continuous, we can find $\hat{\delta} >0$ and $\hat{n} \in \mathbb{N}$ such that $\forall s \in (t-\hat{\delta}, t+\hat{\delta}) \cap [0,T]$ and for all $n \ge \hat{n}$, we have $\psi_i(s, x_n(s)) <0$ and hence $\lambda_i^n(s)=0$. Thus, as $n\to\infty$,  $0=\lambda_i^n(\cdot) \to 0$ on $ (t-\hat{\delta}, t+\hat{\delta}) \cap [0,T]$,  and so  $\lambda_i(t)=0$. Therefore,  Lemma \ref{lipschitz-global} yields that $((\hat{x},\hat{y}),\hat{u})$ is admissible  for $({D})$.
	Using the lower semicontinuinity of $J$, we deduce that \begin{eqnarray*}
		J(\hat{x}(0), \hat{y}(0), \hat{x}(T), \hat{y}(T))&\le& \lim\limits_{n \to \infty}J(x_n(0), y_n(0), x_n(T), y_n(T) ) \\&=& \inf \limits_{((x,y),u)   \text{ admissible for } ({P}) }  J(x(0), y(0), x(T), y(T)),
	\end{eqnarray*}
	showing that   $((\hat{x},\hat{y}),\hat{u})$ is optimal  for $(P)$ over all admissible pairs  $((x,y),u)$.  
\end{proof}		
\section{ Maximum Principle for $(P)$ - Proof}\label{maximumprinciple}		
In this section, we present the proof of the {\textit{first}} main result of our paper: the maximum principle for the problem $(P)$. We employ a modification of the exponential penalization technique used  in \cite{depinho1,verachadihassan, nourzeidan} {for special cases of $(P)$}. {We first  approximate  the  given} optimal solution of $(P)$ with optimal solutions for some  approximating problems having {\it joint}-endpoint constraints,    $(P^{\a,\b}_{\gk})$, which are standard optimal control problems {involving exponential penalty terms} (Proposition \ref{approximating}).  Then, we  find necessary conditions for the approximating problems (Proposition \ref{maxapp}), and we finally conclude the necessary conditions for $(P)$ by taking the limit of the necessary conditions for $(P^{\a,\b}_{\gk})$ (see subsection \ref{proofpmp}).
\subsection{Preparatory results - PMP}
We start the first subsection by presenting   consequences of {$(A3.3)$} that shall be crucial for our approximating problem and the proof of the maximum principle. For this subsection, let $C(\cdot)$ satisfying $(A2)$ for $\rho>0$. Consider  $\bar{x}(\cdot)\in \mathcal{C}([0,T];\mathbb{R}^n)$ with $\bar{x}(t)\in C(t)$ for all $ t \in [0,T]$, and $\bar{\delta}>0$  such that $(A3.1)$ holds at $(\bar{x};\bar{\delta})$.\\
\\
The next remark discusses the significance of (A3.3) in  the proof of the maximum principle. In particular, it  highlights why it is sufficient to prove Theorem \ref{maxprincipleS} under a {\it stronger} assumption.
\begin{remark}\label{gram-a3.2}\textnormal{[Assumption (A3.3)]} Note that when $r=1$, the sets $C(t)$ are smooth and condition $(A3.3)$ trivially holds.  Let  $r>1$, then  the sets $C(t)$ are nonsmooth.  In this case, a condition closely related to $(A3.3)$, see \cite[Theorem 1.3.1]{li97},  has been first mentioned in  \cite{harrison}  to be useful when  sweeping (or reflected)  processes over {\it nonsmooth} sets are studied.
For $t\in I^0(\bar{x})$,  denote by $\mathcal{G}_{\psi}(t)$ the Gramian matrix of the vectors $\{\nabla_x\psi_i(t,\bar{x}(t)): i\in\I^0_{(t,\bar{x}(t))}\}$. 
	If for all  $i\in\I^0_{(t,\bar{x}(t))}$, we have $(A3.3)$ holds for $\bar{\beta}_i(t)\equiv1 $, then  the  matrix $\mathcal{G}_{\psi}(t)$ is   {\it  strictly diagonally dominant}.  For the general case,   $(A3.3)$ says  that for some positive Lipschitz vector function, $\bar{\beta}(\cdot)$,  the matrix  $\mathcal{G}_{\psi}(t) D_{\bar{\beta}(t)}(t)$ is strictly diagonally dominant for all $t\in I^0(\bar{x})$, where  $ D_{\bar{\beta}(t)}(t)$  is the diagonal matrix whose diagonal entries are $(\bar{\beta}_i(t))_{i\in\I^0_{(t,\bar{x}(t))}}$.    Thus,  
	\begin{enumerate}[label=$({\roman*})$]
		\item  $(A3.3)$ yields that the vectors  $\{\nabla_x\psi_i(t,\bar{x}(t)): i\in\I^0_{(t,\bar{x}(t))}\}$ are linearly independent, and hence, when $(A3.3)$ is assumed to hold, $(A3.2)$ is automatically satisfied;
		\item {Setting    $\tilde{\psi}_i(t,x):=\bar{\beta}_i(t) \psi_i(t,x)$,  it easily follows that   $C(t)$ is also the zero-sublevel sets of $(\tilde{\psi}_i(t,\cdot))_{i=1}^r$,   for $i=1,\cdots,r$, for some $L_{\tilde\psi}>0$, $\tilde\psi_i$ satisfies $(A3.1)$ for all $i=1,\cdots, r$, and {condition $(A3.3)$ is equivalent to saying that}  for $t\in I^0(\bar{x})$, the Gramian matrix  $\mathcal{G}_{\tilde\psi}(t)$ of  the vectors $\{\nabla_x\tilde\psi_i(t,\bar{x}(t)): i\in\I^0_{(t,\bar{x}(t))}\}$ is strictly diagonally dominant; a fact that shall be used in the proof of the maximum principle;
            \item From parts $(i)-(ii)$ of this remark, we have $\tilde\psi_1,\cdots, \tilde\psi_r$ satisfy  $(A3.2)$, and hence, \eqref{conditionii} of Lemma \ref{final-a3.2(ii)} is valid at 
           $\tilde\psi_1,\cdots, \tilde\psi_r, \psi_{r+1}$   when replacing   $\bar{\eta}$  by $\tilde{\eta}:=\bar{\eta}\;  {\bf{b}}_{\bar{\beta}}$, where
            \begin{equation*}
            {\bf{b}}_{\bar{\beta}}:= \min\left\{1,\min\{\bar{\beta}_i(t): t\in [0,T], \; i=1,\cdots,r\}\right\}.
             \end{equation*} }
\end{enumerate}
	\end{remark}
Equivalent forms for the strict diagonally dominance {of $\mathcal{G}_{\psi}(t)$}  are given in the following.
\begin{lemma}\label{A2.2lemma} The following assertions are equivalent:
	\begin{enumerate}[label=$({\roman*})$]
	\item For all $t\in I^0(\x)$,  the Gramian matrix  $\mathcal{G}_{\psi}(t)$ of the vectors $\{\nabla_x\psi_i(t,{\x}(t)): i\in\I^0_{(t,{\x}(t))}\}$, is strictly diagonally dominant.
	\item There exists $b\in(0,1)$ such that,  for all $t\in I^0(\x)$ and for all $i\in\I^0_{(t,\x(t))}$, we have \begin{equation}\label{A2.3Bis} \sum_{\substack{j\in\I^0_{(t,\x(t))}\\ j\not=i}}|\<\nabla_x \psi_i(t,\x(t)),\nabla_x\psi_j(t,\x(t)) \>|\leq b\|\nabla_x\psi_i(t,\x(t))\|^2.\end{equation} 
	\item  
	There exist  $\bar{c}>0$,  $\bar{b}\in(0,1)$, and $\bar{a}>0$ such that $\forall (t,x) \in \Gr C(\cdot)\cap \bar{B}_{\bar{c}}(\x(\cdot))$ with  $\I_{(t,x)}^{\bar{a}} \ne \emptyset$, \;\; \text{and}\;\; $\forall \; i\in\I_{(t,x)}^{\bar{a}},\;\; \text{we have}$  
	\begin{equation}\label{supp2} \hspace{-.2 in}\sum_{\substack{j\in\I^{\bar{a}}_{(t,x)}\\ j\not=i}}\abs{\<\nabla_x \psi_i(t,x),\nabla_x\psi_j(t,x)\>}\leq \bar{b}\|\nabla_x\psi_i(t, x)\|^2,\end{equation}
	\begin{equation}\label{Iax}\hspace{-0.26cm}\text{	where } \I^{\bar{a}}_{(t,x)}:=\{i\in\{1,\dots, r\}: -\bar{a}\le\psi_i(t,x)\le0\}.\;\;	\end{equation}	
\end{enumerate}
\end{lemma}
\begin{proof} $(i)\implies(ii)$: For $t\in I^0(\x)$ and $i\in\I^0_{(t,\x(t))}$, we define \begin{equation} \label{defbtj}b(t,i):=\frac{ 1} {\|\nabla_x\psi_i(t,\x(t))\|^2}              \displaystyle\sum_{\substack{j\in\I^0_{(t,\x(t))}\\ j\not=i}}|\<\nabla_x \psi_i(t,\x(t)),\nabla_x\psi_j(t,\x(t))\>|<1,\end{equation}
and set  $b:=\sup\big\{b(t,i) : t\in I^0(\x)\;\hbox{and}\; i\in\I^0_{(t,\x(t))}\big\}.$ Then,  \eqref{A2.3Bis} holds true. To show that $b<1$,  use an argument by contradiction, together  with Lemma \ref{lemmaaux1} and inequality \eqref{defbtj}. \vspace{0.1cm}\\
$(ii)\implies(iii)$:  We fix $\bar{b} \in (b,1)$ and we use an argument by contradiction in conjunction with Lemma \ref{lemmaaux1}. \\
$(iii)\implies(i)$: Follows directly by taking $\bar{a}:=\bar{c}:=0$, $x=\bar{x}(t)$ and using $\bar{b}<1$.\end{proof} 
	
\subsection{Approximating problem}\label{approx-problem}
 Assume that $(A1)$-$(A2)$ are satisfied, and $((\bar{x}, \bar{y}), \bar{u})$ is an {{\it admissible solution } }for $(P)$ with  $\bar{\delta}>0$ such that $(A3.1)$, $(A3.2)$, $(A4)$ and $(A5)$ are satisfied at $((\bar{x}, \bar{y}); \bar{\delta})$. Throughout the rest of this paper, let $ \bar{\varepsilon} \in (0, \bar{\delta})$, $\psi_{r+1}$, $\bar{\eta}$ and $\bar{\rho}=\frac{2\bar\eta}{L_\psi}$ be fixed as in Subsection \ref{approx},  with $L_{\psi} \ge \frac{4 \bar{\eta}}{\rho_o}$.  Let $L_{(\x,\y)}$ denote the Lipschitz constant of $(\bar{x}(\cdot),\bar{y}(\cdot))$,  which, by Lemma \ref{lipschitz-local}, is {\it Lipschitz} and  {\it uniquely} solves  $(D)$ corresponding to $((\x(0),\y(0)),\bar{u})$. Without loss of generality, we assume $L_{(\x,\y)}\ge1$. Therefore,  all the results of Subsection \ref{approx} are valid for the systems $(\bar{D})$ and $(\bar{D}_{\gk})$,  given respectively by \eqref{Dtilda} and \eqref{Dgk1}.\\ 
\\
For given $\delta\in (0,\bar\varepsilon]$,  define {\it the problem} $(\bar{P}_\delta)$ to be the problem $(P)$, in which $(D)$ is replaced by $(\bar{D})$, and $S$ is replaced by  ${S}_\delta$, where
\begin{eqnarray}\label{Sdelta} &&{S}_\delta:= S \cap \bar{\mathscr{B}}_{\delta},\; \,\text{and} \, \; \bar{\mathscr{B}}_{\delta} \;\text{is defined in}\;\; \eqref{mathscrB}.
 \end{eqnarray}
 Notice that when ${S}_\delta$ is replaced by the following set ${S}_{ \delta,\delta}$,  \begin{equation}\label{Sdeltadelta}
	 S_{\delta,\delta}=S\cap [{\bar{\mathscr{N}}}_{(\delta,\delta)}(0)\times {\bar{\mathscr{N}}}_{(\delta,\delta)}(T)]\subset S(\bar\delta) \subset \text{domain of }J,
\end{equation}
the resulting problem - named $(\bar{P}_{ \delta,\delta})$ - has the same sets of admissible and optimal solutions as $(\bar{P}_{\delta})$.\\
\\
The following is an {\it existence result  of an optimal solution}  for $(\bar{P}_{\delta})$ without requiring $(A5)$. 
\begin{theorem}\label{existencebarP} \textnormal{\textbf{[Existence of global optimal solution for $(\bar{P}_{\delta})$] }} 
Assume that all the aforementioned assumptions in the beginning of this subsection are satisfied except for $(A5)$. 
 Let ${J}:\mathbb{R}^n \times \mathbb{R}^l \times \mathbb{R}^n \times \mathbb{R}^l \to \mathbb{R} \cup \{ \infty\}$ be merely lower semicontinuous on $S(\bar\delta)$.  Then, for any $\delta\in (0, \bar{\varepsilon}]$, $(\bar{P}_{\delta})$ has a global optimal solution. 
\end{theorem}
\begin{proof}
Fix $\delta \in (0,\bar\varepsilon]$.  Being  admissible  for $(P)$,  $(({\x},{\y}),{\u})$ is also admissible for $(\bar{P}_{\delta})$, due to  Remark \ref{admissibility}   and that $( \x(0),\y(0),\x(T),\y(T))\in {S}_{ \delta}$. As  any admissible $((x,y),u)$ to $(\bar{P}_{\delta})$ also satisfies $( x(0),y(0),x(T),y(T)) \in {S}_{ \delta,\delta} \subset S(\bar{\delta})$,    then,  the lower 
semicontinuity of $J$ on $S(\bar{\delta})$ and the compactness of ${S}_{ \delta,\delta}$  yield  the infimum of $J$  over $((x,y),u)$ satisfying $(\bar{D})$ and having  its states endpoints in ${S}_{\delta}$, is finite. Let $((x_n, y_n), u_n)$ be a minimizing sequence for $(\bar{P}_{\delta})$. 
The proof from this point on continues as  done in the proof of Theorem \ref{existenceglobal}, in which $(D)$ and $S$ are  now $(\bar{D})$ and ${S}_{\delta}$, respectively,   and we use Lemma \ref{lipschitz-local-2},  system \eqref{lipschitz-xy1}, and the bounds in \eqref{lipschitz-xy} instead of Lemma \ref{lipschitz-global}, system \eqref{lipschitz-D}, and the bounds in \eqref{bound}($\eta_0$ was replaced by $\eta$),  respectively, and we apply Lemma \ref{convergence}  itself, where $\mathcal{R}=r+1$,  $Q(t):=C(t)\cap \bar{B}_{\bar{\varepsilon}}(\x(t))$  and $\zeta_n$ and $\zeta$ are present,  instead of its global version that was used  for $\mathcal{R}:=r$, $Q(t):=C(t)$ and $\zeta_n=\zeta=0$. We deduce the existence of {$((\tilde{x}_{\delta},\tilde{y}_{\delta}),\tilde{u}_{\delta})$} optimal for $(\bar{P}_{\delta})$.
\end{proof}
\begin{remark}\label{existenceL} 
Note that  Theorem \ref{existencebarP}$(I)$ remains valid if  we replace the objective function of $(\bar{P}_{\delta})$,  $J(x(0),y(0),x(T),y(T))$,  by  $J(x(0),y(0),x(T),y(T))+ \int_0^T \mathbb{L}(t,x(t),y(t))\, dt$,  where $\mathbb{L}$  is a {\it Carath\'eodory} function satisfying, for some $\sigma \in L^1([0,T],\R_+)$,   
     \begin{equation}\label{mathbbL} 
     |\mathbb{L}(t,x,y)| \le  \sigma(t),  \quad  \forall (x,y)\in {\bar{\mathscr{N}}}_{(\bar\varepsilon,\bar\delta)}(t), \;\text{and}\; \; t\in[0,T]. 
     \end{equation}
This is so, because for any $u\in \mathcal{U}$, the solution $(x,y)$ of $(\bar{D})$ belongs to the uniformly bounded set valued map ${\bar{\mathscr{N}}}_{(\bar\varepsilon,\bar\delta)}(\cdot)$, and $\mathbb{L}$ is a {\it Carath\'eodory} function satisfying \eqref{mathbbL}, and does not explicitly depend on the control.  Indeed, in the proof of Theorem \ref{existencebarP}(I), the existence of a  minimizing sequence $((x_n,y_n),u_n)$ for $(\bar{P}_{\delta})$, in which this change is implemented, remains valid, and  the limit as $n\to \infty$ of the added term is $\int_0^T\mathbb{L}(t,\tilde{x}_{\delta}(t),\tilde{y}_{\delta}(t)) dt $, by the dominated convergence theorem.\\
 \end{remark}
\begin{remark}\label{optimality} Using Theorem \ref{existencebarP} and  Remark \ref{admissibility},   we have the following.
\begin{itemize}
  \item[$(i)$] If  $((\bar{x}, \bar{y}), \bar{u})$ is a $\bar{\delta}$-{\it strong local minimizer} for $(P)$, {then, for any $\delta\in (0,\bar\varepsilon)$,   $((\bar{x}, \bar{y}), \bar{u})$ is a $\delta$-{\it strong local minimizer} for $(\bar{P}_{\delta})$.\\    
This fact  motivates  formulating  in Proposition \ref{approximating} the approximating problem for $(P)$ near $((\x,\y),\bar{u})$ as being that for $(\bar{P}_{ \delta_o,\delta_o})$, where  $\delta_o$ is chosen strictly less than $\bar\varepsilon$.  It  also plays a key role  in  step 4 of   the proof  of Theorem \ref{pmp-ps} $($Subsection  \ref{proofpmp}$)$   by relaxing instead of $(P)$, the problem $(\mathscr{\bar{P}})$, which is $(\bar{P}_{\frac{\delta}{2}})$ with extended $J$ and added $\mathbb{L}$}.
\item[$(ii)$] Conversely, given $\delta\in (0,\bar\varepsilon]$, if $((\x,\y),\bar{u})$   is a $\hat{\delta}$-strong local minimum for $(\bar{P}_{\delta})$ for  $\hat{\delta} \in (0,\delta]$, then $((\x,\y),\bar{u})$  is a $\hat\delta$-strong local minimum for $(P)$.  \\
\end{itemize}
\end{remark}

{For the rest of the paper,  $((\x,\y),\u)$ is  taken to be a $\bar{\delta}$-{\it  strong local minimum} for $(P)$.  We shall employ the following notations.
 \begin{itemize}
	  	\item  If $\bar{x}(0) \in \inte C(0)$, then,  $\bar{x}(0) \in \inte (C(0)\cap \bar{B}_{\bar{\varepsilon}}(\bar{x}(0))),$ and hence, taking  $c:=\bar{x}(0)$ in Remark \ref{intelemma(k)}$(ii)$,  we deduce that there exist $\hat{k}_{\bar{x}(0)} \in \N$ and $ \hat{r}_{\bar{x}(0)} {\in (0,\bar{\varepsilon})}$,  satisfying  \begin{equation} \label{lastideahope0} \bar{x}(0)\in \bar{B}_{\hat{r}_{\bar{x}(0)}}(\bar{x}(0))\subset  \inte \bar{C}^{\gk}(0,k),\;\;\forall k\ge \hat{k}_{\bar{x}(0)}.
 	\end{equation} If $\bar{x}(0) \in \bdry C(0)$, then $\bar{x}(0) \in \bdry (C(0)\cap \bar{B}_{\bar{\varepsilon}}(\bar{x}(0))),$ and hence, taking  $c:=\bar{x}(0)$ in Proposition \ref{propcgk(k)}$(v)$,  we deduce that there exist a vector $d_{\bar{x}(0)} \ne 0  $, ${k}_{\bar{x}(0)} \ge k_3 $, and  ${r}_{\bar{x}(0)} {\in (0,\bar{\varepsilon})}$, such that						
 		\begin{equation}\label{lastideahope1} \left(C(0)\cap \bar{B}_{{r}_{\bar{x}(0)}}({\bar{x}(0)})\right)+\bar{\sigma}_k\frac{d_{\bar{x}(0)}}{\|d_{\bar{x}(0)}\|}\subset \inte \bar{C}^{\gk}(0,k),\;\;\forall k\geq {k}_{\bar{x}(0)}.
 		\end{equation} 
 		\item Since $\bar{y}(0) \in \inte \bar{B}_{\bar{\delta}}(\bar{y}(0))$, then taking  $\textbf{d}:=\bar{y}(0)$ in Remark \ref{movingto-rhok}$(ii)$,  we deduce that there exist $\textbf{k}_{\bar{y}(0)} \in \N$  and  $ \textbf{r}_{\bar{y}(0)} >0$  satisfying  \begin{equation} \label{lastideahope2} \bar{y}(0)\in \bar{B}_{\textbf{r}_{\bar{y}(0)}}(\bar{y}(0))\subset  \inte  \bar{B}_{\bar{\rho}_k}(\bar{y}(0)),\;\;\forall k\ge \textbf{k}_{\bar{y}(0)}.
 		\end{equation}  
\item {Motivated by Remark \ref{optimality}$(i)$, and equations \eqref{lastideahope0}-\eqref{lastideahope2},} let   $\delta_o>0$  to be the fixed constant 
	\begin{equation}\label{delta_0} \delta_o := \begin{cases}
	\min \{\frac{\bar{\varepsilon}}{2} ,\hat{r}_{\bar{x}(0)}, \textbf{r}_{\bar{y}(0)}\}  \; & \text{ if } \bar{x}(0) \in \inte C(0) \\
		\min \{\frac{\bar{\varepsilon}}{2} ,r_{\bar{x}(0)}, \textbf{r}_{\bar{y}(0)}\} \; & \text{ if } \bar{x}(0) \in \bdry C(0).
	\end{cases}\end{equation}
\item For $ \beta \in (0,1]$,	we define for  $ t \in [0,T]$ \;a.e., and $ (x,y,u) \in { \bar{\mathscr{N}}_{(\bar{\delta},\bar{\delta})}(t)\times U(t): }$
	\begin{eqnarray*}
		f^{\beta}(t,x, y, u):= (1-\beta) f(t,x,y,\bar{u}(t)) + \beta f(t,x,y, u),  \\
		g^{\beta}(t,x, y, u):= (1-\beta) g(t,x,y,\bar{u}(t)) + \beta g(t,x,y, u).
	\end{eqnarray*}
	Note that also {$h^{\beta}:= ({f}^{\beta}, {g}^{\beta})$}  satisfy (A4) as $h$ does, and   hence,  all the results of {Section \ref{auxiliary}} hold true  for  $(\bar{D}^{\beta})$ and $(\bar{D}_{\gk}^{\beta})$,  which are respectively obtained from  $(\bar{D})$ and $(\bar{D}_{\gk})$ by replacing   $h$  by ${h}^{\beta}$. Observe that  $h^{\beta}(t,x,y,\bar{u}(t))=h(t,x,y,\bar{u}(t))$.
\item Let $(\bar{x}_{\gk},\bar{y}_{\gk} )$ the solution of $(\bar{D}_{\gk}^{\beta})$ corresponding to  $\left((\bar{c}_{\gk}, \bar{y}(0)), \bar{u}\right)$,  where, for $k$ large enough,  $\bar{c}_{\gk}\in \inte \bar{C}^{\gk}(0,k)$ is  the sequence  corresponding (and converging) to $c=\bar{x}(0)$ via  Remark \ref{intelemma(k)}, namely, \begin{equation*}
		\bar{c}_{\gk} =	\begin{cases}
		 \bar{x}(0) & \textnormal{ if }  \bar{x}(0) \in \inte C(0) \\
			 \bar{x}(0) + \bar{\sigma}_k \frac{d_{\bar{x}(0)}}{\|d_{\bar{x}(0)}\|}  & \textnormal{ if }  \bar{x}(0) \in \bdry C(0), 
		\end{cases}
	\end{equation*} 
{where $d_{\bar{x}(0)}$ is the vector from Proposition \ref{propcgk(k)}$(v)$ corresponding to $\bar{x}(0)$.} Then, by Corollary \ref{corollary2}, along a subsequence, $(\bar{x}_{\gk}, \bar{y}_{\gk})$ converges uniformly to $(\bar{x}, \bar{y})$ and satisfies all conclusions of Theorem $\ref{theoremrelationship}$. In particular, we have that $(\x_{\gk}(t),\y_{\gk}(t)) \in \bar{C}^{\gk}(t,k) \times \bar{B}_{\bar{\rho}_k}(\bar{y}(t)) \; \forall t \in [0,T]$ and $(\x_{\gk},\y_{\gk})_k$ is uniformly lipschitz.
	\item   We  define for all $k \in \mathbb{N}$
	\begin{eqnarray*} \hspace{-.1 in}&&S^{\gk}(k):= \begin{cases}
			[S_{\delta_o} + (0, 0,\bar{e}_{\gk},\bar{\omega}_{\gk})]\cap [\bar{\mathscr{N}}_{(\bar\varepsilon, \bar{\delta})}(0) \times \bar{\mathscr{N}}_{(\bar\varepsilon, \bar{\delta})}(T)],\quad & \textnormal{ if }  \bar{x}(0) \in \inte C(0) \\
				[S_{\delta_o}  + (\bar{\sigma}_k \frac{d_{\bar{x}(0)}}{\|d_{\bar{x}(0)}\|}, 0,\bar{e}_{\gk},\bar{\omega}_{\gk})]\cap [\bar{\mathscr{N}}_{(\bar\varepsilon, \bar{\delta})}(0)\times \bar{\mathscr{N}}_{(\bar\varepsilon, \bar{\delta})}(T)],\quad & \textnormal{ if } \bar{x}(0) \in \bdry C(0),
		\end{cases} 
\end{eqnarray*}
where, $S_{\delta_o}$ is defined in \eqref{Sdelta}, and   
\begin{eqnarray*}
&&(\bar{e}_{\gk}, \bar{\omega}_{\gk}):= (\bar{x}_{\gk}(T)-\bar{x}(T), \bar{y}_{\gk}(T)-\bar{y}(T))\xrightarrow[k \to \infty]{} (0,0).\end{eqnarray*}
	\end{itemize}
\begin{remark}
	Our sets $S^{\gk}(k)$ satisfy the following properties: 
	\begin{eqnarray}
		&& \hspace{-.7 in} \forall k\in \N,  \; S^{\gk}(k)  \text{ is closed},  \;  \quad \text{and} \; \quad {S^{\gk}(k) \subset {S}(\bar{\delta})}, \text{ for } k \text{ sufficiently large, } \label{boundS}\\          
		& &\hspace{-.8 in}\{ (c,d): (c,d,e,\omega) \in S^{\gk}(k) \} \subset \inte \bar{C}^{\gk}(0,k)  \times \inte \bar{B}_{\bar{\rho}_k}(\bar{y}(0)) \subset\inte\bar{\mathscr{N}}_{(\bar\varepsilon, \bar{\delta})}(0)\; \text{ for }\; k \; \text{large, }\label{subset} \\
		&&\hspace{-.7 in} \lim\limits_{k \to \infty} S^{\gk}(k) = {S}_{\delta_o,\delta_o }, \label{limit}  \\
		&&	\hspace{-.7 in}(\bar{c}_{\gk}, \bar{y}(0),\bar{x}_{\gk}(T), \bar{y}_{\gk}(T)) \in S^{\gk} (k),  \text{ for } k \text{ sufficiently large}.
		\label{inclusion}\end{eqnarray}
	In addition, for $(c,d,e,\omega) \in S^{\gk}(k)$ with   $(e,\omega)\in \inte \bar{\mathscr{N}}_{(\bar\varepsilon, \bar{\delta})}(T)= (\inte C(T)\cap {B}_{\bar\varepsilon}(\bar{x}(T))) \times {B}_{\bar\delta}(\bar{y}(T)),$ we have 
	\begin{eqnarray} \label{normalconeS} N_{S^{\gk}(k)}^L (c,d,e,\omega)= \begin{cases}
			N_{S}^L(c,d ,e-\bar{e}_{\gk}, \omega-\bar{\omega}_{\gk})  \hspace{1 in} \textnormal{ if }  \bar{x}(0) \in \inte C(0) \textnormal{ and } \\
			 \hspace{2.4 in}(c,d,e-\bar{e}_{\gk},\omega-\bar{\omega}_{\gk})  \in   \inte \bar{\mathscr{B}}_{\delta_o} \\
			\\
			N_{S}^L(c - \bar{\sigma}_k \frac{d_{\bar{x}(0)}}{\|d_{\bar{x}(0)}\|},d,e-\bar{e}_{\gk},\omega-\bar{\omega}_{\gk})  \hspace{.24 in}  \textnormal{ if } \bar{x}(0) \in \bdry C(0) \textnormal{ and } \\
		\hspace{1.6in} (c - \bar{\sigma}_k \frac{d_{\bar{x}(0)}}{\|d_{\bar{x}(0)}\|},d,e-\bar{e}_{\gk}, \omega-\bar{\omega}_{\gk}) \in  \inte \bar{\mathscr{B}}_{\delta_o}. \\
	\end{cases} \end{eqnarray}
\end{remark}
This next proposition provides a sequence of optimal control problems with specific \textit{joint} endpoint constraints that approximates our initial problem $(P)$ near $((\bar{x},\bar{y}),\bar{u})$, that is, the problem $(\bar{P}_{ \delta_o,\delta_o})$. \\	

\begin{proposition} \textnormal{\textbf{[Approximating problems for $(P)$]}}
		\label{approximating} 
For all $\alpha >0$ and $ \beta \in (0,1]$, there exists a subsequence of $(\gk)_k$ (we do not relabel) and a sequence  $\left(c_{\gk}, d_{\gk}, e_{\gk},\omega_{\gk}; u_{\gk}\right)\in S^{\gk}(k)\times \mathcal{U}$ such that the associated problem $(P_{\gk}^{\alpha,\beta}) $ defined by:\\
\\\begin{equation*}\label{dgkba} \begin{array}{l} ({P}_{\gk}^{\a,\b})\colon\; \textnormal{Minimize}\;J(x(0),y(0), x(T), y(T)) \;+ \; \alpha\|u-u_{\gk}\|_1  \\ \hspace{3cm} + \; \alpha \; \|\left(x(0),y(0),x(T),y(T)\right)-\left(c_{\gk} ,d_{\gk},e_{\gk},{\omega}_{\gk}\right)\| ,  \\ \hspace{1.5cm} \textnormal{over}\;((x,y),u)\;\textnormal{such that}\;u(\cdot)\in \mathscr{U}\;\textnormal{and}\\[1pt] \hspace{1.5cm} \begin{cases} ({\bar{D}}_{\gk}^{\beta})\begin{sqcases} 	\dot{x}(t) = f^{\beta}(t,x(t), y(t), u(t)) - \sum_{i=1}^{r+1}\gk e^{\gk \psi_i(t,x(t))} \nabla_x \psi_i(t,x(t))  \textnormal{ a.e. } t \in [0,T],\\ \dot{y}(t) = g^{\beta}(t, x(t), y(t), u(t))- \gamma_k e^{\gk \varphi(t,y(t))} \nabla_y \varphi(t,y(t)) \textnormal{ a.e. } t \in [0,T],\end{sqcases} \vspace{0.1cm}\\ 	(x(t),y(t)) \in \bar{B}_{\delta_o}(\bar{x}(t),\bar{y}(t)) \quad \forall t \in [0,T], \quad \quad \textnormal{\textbf{(S.C)}},	\\ 
			(x(0), y(0),x(T),y(T)) \in S^{\gk}(k),\\ \end{cases}
\end{array}
\end{equation*}
has an optimal solution $((x_{\gk},y_{\gk}),u_{\gk})$ such that 
$(x_{\gk}(0),y_{\gk}(0), x_{\gk}(T), y_{\gk}(T))=(c_{\gk},d_{\gk}, e_{\gk},\omega_{\gk})$ and $(x_{\gk})_k$ and $(y_{\gk})_k$ are uniformly Lipschitz. Moreover,  
\begin{eqnarray}
& &\hspace{-.8 in} (x_{{\gk}}(0), y_{{\gk}}(0),x_{{\gk}}(T), y_{{\gk}}(T)) \in 
			\left (S_{\delta_o}  + {\rho}_1 B \right) \cap \inte\left(\bar{\mathscr{N}}_{(\bar\varepsilon, \bar{\delta})}(0) \times \bar{\mathscr{N}}_{(\bar\varepsilon, \bar{\delta})}(T)\right) \subset \inte S(\bar{\delta}), \label{b11}\\
	& &\hspace{-.2 in}(x_{\gk}(t),y_{\gk}(t)) \in  \bar{C}^{\gk}(t,k) \times \bar{B}_{\bar{\rho}_k}(\bar{y}(t)), \quad \forall  t \in [0,T], \label{xkykdomain}\\		
						& &\hspace{-.2 in}(x_{\gk}, y_{\gk}) \xrightarrow[]{unif} (\bar{x}, \bar{y}), \quad  u_{\gk} \xrightarrow[in \; L^{1}]{strongly} \bar{u},\quad\text{and}\quad   (\dot{x}_{\gk}, \dot{y}_{\gk}) \xrightarrow[in \; L^{\infty}]{w*} (\dot{\x}, \dot{\y}). \label{convapp}		\end{eqnarray}
 The functions  $\xi^i_{\gk}$ $(i=1,\cdots,r+1)$ and $\zeta_{\gk}$,  corresponding to $x_{\gk}$ and $y_{\gk}$ via \eqref{defxi}, satisfy \eqref{boundxi}   and there exists $(\xi^1,\cdots,\xi^r)\in L^{\infty}([0,T],\R^r_+)$  such that		
\begin{eqnarray}			
& &\hspace{-1 in}    \xi^i_{\gk}  \xrightarrow[in \; L^{\infty}]{w*} \xi^i, \; \xi^i=0\; \text{on}\; I_i^{\-}(\x) \;(\forall i=1,\cdots,r), \; \|\sum_{i=1}^r\xi^i\|_{\infty} \le \frac{2\bar{\mu}}{\bar{\eta}^2}, \quad 
(\gk\xi^{r+1}_{\gk}, \gk\zeta_{\gk})  \xrightarrow[]{unif} 0,			
			\label{xi-zetaconv}
\end{eqnarray}			 
and $((\x,\y), \bar{u})$ together with $(\xi^1,\cdots,\xi^r)$ satisfies Theorem \ref{maxprincipleS} $(i)$.	\end{proposition}
	\begin{proof} 
Given that  $(\bar{x}_{\gk},\bar{y}_{\gk} )$ is the solution of $(\bar{D}_{\gk}^{\beta})$ corresponding to  $\left((\bar{c}_{\gk}, \bar{y}(0)), \bar{u}\right)$,  the inclusion \eqref{inclusion} holds, and  $(\bar{x}_{\gk},\bar{y}_{\gk} ) \to (\bar{x},\bar{y})$,  then for $k$ large enough, the sequence $((\bar{x}_{\gk},\bar{y}_{\gk} ),\bar{u})$ is admissible for $({P}_{\gk}^{\alpha, \beta})$   for every $\alpha \ge 0$ and $\beta \in (0,1]$.  In particular, for $k$ large enough,  $((\bar{x}_{\gk}, \bar{y}_{\gk}), \bar{u})$ is admissible for $({P}_{\gk}^{0, \beta})$. 	We fix $k$ large enough. Hence, for fixed $\beta\in(0,1]$,  by \cite[Theorem 23.10]{clarkebook}, $({P}_{\gk}^{0, \beta})$ admits an optimal solution $((\hat{x}_{\gk}, \hat{y}_{\gk}),  \hat{u}_{\gk})$. Using equations \eqref{boundS} and \eqref{limit}, we deduce that there exists $(c,d,e,\omega) $ such that,  up to a subsequence \begin{eqnarray*}
	(\hat{x}_{\gk}(0),\hat{y}_{\gk}(0),\hat{x}_{\gk}(T),\hat{y}_{\gk}(T)) \f (c,d,e,\omega) \in {S}_{ \delta_o,\delta_o}
		\end{eqnarray*}
			As $(\hat{x}_{\gk}(0),\hat{y}_{\gk}(0)) \in \bar{C}^{\gk}(0,k) \times \bar{B}_{\bar{\rho}_k}(\bar{y}(0))$ (see equation \eqref{subset}) and its limit $(c,d) \in \bar{\mathscr{N}}_{(\bar\varepsilon, \bar{\delta})}(0) $, then,  applying  Theorem \ref{theoremrelationship}(I) to {$((\hat{x}_{\gk}(0), \hat{y}_{\gk}(0)),  \hat{u}_{\gk})$}, we deduce {that the resulting unique solution of $(\bar{D}^{\beta}_{\gk})$ is} $(\hat{x}_{\gk}, \hat{y}_{\gk})$ and  satisfies \eqref{boundxy}-\eqref{defxi}, and hence,  by $(A1)$, $(A4.2)$,   and Theorem	 \ref{theoremrelationship}(II), 	 there exists  $((\hat{x},\hat{y}), u)$, such that along a subsequence of $(\hat{x}_{\gk}, \hat{y}_{\gk})$ (we do not relabel), we have $
			(\hat{x}_{\gk}, \hat{y}_{\gk}) \xrightarrow[]{unif} (\hat{x}, \hat{y}), 
	$
		   $(\hat{x}(t),\hat{y}(t)) \in  {\bar{\mathscr{N}}_{(\bar\varepsilon, \bar{\delta})}(t)}$ for all $ t \in [0,T]$, and $((\hat{x}, \hat{y}), u) \text{ uniquely solves } (\bar{D}^{\beta})$ starting at  $(c,d)$.  It follows that $(\hat{x}(T),\hat{y}(T))=(e,\omega)$. Moreover,  as $(\hat{x}_{\gk},\hat{y}_{\gk})$ satisfies \textbf{(S.C)},  then  we have $(\hat{x}(t),\hat{y}(t)) \in \bar{B}_{\delta_o}(\bar{x}(t),\bar{y}(t))$ for all $ t \in [0,T]$.  
		Using $(A4.2)$ and Filipov Selection Theorem, we can find $\hat{u} \in \mathcal{U}$ such that $((\hat{x}, \hat{y}), \hat{u})$ satisfies $(\bar{D})$, and hence $((\hat{x}, \hat{y}), \hat{u})$ is admissible for $(\bar{P}_{ \delta_o,\delta_o})$ with $\| (\hat{x},\hat{y})   -(\x,\y) \|_{\infty} \le \delta_o$. 

Since $((\bar{x}, \bar{y}), \bar{u})$ is a $\bar{\delta}$-strong local minimizer for $(P)$, then, by Remark \ref{optimality}$(i)$ and $\delta_o<\bar{\varepsilon}$, $((\bar{x}, \bar{y}), \bar{u})$ is a $\delta_o$-strong local minimizer for $(\bar{P}_{\delta_o})$ and hence for $(\bar{P}_{ \delta_o,\delta_o})$,	and hence, we have $$	J(\bar{x}(0),\bar{y}(0), \bar{x}(T), \bar{y}(T) ) \le J(\hat{x}(0),\hat{y}(0), \hat{x}(T), \hat{y}(T)). $$
		On the other hand,  $((\hat{x}_{\gk}, \hat{y}_{\gk}), \hat{u}_{\gk})$ is an optimal solution for $(P_{\gk}^{0, \beta})$ for which $((\bar{x}_{\gk}, \bar{y}_{\gk}),\bar{u}_{\gk})$ is admissible, we deduce that 
		\begin{equation*}
		J(\hat{x}_{\gk}(0),\hat{y}_{\gk}(0), \hat{x}_{\gk}(T), \hat{y}_{\gk}(T)) \le J(\bar{x}_{\gk}(0),\bar{y}_{\gk}(0), \bar{x}_{\gk}(T), \bar{y}_{\gk}(T) ).
		\end{equation*}
		Combining the above two inequalities and using the continuity of $J(\cdot,\cdot,\cdot,\cdot)$, we deduce that \begin{equation*}
		\lim\limits_{k \to \infty} \big [ J(\bar{x}_{\gk}(0),\bar{y}_{\gk}(0), \bar{x}_{\gk}(T), \bar{y}_{\gk}(T) )- J(\hat{x}_{\gk}(0),\hat{y}_{\gk}(0), \hat{x}_{\gk}(T), \hat{y}_{\gk}(T))  \big]=0.
		\end{equation*}
		Thus, for  fixed $\alpha >0$, there exists an increasing sequence $(k_n)_n$ such that  $\forall\;n  \ge 1$, $\forall\;k_n >n$, \begin{eqnarray}
			 J(\bar{x}_{{\gk}_n}(0),\bar{y}_{{\gk}_n}(0), \bar{x}_{{\gk}_n}(T), \bar{y}_{{\gk}_n}(T) ) &\le&  J(\hat{x}_{{\gk}_n}(0),\hat{y}_{{\gk}_n}(0), \hat{x}_{{\gk}_n}(T), \hat{y}_{{\gk}_n}(T)) + \frac{\alpha}{n}.  \nonumber 
		\end{eqnarray}
		 The rest of the proof follows from imitating the proof of \cite[Proposition 6.2]{nourzeidan}, and applying Ekeland Variational Principle (\cite[Theorem 3.3.1]{vinter}), to the following version of the data corresponding to our problem: 
		\begin{itemize}
			\item $X= \{ \left(c,d,e,\omega; u\right) \in S^{\gk}(k_n)  \times \mathcal{U}: \text{the unique solution }((x,y),u) \text{ of}\; (\bar{D}_{{\gk}_n}^{\beta}) \text{  with }\\ (x(0),y(0))=(c,d) \text{ satisfies } (x(T),y(T))=(e,\omega) \text{ and } (x(t),y(t) ) \in \bar{B}_{\delta_o} (\bar{x}(t),\bar{y}(t)) \; \forall t\}.$
			\item For $(c,d,e,\omega; u), \; (c',d',e',\omega'; u')  \in X$, we define the distance  $$\mathbb{D}\left((c,d,e,\omega; u), (c',d',e',\omega'; u') \right):= \|u-u'\|_{L^1} + \|(c,d,e,\omega)-(c', d',e',\omega')\| .$$
			\item For\; $(c,d,e,\omega; u) \in X$,\;$\mathbb{F}(c,d,e,\omega; u):= J(c,d,e,\omega).$ 
			\item $\alpha:=\alpha$ and $\lambda:= \frac{1}{n}$.
		\end{itemize}
		Notice that $(X, \mathbb{D})$ is a non-empty complete metric space, and $\mathbb{F}$ is continuous on $X$. Therefore, we deduce the existence of $(c_{{\gk}_n},d_{{\gk}_n}, e_{{\gk}_n},\omega_{{\gk}_n}; u_{{\gk}_n} ) \in X$ such that, for  $(x_{{\gk}_n}, y_{{\gk}_n})$, the solution  of $(\bar{D}_{{\gk}_n}^\beta)$ corresponding to $((c_{{\gk}_n},d_{{\gk}_n}), u_{{\gk}_n} ),$ satisfies   $(x_{{\gk}_n}(T), y_{{\gk}_n}(T)) =(e_{{\gk}_n},\omega_{{\gk}_n}) $ and  $(x_{{\gk}_n}(t), y_{{\gk}_n}(t)) \in \bar{B}_{\delta_o}(\bar{x}(t),\bar{y}(t))\; \forall t,$  and the following holds: \begin{itemize}
			\item $J(x_{{\gk}_n}(0), y_{{\gk}_n}(0), x_{{\gk}_n}(T), y_{{\gk}_n}(T)) \le J(\bar{x}_{{\gk}_n}(0),\bar{y}_{{\gk}_n}(0), \bar{x}_{{\gk}_n}(T), \bar{y}_{{\gk}_n}(T) ),$ 
			\item $\|u_{{\gk}_n} - \bar{u}\|_{L^1} + \|\left(c_{{\gk}_n},d_{{\gk}_n},e_{{\gk}_n},\omega_{{\gk}_n}\right)  - \left( \bar{c}_{{\gk}_n},\bar{y}(0), \bar{x}_{{\gk}_n}(T),\bar{y}_{{\gk}_n}(T) \right)\|   \le \frac{1}{n}, $
			\item  $\forall\; ((c,d,e,\omega);u) \in X$, we have \begin{eqnarray*}
			&J(x_{{\gk}_n}(0), y_{{\gk}_n}(0), x_{{\gk}_n}(T), y_{{\gk}_n}(T)) \le J(x(0), y(0), x(T), y(T))  \\&+ \alpha (\|u-u_{{\gk}_n} \|_{L^1} + \|\left(c,d,e,\omega \right)- \left(c_{{\gk}_n}, d_{{\gk}_n}, e_{{\gk}_n}, \omega_{{\gk}_n}\right)\|, 
			\end{eqnarray*}
			where $((x,y),u) \text{is the unique solution of}\; (\bar{D}_{{\gk}_n}^{\beta})$ starting with $ (x(0),y(0))=(c,d) $ and satisfying $ (x(T),y(T))=(e,\omega) \text{ and } (x(t),y(t) ) \in \bar{B}_{\delta_o} (\bar{x}(t),\bar{y}(t))\; \forall\, t\in[0,T]$. 
		\end{itemize}
		Hence, for $n$ large, the problem $(P_{{\gk}_n}^{\alpha, \beta})$ defined by means of  $((c_{{\gk}_n}, d_{{\gk}_n},e_{{\gk}_n}, \omega_{{\gk}_n}); u_{{\gk}_n})$, has $((x_{{\gk}_n}, y_{{\gk}_n}), u_{{\gk}_n})$ as optimal solution satisfying \begin{eqnarray*}
			& &\hspace{-.3in}	(x_{{\gk}_n}(0), y_{{\gk}_n}(0), x_{{\gk}_n}(T), y_{{\gk}_n}(T))=(c_{{\gk}_n}, d_{{\gk}_n},e_{{\gk}_n}, \omega_{{\gk}_n}) \xrightarrow[n \to \infty]{} (\bar{x}(0), \bar{y}(0),\bar{x}(T), \bar{y}(T)) \in S,\\
			& &	\hspace{1.7 in}u_{{\gk}_n} \xrightarrow[L^1]{strongly} \bar{u}, \quad  (x_{{\gk}_n}, y_{{\gk}_n}) \xrightarrow[n \to \infty]{unif} (\bar{x}, \bar{y}),
		\end{eqnarray*}
and all conclusions of Theorem $\ref{theoremrelationship}$.  Hence, \eqref{convapp}   is valid,  and,  for $(\xi_{\gk}^i)_{i=1}^{r+1}$ and $\zeta_{\gk}$ corresponding to $(x_{\gk}, y_{\gk})$ via \eqref{defxi}, there exist $(\xi^i)_{i=1}^{r+1}$ and $\zeta$ such that   \eqref{boundxi}, \eqref{converging}, \eqref{boundofxil}, \eqref{xi-support}, and \eqref{dual1}-\eqref{dual3} hold. 
Notice that,  as  $\psi_{r+1}(t,\x(t))=-\frac{\bar{\varepsilon}^2}{2} <0$ and $\varphi(t,\y(t))=-\frac{\bar{\delta}^2}{2}<0$  $\forall t\in [0,T]$,  we have that $\xi^{r+1}\equiv 0$, $\zeta \equiv 0$, and,  for some $\tilde{k}\in \N$,   $\psi_{r+1}(t,x_{\gk}(t))\le-\frac{\bar{\varepsilon}^2}{4} $ and $\varphi(t,y_{\gk}(t))\le-\frac{\bar{\delta}^2}{4},$ $\forall k\ge \tilde{k}$ and  $\forall t\in [0,T]$,  and hence,  \vspace{-.1in}
\begin{equation}\label{gk-xi-gk}\gk \xi_{\gk}^{r+1}(t)\le \gk^2 e^{-\gk\frac{{\bar\varepsilon}^2}{4}}  \quad \text{and}\quad \gk\zeta_{\gk}(t)\le  \gk^2 e^{-\gk\frac{{\bar\delta}^2}{4}}, \;\; \forall k\ge \tilde{k}, \;\; \text{and}\;\; \forall t\in [0,T].\end{equation}
That is,  $(\gk \xi_{\gk}^{r+1},\gk \zeta_{\gk})\xrightarrow[]{unif}  0,$  and thus, \eqref{xi-zetaconv} holds. Furthermore,  since  $h^{\beta}(t, \bar{x}(t), \bar{y}(t), \bar{u}(t))= h(t, \bar{x}(t), \bar{y}(t), \bar{u}(t))$, it follows  that Theorem \ref{maxprincipleS} $(i)$ is valid at $((\x,\y), \bar{u})$ and $(\xi^1,\cdots,\xi^r)$.\\
Finally,  \eqref{b11} is also valid, due to having  $(x_{{\gk}_n}(0), y_{{\gk}_n}(0), x_{{\gk}_n}(T), y_{{\gk}_n}(T)) \in S^{\gk}(k_n)$, ${\bar{\sigma}}_{k_n} \to 0$,   $(\bar{e}_{{\gk}_n}, \bar{\omega}_{{\gk}_n}) \to (0, 0)$ (as $n \to \infty$), and $(x_{{\gk}_n},y_{{\gk}_n}) \in \bar{C}^{{\gk}_n}(\cdot,k) \times \bar{B}_{\bar{\rho}_{k_n}}(\bar{y}(\cdot)) \subset  \inte \bar{\mathscr{N}}_{(\bar{\varepsilon}, \bar{\delta})}(\cdot )$. 
		  					 \end{proof}					 
The next result is obtained as a direct application  of the nonsmooth Pontryagin maximum principle for state constrained problems to each of the approximating problem $({P}_{\gk}^{\alpha,\beta})$ defined in Proposition $\ref{approximating}$. 
	\begin{proposition} \textnormal{\textbf{[Maximum principle for the approximating problems $(P_{\gk}^{\alpha,\beta})$]}} \label{maxapp}
Let $\alpha >0$ and $\beta \in (0,1]$ be fixed. Let $((x_{\gk}, y_{\gk}), u_{\gk})$ be the sequence from Proposition $\ref{approximating}$ which is optimal for $({P}_{\gk}^{\alpha,\beta})$ and satisfying $\lim\limits_{k \to \infty} \left(x_{\gk}(0), y_{\gk}(0),x_{\gk}(T), y_{\gk}(T)\right)=(\bar{x}(0), \bar{y}(0),\bar{x}(T), \bar{y}(T))$. Then, for each $k \in \mathbb{N}$,  there exist 
$ p_{\gk}=({q}_{\gk}, {v}_{\gk})\in W^{1,1}([0,T]; \mathbb{R}^n \times \mathbb{R}^l)$ and a scalar $\lambda_{\gk} \ge 0$ such that
	\begin{enumerate}[label=$({\roman*})$]
		\item \textbf{Nontriviality condition} For all $k \in \mathbb{N}$, we have \begin{equation} \|p_{\gk}{(T)}\| + \lambda_{\gk} =1. \label{m1}
		\end{equation}  
		\item  \textbf{Transversality equation} \begin{eqnarray} 
			 	(p_{\gk}(0), -p_{\gk}(T))\in&\lambda_{\gk} \partial^{L} J(x_{\gk}(0), y_{\gk}(0), x_{\gk}(T), y_{\gk}(T)) + \label{m2} \\
			& \alpha \bar{B} + N_{S^{\gk}(k)}^L\left(x_{\gk}(0), y_{\gk}(0),x_{\gk}(T), y_{\gk}(T)\right).  \nonumber
		\end{eqnarray}
		\item  \textbf{Maximization condition} \begin{equation} \max\limits_{u \in U(t)} \bigg \{ \left\langle (q_{\gk}(t),v_{\gk}(t)),  (f,g)(t, x_{\gk}(t), y_{\gk}(t), u)) \right\rangle  - \frac{\lambda_{\gk} \alpha}{\beta} \|u-u_{\gk}(t)\|   \bigg  \} 
			\label{m3}  
		\end{equation} is attained at $u=u_{\gk}(t)$ $\text{ a.e. }t \in [0,T]$. 
		\item  \textbf{Adjoint equation} For almost all $ t \in [0,T]$, 				\begin{eqnarray} 
			\hspace{-1 in} -\dot{p}_{\gk}(t) = \begin{bmatrix}
				-\dot{q}_{\gk}(t)  \\
				-\dot{v}_{\gk}(t)
			\end{bmatrix} 
			&\in& (1-\beta)(\partial^{(x,y)} f(t,x_{\gk}(t), y_{\gk}(t), \bar{u}(t)))^{T} q_{\gk}(t)  \nonumber\\
			& &+\beta (\partial^{(x,y)} f(t,x_{\gk}(t), y_{\gk}(t), u_{\gk}(t)))^{T} q_{\gk}(t) \nonumber\\ &&+(1-\beta)(\partial^{(x,y)}g(t,x_{\gk}(t), y_{\gk}(t),\bar{u}(t)))^{T} v_{\gk}(t)  \nonumber\\ & &+\beta(\partial^{(x,y)}g(t,x_{\gk}(t), y_{\gk}(t),u_{\gk}(t)))^{T} v_{\gk}(t) \nonumber\\
			&&- \begin{bmatrix}
				\left(\partial^{x} \left(\sum_{i=1}^{r+1}\gk e^{\gk \psi_i(t,x_{\gk}(t))} \nabla_x \psi_i (t,x_{\gk}(t))\right)\right)^{T} q_{\gk}(t)  \\
				{	\left({\nabla_{y}} \left(\gk e^{\gk \varphi(t,y_{\gk}(t))} \nabla_y \varphi (t,y_{\gk}(t))\right)\right)^{T} v_{\gk}(t)  }
			\end{bmatrix} , \label{m4}
		\end{eqnarray}
	where,
		\begin{eqnarray*}
			\hspace{-.2 in}
			\partial^{x} \left(\sum_{i=1}^{r+1}\gk e^{\gk \psi_i(t,x_{\gk}(t))}     \nabla_x \psi_i (t,x_{\gk}(t))\right)&\subset& \sum_{i=1}^{r+1}\gk e^{\gk \psi_i(t,x_{\gk}(t))} \partial^{xx} \psi_i(t,x_{\gk}(t))  \\
			&&+\sum_{i=1}^{r+1}\gk^2 e^{\gk \psi_i(t,x_{\gk}(t))} \nabla_x \psi_i(t,x_{\gk}(t))\nabla_x \psi_i(t,x_{\gk}(t))^{T} ,  \nonumber \\
			\hspace{-.5 in}{\nabla_y} \left(  \gk e^{\gk \varphi(t,y_{\gk}(t))} \nabla_y \varphi (t,y_{\gk}(t)) \right)&=& \gk e^{\gk \varphi(t,y_{\gk}(t))} \textnormal{ I}_{l\times l}  \\
			&&+\gk^2 e^{\gk \varphi(t,y_{\gk}(t))} \nabla_y \varphi(t,y_{\gk}(t))\nabla_y \varphi(t,y_{\gk}(t))^{T}.  \nonumber
		\end{eqnarray*}
	\end{enumerate}
	{In addition, if $S= C_0\times \R^{n+l}$ for a closed $C_0\subset C(0)\times \R^l$, then $\lambda_{\gk}\neq 0$ and it is taken to be $1$ and the nontriviality
condition $(i)$ is eliminated.} \end{proposition}
\begin{proof}
	As $({P}_{\gk}^{\alpha,\beta})$ is a standard optimal control problem with implicit state constraints, we shall apply   \cite[Theorem 9.3.1 and P.332]{vinter}  for  the optimal solution $((x_{\gk}, y_{\gk}),u_{\gk})$ of $(P_{\gk}^{\alpha,\beta})$ obtained in Proposition $\ref{approximating}$.
Note that  for $k$ large enough, the required constraint qualification $(CQ)$ in  \cite[Page 332]{vinter}  is satisfied  by the multifunction $\bar{B}_{{\delta_o}}(\bar{x}(\cdot),\bar{y}(\cdot))$ at $(x_{\gk}(t),y_{\gk}(t))$.  In other words,   $\conv$($\bar{N}^L_{\bar{B}_{\delta_o}(\bar{x}(t),\bar{y}(t))}(x_{\gk}(t),y_{\gk}(t)) $) is pointed $\forall t \in [0,T]$, where the graph of $\bar{N}^L_{\bar{B}_{\delta_o}(\bar{x}(\cdot),\bar{y}(\cdot))}(\cdot)$ is defined to be the closure of the graph of  ${N}^L_{\bar{B}_{\delta_o}(\bar{x}(\cdot),\bar{y}(\cdot))}(\cdot)$.  This is due to  $(x_{\gk}, y_{\gk})$ converging uniformly to $(\x,\y)$  and to ${\bar{B}_{\delta_o}(\bar{x}(\cdot),\bar{y}(\cdot))}$   being lower semicontinuous, with  closed, convex, and  nonempty interior values (hence epi-Lipschitz), (see \cite[Remark 4.8$(ii)$]{nourzeidan2}). Notice that 
the measure $\eta_{\gk} \in C^{*}([0,T],\mathbb{R}^{n+l})$  corresponding to the state constraint (S.C)  produced  by \cite[Theorem 9.3.1]{vinter},    is actually null. This is due to the fact that   its support satisfies 	 \begin{eqnarray*}
		\text{supp }\{ \eta_{\gk} \} &\subset& \{t \in [0,T]: (t,x_{\gk}(t),y_{\gk}(t)) \in \bdry \text{Gr } \bar{B}_{\delta_o}(\bar{x}(\cdot),\bar{y}(\cdot)) \},\\
			&=& \{t \in [0,T]: (t,x_{\gk}(t),y_{\gk}(t)) \in \cup_{t\in [0,T]} \{ t\} \times \mathscr{S}_{\delta_o} (\bar{x}(t),\bar{y}(t)) \}\\
			&=& \emptyset, 
		\end{eqnarray*} 
		where $\mathscr{S}_{\delta_o} (\bar{x}(t),\bar{y}(t)) =\{ (x,y): \|(x,y)-(\bar{x}(t),\bar{y}(t)) \| = \delta_o\}$. The last equality follows from the uniform convergence to $(\x,\y)$ of  $(x_{\gk}(t),y_{\gk}(t))$, \eqref{convapp}.  The rest of the proof is obtained from  translating the conditions of \cite[Theorem 9.3.1]{vinter} to our data, and using the standard state augmentation technique. 
			\end{proof}
\subsection{Proof of the maximum principle}\label{proofpmp}
 We now prove the \textit{first} main result of our paper, that is, the maximum principle for $(P)$. We start by proving the theorem under the temporary assumption (A4.2), and without assuming any uniform bound on the sets $U(t)$ (Steps 1-2-3). In Step 4, we show that, when the compact sets $U(t)$ are {\it uniformly} bounded, the convexity assumption (A4.2) can be removed.     
 \begin{proof}[\textnormal{\textbf{Proof of Theorem \ref{maxprincipleS}}}]\label{maxinter} 
Assume for now the temporary assumption (A4.2) holds true. All the  previous results including the consequences in subsection \ref{approx-problem} are valid. In particular, $(\bar{x}(\cdot),\bar{y}(\cdot))$ is $L_{(\x,\y)}$-Lipschitz with  $L_{(\x,\y)} \ge 1.$ By {Remark \ref{gram-a3.2}$(ii)$-$(iii)$ and using arguments similar to those  at the last step of the proof of \cite[Theorem 3.1]{nourzeidan3}},     it is sufficient to prove Theorem \ref{maxprincipleS} under the additional assumption:\\
 (\textbf{A3.3})$'$  $\forall t\in I^0(\x)$,  $\mathcal{G}_{\psi}(t)$, {\it the Gramian matrix of the vectors} $\{\nabla_x\psi_i(t,{\x}(t)): i\in\I^0_{(t,{\x}(t))}\}$, 
is\vspace{-.2 in}\\  

$\hspace{.34 in}${\it strictly diagonally dominant.}  \\  
 \\
 Since $\{\psi_i\}_{i=1}^r$ satisfy (A3.3)$'$, then by Lemma \ref{A2.2lemma}, there exist $0<\bar{a}\le 2a_o$, $0<\bar{b}<1$, and $\bar{c}>0$ such that \eqref{supp2} is satisfied, where $a_o$ is the constant in  Remark \ref{pisiepse}.\\
 \\
 We begin our proof by introducing the functions $\hat{\psi}_i$ ($i=1,\cdots, r$),  which generate the same $C(t)$ as $\psi_i$, and have the special feature that their corresponding adjoint variable $q_{\gk}$, via Proposition \ref{maxapp}, is of bounded variation (Step 2). Theorem \ref{maxprincipleS} will be first established in terms of  $\hat{\psi}_i$'s and then translated back in terms of  $\psi_i$'s (Step 3). \\
 \\
 Define the following function $\hat{\psi}_i(\cdot,\cdot)$ on the same domain of $\psi_i(\cdot,\cdot)$ as 
 \begin{equation}\label{hatpsi}\hat{\psi}_i(t,x) := \begin{cases}
 	\psi_i(t,x)   \quad &\text{ if } -\frac{\bar{a}}{2} \le \psi_i(t,x) \le 0  \;\; \text{ or }\;\;  \psi_i(t,x)>0 \\
 	s(\psi_i(t,x)) \quad &\text{ if } -\bar{a} \le \psi_i(t,x) < -\frac{\bar{a}}{2} \\
 	s(-\bar{a})  \quad  &\text{ if } \psi_i(t,x) <-\bar{a},
 \end{cases}\end{equation}
 where $$s(z):=-\frac{3}{4}\bar{a} + \frac{1}{\bar{a}}(z+\bar{a})^2, \text{ for } -\bar{a} \le z \le -\frac{\bar{a}}{2}.  $$
 Notice that $s(\cdot)$ is a quadratic function with: \begin{itemize}
 	\item $s(-\bar{a})=-\frac{3}{4} \bar{a}$ and $s(-\frac{\bar{a}}{2})=-\frac{\bar{a}}{2}$.
 	\item $s'(-\bar{a}) =0 $ and $s'(-\frac{\bar{a}}{2})=1$.
 	\item $0\le s'(z)\le 1$ for all $-\bar{a} \le z \le -\frac{\bar{a}}{2}$. 
 \end{itemize}
 We also have 
 \begin{equation}\label{hatgradpsi}\nabla_x\hat{\psi}_i(t,x) := \begin{cases}
 	\nabla_x\psi_i(t,x)   \quad &\text{ if } -\frac{\bar{a}}{2} \le \psi_i(t,x) \le 0 \;\; \text{ or }\;\;  \psi_i(t,x)>0 \\
 	s'(\psi_i(t,x))\;\nabla_x\psi_i(t,x) \quad &\text{ if } -\bar{a} \le \psi_i(t,x) < -\frac{\bar{a}}{2} \\
 	0  \quad  &\text{ if } \psi_i(t,x) <-\bar{a}.
 \end{cases}\end{equation}
 Notice the following:
 \begin{itemize}
 	\item $ \{x \in \mathbb{R}^n: \hat{\psi}_i(t,x) \le 0, \forall i=1,\cdots, r \}= \{x \in \mathbb{R}^n: {\psi}_i(t,x) \le 0, \forall i=1,\cdots, r \}= C(t).$
 	\item Since $\{\psi_i\}_{i=1}^r$ satisfy (A3.1) and (A3.2), then $\{\hat{\psi}_i\}_{i=1}^r$ satisfy (A3.1) and (A3.2) with \\  $L_{\hat{\psi}}:=L_{\psi}(1+\frac{2}{\bar{a}}L_{\psi})$ replacing $L_{\psi}$.
 	\item Hence, all results of Subsection \ref{approx-problem}, in particular, Proposition \ref{approximating} and Proposition \ref{maxapp}, are also valid when we replace ${\psi}_i$ ($i=1,\cdots,r$) by  $\hat{\psi}_i$($i=1,\cdots,r$). 
 	\item Since $\{\psi_i\}_{i=1}^r$ satisfy (A3.3)$'$, and equation \eqref{supp2} is satisfied, we deduce that $\forall (t,x) \in \Gr C(\cdot)\cap \bar{B}_{\bar{c}}(\x(\cdot))$ with  $\I_{(t,x)}^{\frac{\bar{a}}{2}} \ne \emptyset$, \;\; \text{and}\;\; $\forall \; i\in\I_{(t,x)}^{\frac{\bar{a}}{2}},\;\; \text{we have}$  
 	\begin{equation} \hspace{-.2 in}\sum_{\substack{j\in\I^{\bar{a}}_{(t,x)}\\ j\not=i}}\abs{\<\nabla_x \hat{\psi}_i(t,x),\nabla_x\hat{\psi}_j(t,x)\>}\leq \bar{b}\|\nabla_x\hat{\psi}_i(t, x)\|^2. \label{supp3}\end{equation}	
 	This is due to the fact that $\hat{\psi}_i(t,x) =\psi_i(t,x)$ for $i \in \I_{(t,x)}^{\frac{\bar{a}}{2}}$, and $s'(z)\le 1$ $\;\; \forall -\bar{a} \le z \le -\frac{\bar{a}}{2}$. Note that, unlike $\psi_i$'s, the functions $\hat{\psi}_i$'s do not necessarily satisfy  \eqref{supp2}, since \eqref{supp3} holds only for $i\in\I_{(t,x)}^{\frac{\bar{a}}{2}}$ and not necessarily for $i\in\I_{(t,x)}^{\bar{a}}$.\\
 \end{itemize} 
 \textbf{Step 1.} Results from Proposition \ref{approximating} and  Proposition \ref{maxapp}.\\
 Fix $ \alpha >0$ and $\beta \in (0,1]$. Since Proposition $\ref{approximating}$ holds for $\hat{\psi}_i$, then there exist a subsequence of $(\gamma_k)_k$ (we do not relabel),  an optimal solution $((x_{\gk}, y_{\gk}), u_{\gk})$ for $(P_{\gk}^{\alpha,\beta})$ with corresponding $({\xi}_{\gk}^1, \cdots, {\xi}_{\gk}^r,  \xi_{\gk}^{r+1},\zeta_{\gk})$ via \eqref{defxi}, and $({\xi}^1,\cdots,{\xi}^{r})\in L^{\infty}([0,T];\R^{r}_+),$ such that     \eqref{b11}-\eqref{xi-zetaconv} hold   and  $((\x,\y),\bar{u})$ together with $({\xi}^1,\cdots,{\xi}^r)$ satisfies Theorem \ref{maxprincipleS} $(i)$ in terms of $\hat\psi_i$'s.   Moreover, Proposition $\ref{maxapp}$ produces $\forall k \in \mathbb{N}$,   $ {p}_{\gamma_k}=({q}_{\gamma_k}, {v}_{\gamma_k}) \in W^{1,1}([0,T]; \mathbb{R}^n \times \mathbb{R}^l ) $, and  ${\lambda}_{\gamma_k} \ge 0$ such that equations (\ref{m1})-(\ref{m4}) are valid. For simplicity of the presentation, we suppress the ``hat'' notations from $x_{\gk}, y_{\gk}, u_{\gk}, \xi_{\gk}^i, \xi^i, p_{\gk}, q_{\gk}, v_{\gk}$, and $\lambda_{\gk} $, as well as the $(\alpha, \beta)$-dependency, which shall only be made visible at the stage when the limit in  $(\alpha, \beta)$ is performed. 
	\\
	Since $(x_{\gk}(t), y_{\gk}(t)) \in  \inte  \left(C(t) \cap \bar{B}_{\bar{\delta}}(\bar{x}(t))\right) \times B_{\bar{\delta}}(\bar{y}(t)) $ for all  $t \in [0,T]$, then \begin{eqnarray*}
		\partial^{(x,y)} (f,g)(t, x_{\gk}(t), y_{\gk}(t), u_{\gk}(t)) &=& \partial^{(x,y)}_{\ell} (f,g)(t, x_{\gk}(t), y_{\gk}(t), u_{\gk}(t)), \\
		\partial^{(x,y)} (f,g)(t, x_{\gk}(t), y_{\gk}(t), \bar{u}(t)) &=& \partial^{(x,y)}_{\ell} (f,g)(t, x_{\gk}(t), y_{\gk}(t), \bar{u}(t)),  \\
		\partial^{xx} \hat{\psi}_i(t,x_{\gk}(t)) &=& \partial^{xx}_{\ell} \; \hat{\psi}_i(t,x_{\gk}(t)) \textnormal{ for } i=1, \cdots, r+1.   
	\end{eqnarray*}
	Also,  
	$(x_{{\gk}}(0), y_{{\gk}}(0),x_{{\gk}}(T), y_{{\gk}}(T)) \in \inte {S}(\bar{\delta}),$ 
	 yields
	\begin{equation*}
		\partial^{L} J(x_{\gk}(0), y_{\gk}(0), x_{\gk}(T), y_{\gk}(T)) = \partial^{L}_{\ell} J(x_{\gk}(0), y_{\gk}(0), x_{\gk}(T), y_{\gk}(T)).
	\end{equation*}
 Using $(A4.1)$, \eqref{convapp}$(i)$, and Filipov Selection Theorem,  equation $(\ref{m4})$ yields  the existence of  measurable ${\bar{A}}_{\gk}(\cdot)$, ${A}_{\gk}(\cdot)$ in $\mathcal{M}_{n \times n}[0,T]$, ${\bar{E}}_{\gk}(\cdot)$, ${E}_{\gk}(\cdot)$ in $\mathcal{M}_{n \times l}[0,T]$, ${\bar{\mathcal{A}}}_{\gk}(\cdot)$,  ${\mathcal{A}}_{\gk}(\cdot)$ in $\mathcal{M}_{l \times n}[0,T]$, ${\mathcal{E}}_{\gk}(\cdot)$,  ${\bar{\mathcal{E}}}_{\gk}(\cdot)$ in $\mathcal{M}_{l \times l}[0,T]$, ${\vartheta}^i_{\gk}(\cdot)$  in $\mathcal{M}_{n \times n}[0,T]$ such that for almost all $t\in[0,T]$, 
 \begin{eqnarray*}
 	&&\hspace{-.2 in}({\bar{A}}_{\gamma_k}, {\bar{E}}_{\gamma_k})(t) \in \partial^{(x,y)}_{\ell} f(t,x_{\gamma_k}(t), y_{\gamma_k}(t), \bar{u}(t)), \;\; 
 	({A}_{\gamma_k}, {E}_{\gamma_k})(t) \in \partial^{(x,y)}_{\ell} f(t,x_{\gamma_k}(t), y_{\gamma_k}(t), u_{\gamma_k}(t));\\
 	& &\hspace{-.2 in}({\bar{\mathcal{A}}}_{\gamma_k}, {\bar{\mathcal{E}}}_{\gamma_k})(t) \in \partial^{(x,y)}_{\ell} g(t,x_{\gamma_k}(t), y_{\gamma_k}(t), \bar{u}(t)),\;\;
 	({\mathcal{A}}_{\gamma_k}, {\mathcal{E}}_{\gamma_k})(t) \in \partial^{(x,y)}_{\ell} g(t,x_{\gamma_k}(t), y_{\gamma_k}(t), u_{\gamma_k}(t)); \\
 	&&\hspace{-.2 in}{\vartheta}^i_{\gk}(t) \in \partial^{xx}_{\ell} \hat{\psi}_i(t,x_{\gk}(t))  \text{ for } i=1, \cdots, r, \quad \;\;\; \vartheta^{r+1}_{\gk}(t)= \textnormal{I}_{n \times n}; \\
 	&&\hspace{-.2 in}\max\left\{{\|({\bar{A}}_{\gk}, {\bar{E}}_{\gk})\|}_2,\; {\|({A}_{\gamma_k}, {E}_{\gamma_k})\|}_2,\; {\|({\bar{\mathcal{A}}}_{\gamma_k},\; {\bar{\mathcal{E}}}_{\gamma_k})\|}_2,\;
 	{\|({\mathcal{A}}_{\gamma_k}, {\mathcal{E}}_{\gamma_k})\|}_2\right\} \le {\|L_h\|}_2; \\
 	&&\hspace{-.2 in}  \|{\vartheta}^i_{\gk}\|_{\infty} \le L_{\hat{\psi}}  \text{ for } i=1, \cdots, r, \quad\;\; \; \; \|\vartheta^{r+1}_{\gk}\|_{\infty}=1,\;\; \text{ and } 
 \end{eqnarray*}
 \vspace{-0.3 in}
 \begin{eqnarray}
 	\dot{{q}}_{\gamma_k}(t) &=& \underbrace{-\big[(1-\beta) {\bar{A}}_{\gk}^T(t)   + \beta {A}_{\gk}^T(t) \big]{q}_{\gk}(t) - \left[(1-\beta) {\bar{\mathcal{A}}}_{\gk}(t)^T  + \beta {\mathcal{A}}_{\gk}^T(t) \right] {v}_{\gamma_k}(t) }_{\mathcal{Q}_{\gk(t)}} \nonumber\\ &&
 	+ \underbrace{ \sum_{i=1}^{r} 
 		\gamma_k e^{\gk \hat{\psi}_i(t,x_{\gk}(t))}  {\vartheta}^i_{\gk}(t){q}_{\gk}(t) + \gamma_k e^{\gk \psi_{r+1}(t,x_{\gk}(t))} {q}_{\gk}(t)}_{\mathcal{X}_{\gk(t)}} \nonumber \\ &&
 	+\underbrace{\sum_{i=1}^{r}\gamma_k^2 e^{\gk \hat{\psi}_i(t,x_{\gk}(t))} \nabla_x \hat{\psi}_i(t,x_{\gk}(t)) \<\nabla_x \hat{\psi}_i(t,x_{\gk}(t)), {q}_{\gamma_k}(t) \> }_{\mathcal{Y}_{\gk}(t)} 
 	\nonumber\\	
 	& & +\underbrace{\gamma_k^2 e^{\gk \psi_{r+1}(t,x_{\gk}(t))} \nabla_x \psi_{r+1}(t,x_{\gk}(t)) \<\nabla_x \psi_{r+1}(t,x_{\gk}(t)), {q}_{\gamma_k}(t) \>}_{\mathcal{Z}_{\gk}(t)}, \label{ref4a} \\
 	\hspace{-.2 in}	\dot{{v}}_{\gamma_k}(t) &=& -\left[(1-\beta) {\bar{E}}_{\gamma_k}(t)^T + \beta {E}_{\gamma_k}(t)^T\right] {q}_{\gamma_k}(t) - \left[(1-\beta) {\bar{\mathcal{E}}}_{\gamma_k}(t)^T  + \beta {\mathcal{E}}_{\gamma_k}(t)^T\right] {v}_{\gamma_k}(t)  \nonumber \\& &  \hspace{-0.5cm}+{\gk e^{\gk \varphi(t,y_{\gk}(t))}  {v}_{\gamma_k}(t)} +\gamma_k^2 e^{\gk \varphi(t,y_{\gk}(t))} \nabla_y \varphi(t,y_{\gk}(t)) \<\nabla_y \varphi(t,y_{\gk}(t)), {v}_{\gamma_k}(t) \>.  \label{ref4b}
 \end{eqnarray}
\textbf{Step 2.} Uniform boundedness of $\{{p}_{\gk}\}$,  $\{\|\dot{{v}}_{\gk}\|_2\}$, and $\{\|\dot{{q}}_{\gk}\|_1\}$  \\
Our  proof of the uniform boundedness of $\{{p}_{\gk}\}$  is a generalization to our general setting of  the corresponding  proof in  \cite[Theorem 3.1]{nourzeidan3}, while that of  $\{\|\dot{{q}}_{\gk}\|_1\}$ requires a new approach.  Observe that
\begin{eqnarray*}
	\frac{1}{2}\frac{d}{dt}\|{p}_{\gk}(t)\|^2&= &
	\< {q}_{\gk}(t), \dot{{q}}_{\gk}(t) \> + \< {v}_{\gk}(t), \dot{{v}}_{\gk}(t) \> \\
	& &\hspace{-2.5cm}\overset{(\ref{ref4a})+(\ref{ref4b})}{=}  \left\< {q}_{\gk}(t), -[\beta {A}_{\gk}(t)^T+ (1-\beta) {\bar{A}}_{\gk}(t)^T] {q}_{\gk}(t)-[\beta {\mathcal{A}}_{\gk}(t)^{T} +(1-\beta) {\bar{\mathcal{A}}}_{\gk}(t)^{T}]{v}_{\gk}(t) \right\>    \\ 
	& \hspace{-2.5cm}+ & \sum_{i=1}^{r}\gk e^{\gk \hat{\psi}_i(t,x_{\gk}(t))}\left[ \< {q}_{\gk}(t),  
	{{\vartheta}^i_{\gk}(t)} {q}_{\gk}(t) \> +  \underbrace{ \gk  \lvert \< {q}_{\gk}(t), \nabla_x \hat{\psi}_i(t,x_{\gk}(t))\> \rvert{^2}}_{\text{positive term}}\right] \\
	& \hspace{-2.5cm}+ & \underbrace{ \gk e^{\gk \psi_{r+1}(t,x_{\gk}(t))} \| {q}_{\gk}(t)\|^2+ \gk^2 e^{\gk \psi_{r+1}(t,x_{\gk}(t))}  \lvert \< {q}_{\gk}(t), \nabla_x \psi_{r+1}(t,x_{\gk}(t))\>\rvert{^2}}_{\text{positive term}} \\
	&\hspace{-2.5cm}+& \left\< {v}_{\gk}(t), - [\beta {E}_{\gk}(t)^T +(1-\beta) {\bar{E}}_{\gk}(t)^T]{q}_{\gk}(t) - [\beta{\mathcal{E}}_{\gk}(t)^{T} +(1-\beta){\bar{\mathcal{E}}}_{\gk}(t)^{T}] {v}_{\gk}(t) \right\>  \\ &\hspace{-2.5cm} + & \underbrace{ \gk e^{\gk \varphi(t,y_{\gk}(t))} \| {v}_{\gk}(t)\|^2+ \gk^2 e^{\gk \varphi(t,y_{\gk}(t))}  \lvert \< {v}_{\gk}(t), \nabla_y \varphi(t,y_{\gk}(t))\> \rvert{^2}}_{\text{positive term}} \\
	&\hspace{-2.5cm} \ge & \left[- 2L_h(t) -L_{{\psi}} \frac{2{\bar{\mu}}}{ \bar{\eta}^2}\right] \|{p}_{\gk}(t)\|^2:=-{L}_p(t) \|{p}_{\gk}(t)\|^2,
\end{eqnarray*}
where \eqref{boundxi} is employed and ${L}_p(\cdot) \in L^2([0,T],\mathbb{R}_{+})$. 
Using Gronwall's Lemma, we deduce that there exists a constant  ${ M_p>0}$ such that 
\begin{eqnarray}
	\|{p}_{\gk}(t)\| \le e^{\|{L}_p(\cdot)\|_1} \|{p}_{\gk}(T)\| \le M_p, \quad \forall t \in [0,T], \quad \forall k \in \mathbb{N}, \label{bmd}
\end{eqnarray}
where the last inequality is due to the uniform boundedness of $\|{p}_{\gk}(T)\|$ obtained from the nontriviality condition \eqref{m1} when $S$ has a general form,  and to  the transversality condition \eqref{m2}, ${\lambda}_{\gamma_k}=1$, and equation \eqref{normalconeS}, when $S=C_0\times \R^{n+l}$. \\
\\
For the uniform boundedness of  $\{\|\dot{{v}}_{\gk}\|_2\}$,  notice that by \eqref{ref4b}, \eqref{bmd},  \eqref{boundxi}, and  \eqref{gk-xi-gk},  there exist   $L_v(\cdot)\in L^2([0,T],\mathbb{R}_+)$  and $k_v\in \N$, such that  for $k\ge k_v$ we have
\begin{eqnarray*}
	\|\dot{{v}}_{\gk}(t)\| &\le& \|\left[(1-\beta) {\bar{E}}_{\gamma_k}(t)^T + \beta {E}_{\gamma_k}(t)^T\right] {q}_{\gamma_k}(t) + \left[(1-\beta) {\bar{\mathcal{E}}}_{\gamma_k}(t)^T  + \beta {\mathcal{E}}_{\gamma_k}(t)^T\right] {v}_{\gamma_k}(t) \|\\ 
	&\hspace{.3 in}+& \frac{2\bar{\mu}}{\bar{\eta}^2} M_p + \gk^2 e^{-\gk \frac{\bar\delta^2}{4}}  \bar{\delta}^2 M_p
	\le L_v(t) M_p, \quad \forall t \in [0,T].
\end{eqnarray*}
Thus, for all $k\ge k_v$, $(\dot{{v}}_{\gk})$ is uniformly bounded in $L^2$ by a constant $M_v$.\\
\\
We now proceed to prove that $(\dot{{q}}_{\gk})$ is uniformly bounded in $L^1$.  Observe that \eqref{gk-xi-gk} together with \eqref{xkykdomain} and \eqref{bmd}, yields that for some $\bar{k}_1\in\N$, $\bar{k}_1\ge k_v$, we have 
\begin{eqnarray}\label{bds}
	\hspace{-0.5cm}	\|\mathcal{Q}_{\gk}(t) \| \le 2L_h(t) M_p; \quad  \|\mathcal{X}_{\gk}(t) \| \le \frac{2\bar{\mu}}{\bar{\eta}^2} \max\{1,L_{\hat{\psi}}\} M_p; \quad 	\| \mathcal{Z}_{\gk}(t) \| \le \gamma_k^2  \; e^{-\gk \frac{ {\bar{\varepsilon}}^2}{4} }\; \bar{\varepsilon}^2\; M_p.  
\end{eqnarray}
Hence, using \eqref{ref4a} and \eqref{bds}, we can see $\{{q}_{\gk}\}$ is  of uniformly bounded variation once we prove $$	\int_{0}^{T} \|\mathcal{Y}_{\gk}(t)\|dt= \int_0^T \sum_{i=1}^{r}\gamma_k^2 e^{\gk \hat{\psi}_i(t,x_{\gk}(t))} \|\nabla_x \hat{\psi}_i(t,x_{\gk}(t))\| \abs{\<\nabla_x \hat{\psi}_i(t,x_{\gk}(t)), {q}_{\gamma_k}(t) \>} dt $$ is uniformly bounded, for $k$ large. \\
We note that our proof here deviates from the corresponding proofs in \cite{nourzeidan3} and \cite{depinho3}. Denote by \begin{equation}
		\I_k^{\bar{a}}{:}=\I^{\bar{a}}_{(t,x_{\gk}(t))}=\{i\in\{1,\dots, r\}: -\bar{a}\le\psi_i(t,x_{\gk}(t))\le 0\}
	\end{equation} and define  \begin{equation}
		I^{\bar{a}}(x_{\gk}):=\{t\in [0,T]: \I^{\bar{a}}_{(t,x_{\gk}(t))}\neq\emptyset\}.
	\end{equation}
	Using the above definition of $I^{\bar{a}}(x_{\gk})$,  $\I_k^{\bar{a}}$ and $\I_k^{\frac{\bar{a}}{2}}$, and equations \eqref{hatpsi}-\eqref{hatgradpsi}, we deduce that 
	\begin{eqnarray}  
		&& \hspace{-0.3in} \forall t\in [I^{\bar{a}}(x_{\gk})]^c, \;\; \forall i=1,\cdots,r, \;\;  \hat{\psi}_i(t,x_{\gk}(t))  =  - \frac{3\bar{a}}{4}, \quad \nabla_x \hat{\psi}_i(t,x_{\gk}(t))=0,  \label{a11}\\
		&& \hspace{-0.3in}\forall t\in I^{\bar{a}}(x_{\gk}),\;\; \forall i\in { [\I_k^{\bar{a}}]^c} , \;\; \hat{\psi}_i(t,x_{\gk}(t))  = - \frac{3\bar{a}}{4}, \quad \nabla_x \hat{\psi}_i(t,x_{\gk}(t))=0, \label{a111} \\
		&& \hspace{-0.3in} \forall t\in I^{\bar{a}}(x_{\gk}),\; \forall i\in \I_k^{\frac{\bar{a}}{2}}, \;\; \hat{\psi}_i(t,x_{\gk}(t))  = {\psi}_i(t,x_{\gk}(t)), \; \nabla_x \hat{\psi}_i(t,x_{\gk}(t))=\nabla_x {\psi}_i(t,x_{\gk}(t)), 	\label{a21} \\
		&& \hspace{-0.3in}\forall t\in I^{\bar{a}}(x_{\gk}), \;\; \forall i\in \I_k^{\bar{a}} \setminus \I_k^{\frac{\bar{a}}{2}}, \;\; \hat{\psi}_i(t,x_{\gk}(t))  <  -\frac{\bar{a}}{2}, \; \nabla_x \hat{\psi}_i(t,x_{\gk}(t))=s'({\psi}_i(t,x_{\gk}(t)))\nabla_x {\psi}_i(t,x_{\gk}(t)). \nonumber \\ \label{a23}
	\end{eqnarray} 
	As a result of \eqref{a11}-\eqref{a23} and the fact that $0\le s'(z)\le 1$ for all $-\bar{a} \le z \le -\frac{\bar{a}}{2}$, to prove $\int_{0}^{T} \|\mathcal{Y}_{\gk}(t)\|dt$ is uniformly bounded, it remains to prove that, for a certain constant $M_1>0,$\begin{equation}	\textbf{I}_1:= \int_{I^{\bar{a}}(x_{\gk})} \sum_{i\in \I_k^{\frac{\bar{a}}{2}}}\gamma_k^2 e^{\gk {\psi}_i(t,x_{\gk}(t))} \|\nabla_x {\psi}_i(t,x_{\gk}(t))\| \abs{\<\nabla_x {\psi}_i(t,x_{\gk}(t)), {q}_{\gamma_k}(t) \>} dt \le M_1, \label{desired1} \end{equation}
	for $k$ large enough. 
	For that, it is sufficient to prove that there exists $M_2>0$ such that, for $k$ large,  
	\begin{equation} 	\textbf{I}_2:= \int_{I^{\bar{a}}(x_{\gk})}\sum_{i\in \I_k^{\frac{\bar{a}}{2}}}\gamma_k^2 e^{\gk {\psi}_i(t,x_{\gk}(t))} \|\nabla_x {\psi}_i(t,x_{\gk}(t))\|^2 \abs{\<\nabla_x {\psi}_i(t,x_{\gk}(t)), {q}_{\gamma_k}(t) \>} dt \le M_2. \label{desired21} \end{equation}
	Indeed, for  $t\in I^{\bar{a}}(x_{\gk})$ and $i\in \I_k^{\frac{\bar{a}}{2}}$, we have  ${\psi}_{i}(t,x_{\gk}(t))\ge -\frac{\bar{a}}{2} \ge -a_o$, and hence 
	the uniform convergence of $x_{\gk}$ to $\x$ and Remark \ref{pisiepse} yield the existence of $\bar{k}_2\in \N$  such that for all $k\ge\bar{k}_2$, we have $\| \nabla_x {\psi}_i(t,x_{\gk}(t)) \|> \bar{\eta}$.
	Thus, if  $\textbf{I}_2$ is uniformly bounded by a constant $M_2$, then it follows that $\textbf{I}_1 \le \frac{M_2}{\bar{\eta}}$, for $k$ large enough.\\
	We proceed to prove that \eqref{desired21} holds true.
	We first calculate for each {$j=1, \cdots, r$} and $ t\in [0,T]$ a.e.: \begin{eqnarray}
		\frac{d}{dt} \left\lvert \< {q}_{\gamma_k}(t), \nabla_x \hat{\psi}_j(t,x_{\gk}(t)) \>\right \rvert & = &\< \dot{{q}}_{\gamma_k}(t), \nabla_x \hat{\psi}_j(t,x_{\gk}(t)) \> s^j_{\gk}(t) \nonumber\\&+& \< {q}_{\gamma_k}(t), \Theta^j_{\gamma_k}(t). (1,\dot{x}_{\gamma_k}(t)) \>s^j_{\gk}(t),\; \text{ a.e.} \label{derivinner21}
	\end{eqnarray}
	where $s^j_{\gk}(t)$ is the sign of $\< {q}_{\gamma_k}(t), \nabla_x \hat{\psi}_j(t,x_{\gk}(t)) \>$ and ${\Theta}^{j}_{\gamma_k}(t) \in \partial^{(t,x)} \nabla_x \hat{\psi}_j(t,x_{\gamma_k}(t))$.\\ 
	Using equation \eqref{ref4a}  in (\ref{derivinner21}), we get 
	\begin{eqnarray}
		&&\frac{d}{dt} \left\lvert \< {q}_{\gamma_k}(t), \nabla_x \hat{\psi}_j(t,x_{\gk}(t)) \>\right \rvert  = \nonumber\\
		&&\<\mathcal{Q}_{\gk}(t)+ \mathcal{X}_{\gk}(t)+\mathcal{Z}_{\gk}(t), \nabla_x \hat{\psi}_j(t,x_{\gk}(t)) \> s^j_{\gk}(t) +\< {q}_{\gamma_k}(t), \Theta^j_{\gamma_k}(t). (1,\dot{x}_{\gamma_k}(t)) \>s^j_{\gk}(t) \nonumber\\
		&&+\sum_{i=1}^{r}\gamma_k^2 e^{\gk \hat{\psi}_i(t,x_{\gk}(t))} \<\nabla_x \hat{\psi}_i, \nabla_x \hat{\psi}_j\>|_{(t,x_{\gk}(t))}  \langle \nabla_x \hat{\psi}_i(t,x_{\gk}(t)), {q}_{\gk}(t)\rangle s^j_{\gk}(t)  \text{ a.e.}
		\label{derivative21}\end{eqnarray}
	Let $t \in I^{\bar{a}}(x_{\gk})$. Summing the previous equality over $  j \in \I_k^{\bar{a}} $, we obtain that: 
	\begin{eqnarray}
		&&\textbf{J}_1:= \sum_{j \in \I_k^{\bar{a}}}\sum_{i=1}^{r}\gamma_k^2 e^{\gk \hat{\psi}_i(t,x_{\gk}(t))} \<\nabla_x \hat{\psi}_i, \nabla_x \hat{\psi}_j\>|_{(t,x_{\gk}(t))}  \langle \nabla_x \hat{\psi}_i(t,x_{\gk}(t)), {q}_{\gk}(t)\rangle s^j_{\gk}(t)\nonumber \\
		&&=\sum_{j \in \I_k^{\bar{a}}}\frac{d}{dt} \left\lvert \< {q}_{\gamma_k}(t), \nabla_x \hat{\psi}_j(t,x_{\gk}(t)) \>\right \rvert  -\sum_{j \in \I_k^{\bar{a}}}\<\mathcal{Q}_{\gk}(t)+ \mathcal{X}_{\gk}(t)+\mathcal{Z}_{\gk}(t), \nabla_x \hat{\psi}_j(t,x_{\gk}(t)) \> s^j_{\gk}(t) \nonumber\\
		&& -\sum_{j \in \I_k^{\bar{a}}}\< {q}_{\gamma_k}(t), \Theta^j_{\gamma_k}(t). (1,\dot{x}_{\gamma_k}(t)) \>s^j_{\gk}(t)	 \label{bdse1}  \text{ a.e.}
	\end{eqnarray}
	\noindent On the other hand, splitting in the definition of $\textbf{J}_1$ the summation over $i$ and switching the order of summation between $i$ and $j$, we have 
	\begin{eqnarray*}
		&&  \textbf{J}_1 =  \sum_{i=1}^r \sum_{j \in \I_k^{\bar{a}}} \gamma_k^2 e^{\gk \hat{\psi}_i(t,x_{\gk}(t))}  \<\nabla_x \hat{\psi}_i, \nabla_x \hat{\psi}_j\>|_{(t,x_{\gk}(t))} \;\; \langle \nabla_x \hat{\psi}_i(t,x_{\gk}(t)), {q}_{\gk}(t)\rangle s^j_{\gk}(t)  \\ 
	&&\hspace{.1 in}	=   \sum_{i\in\I_k^{\bar{a}}} \sum_{j\in\I_k^{\bar{a}}} \gamma_k^2 e^{\gk \hat{\psi}_i(t,x_{\gk}(t))}  \<\nabla_x \hat{\psi}_i, \nabla_x \hat{\psi}_j\>|_{(t,x_{\gk}(t))} \;\; \langle \nabla_x \hat{\psi}_i(t,x_{\gk}(t)), {q}_{\gk}(t)\rangle s^j_{\gk}(t)  \\ 
		&&\hspace{.1 in}	=   \sum_{i\in\I_k^{\frac{\bar{a}}{2}}} \sum_{j\in\I_k^{\bar{a}}} \gamma_k^2 e^{\gk \hat{\psi}_i(t,x_{\gk}(t))}  \<\nabla_x \hat{\psi}_i, \nabla_x \hat{\psi}_j\>|_{(t,x_{\gk}(t))} \;\; \langle \nabla_x \hat{\psi}_i(t,x_{\gk}(t)), {q}_{\gk}(t)\rangle s^j_{\gk}(t)  \\ 
		&&\hspace{.1 in}	+   \sum_{i\in\I_k^{\bar{a}} \setminus \I_k^{\frac{\bar{a}}{2}}} \sum_{j\in\I_k^{\bar{a}}} \gamma_k^2 e^{\gk \hat{\psi}_i(t,x_{\gk}(t))}  \<\nabla_x \hat{\psi}_i, \nabla_x \hat{\psi}_j\>|_{(t,x_{\gk}(t))} \;\; \langle \nabla_x \hat{\psi}_i(t,x_{\gk}(t)), {q}_{\gk}(t)\rangle s^j_{\gk}(t)  \\ 
		&&\hspace{.1 in}=  \sum_{i\in\I_k^{\frac{\bar{a}}{2}}}  \gamma_k^2 e^{\gk \hat{\psi}_i(t,x_{\gk}(t))} \Biggl(  \| \nabla_x \hat{\psi}_i(t,x_{\gk}(t))\|^2 +  \\ 
		&&\hspace{.2 in}\sum_{\substack{j\in\I_k^{\bar{a}} \\ j \ne i}} s^j_{\gk}(t) s^i_{\gk}(t) \<\nabla_x \hat{\psi}_i(t,x_{\gk}(t)), \nabla_x \hat{\psi}_j(t,x_{\gk}(t))\>        \Biggl) \;\; | \langle \nabla_x \hat{\psi}_i(t,x_{\gk}(t)), {q}_{\gk}(t)\rangle |  \\
		&&\hspace{.1 in}	+   \sum_{i\in\I_k^{\bar{a}} \setminus \I_k^{\frac{\bar{a}}{2}}} \sum_{j\in\I_k^{\bar{a}}} \gamma_k^2 e^{\gk \hat{\psi}_i(t,x_{\gk}(t))}  \<\nabla_x \hat{\psi}_i, \nabla_x \hat{\psi}_j\>|_{(t,x_{\gk}(t))} \;\; \langle \nabla_x \hat{\psi}_i(t,x_{\gk}(t)), {q}_{\gk}(t)\rangle s^j_{\gk}(t).  \\ 
	\end{eqnarray*}
	Using the fact that $x_{\gk}$ converges uniformly to $\bar{x}$, we deduce from equation \eqref{supp3} that,  there exists $\bar{k}_3 \in \mathbb{N} $ such that for $k \ge \bar{k}_3$, we have for $i \in \I_k^{\frac{\bar{a}}{2}}$, 	$$\sum_{\substack{j\in\I_k^{\bar{a}} \\ j \ne i}} s^j_{\gk}(t) s^i_{\gk}(t) \<\nabla_x \hat{\psi}_i(t,x_{\gk}(t)), \nabla_x \hat{\psi}_j(t,x_{\gk}(t))\>  \ge -\bar{b} \| \nabla_x \hat{\psi}_i(t,x_{\gk}(t))\|^2. $$
	Then, \begin{eqnarray*}
		&&\hspace{.1 in} \textbf{J}_1 \ge (1-\bar{b}) \sum_{i\in\I_k^{\frac{\bar{a}}{2}}}  \gamma_k^2 e^{\gk \hat{\psi}_i(t,x_{\gk}(t))} \| \nabla_x \hat{\psi}_i(t,x_{\gk}(t))\|^2 | \langle \nabla_x \hat{\psi}_i(t,x_{\gk}(t)), {q}_{\gk}(t)\rangle | \\
		&&\hspace{.1 in}	+  \underbrace{ \sum_{i\in\I_k^{\bar{a}} \setminus \I_k^{\frac{\bar{a}}{2}}} \sum_{j\in\I_k^{\bar{a}}} \gamma_k^2 e^{\gk \hat{\psi}_i(t,x_{\gk}(t))}  \<\nabla_x \hat{\psi}_i, \nabla_x \hat{\psi}_j\>|_{(t,x_{\gk}(t))} \;\; \langle \nabla_x \hat{\psi}_i(t,x_{\gk}(t)), {q}_{\gk}(t)\rangle s^j_{\gk}(t)}_{\textbf{J}_2}. 
	\end{eqnarray*}
	Hence, (recalling \eqref{a21}) we have 
	\begin{eqnarray}
		\hspace{-.5 in}	\sum_{i\in\I_k^{\frac{\bar{a}}{2}}}  \gamma_k^2 e^{\gk {\psi}_i(t,x_{\gk}(t))} \| \nabla_x {\psi}_i(t,x_{\gk}(t))\|^2 | \langle \nabla_x {\psi}_i(t,x_{\gk}(t)), {q}_{\gk}(t)\rangle | \le  \frac{1}{1-\bar{b}}   (\textbf{J}_1-\textbf{J}_2).  \end{eqnarray}
	Integrating the last inequality over $I^{\bar{a}}(x_{\gk})$, we deduce from the definition of $\textbf{I}_2$ that \begin{eqnarray*}
		&&\hspace{-0.5cm}0\le \textbf{I}_2 \le \frac{1}{1-\bar{b}} \int_{I^{\bar{a}}(x_{\gk})} \left( \textbf{J}_1  - \textbf{J}_2 \right) dt \le  \frac{1}{1-\bar{b}} \left|\int_{I^{\bar{a}}(x_{\gk})} \textbf{J}_1 \;dt\right| +  \frac{1}{1-\bar{b}}\int_{I^{\bar{a}}(x_{\gk})} |\textbf{J}_2|  dt.
	\end{eqnarray*}
	Using \eqref{a23}, we deduce that, there exists $\bar{k}_4 \in \mathbb{N}$,  there exists  constant $M_3>0$ such that for all $k \ge \bar{k}_4$, we have \begin{equation} \int_{I^{\bar{a}}(x_{\gk})} |\textbf{J}_2|  dt \le \frac{ \gamma_k^2 e^{-\gamma_k \frac{\bar{a}}{2}}  L_{\psi}^3 r^2M_p T}{1-\bar{b}} \le M_3. \end{equation}
	Hence, \vspace{-.2 in}\begin{eqnarray}
		&&\hspace{-0.5cm}0\le \textbf{I}_2   \le  \frac{1}{1-\bar{b}} \left|\int_{I^{\bar{a}}(x_{\gk})} \textbf{J}_1 \;dt\right| +  M_3,  \quad \;\forall k\ge \bar{k}_4. \label{I5}
	\end{eqnarray}\vspace{-.1 in}	
	Note that,  $\eqref{bdse1}$ yields that the uniform boundedness of  $\left|\int_{I^{\bar{a}}(x_{\gk})} \textbf{J}_1\; dt\right|$ is equivalent to that of  
	$$  \Big|\int_{I^{\bar{a}}(x_{\gk})}  \sum_{ j \in \I_k^{\bar{a}}}	\frac{d}{dt} \lvert \< {q}_{\gamma_k}(t), \nabla_x \hat{\psi}_j(t,x_{\gk}(t)) \> \rvert \;  dt\Big|, $$
	since \begin{eqnarray*}  &&\hspace{-0.3in} \Big|\int_{I^{\bar{a}}(x_{\gk})}  \sum_{j \in \I_k^{\bar{a}}}\< {q}_{\gamma_k}(t), \Theta^j_{\gamma_k}(t). (1,\dot{x}_{\gamma_k}(t)) \>s^j_{\gk}(t) dt\Big| \le M_p L_{\psi} (1 + M_h + \frac{2 \bar{\mu}}{\bar{\eta}^2} \bar{L}) rT, \\
		&& \hspace{-0.3in}
		\Big|\int_{I^{\bar{a}}(x_{\gk})}  \sum_{j \in \I_k^{\bar{a}}} \< \mathcal{Q}_{\gk}(t)+\mathcal{X}_{\gk}(t)+\mathcal{Z}_{\gk}(t), \nabla_x\hat{\psi}_j(t,x_{\gk}(t)) \> s_{\gk}^j(t) dt\Big|  \\ 
		&& \hspace{1.5in} \le \left[(2L_h(t)M_p) + (\frac{2\bar{\mu}}{\bar{\eta}^2}\max\{L_{\psi},1\} M_p) + (\gk^2 e^{-\gk \frac{\bar{\varepsilon}^2}{4}} \bar{\varepsilon}^2 M_p)   \right] L_{\psi} r T. 
	\end{eqnarray*}
	We now proceed to prove the boundedness of $$  \Big|\int_{I^{\bar{a}}(x_{\gk})}  \sum_{ j \in \I_k^{\bar{a}}}	\frac{d}{dt} \lvert \< {q}_{\gamma_k}(t), \nabla_x \hat{\psi}_j(t,x_{\gk}(t)) \> \rvert \;  dt\Big|. $$
	Using the Fundamental Theorem of Calculus, we have that \begin{equation}
		\Big|\int_{0}^T  \sum_{j=1}^r	\frac{d}{dt} \lvert \< {q}_{\gamma_k}(t), \nabla_x \hat{\psi}_j(t,x_{\gk}(t)) \> \rvert \;  dt\Big| \le 2 r L_{\psi}M_p.\label{bdse21}
	\end{equation}
	Using \eqref{a11}, \eqref{a111}, \eqref{derivative21}, and the uniform boundedness of $(\dot{x}_{\gk})$, we deduce that that there exists a constant $M_4>0$ such that   \begin{eqnarray}
		\Big|\int_{ [I^{\bar{a}}(x_{\gk})]^c}  \sum_{j=1}^r	\frac{d}{dt} \lvert \< {q}_{\gamma_k}(t), \nabla_x \hat{\psi}_j(t,x_{\gk}(t)) \> \rvert \;  dt\Big| \le M_4,  \label{bdse22}  \\
		\Big|\int_{I^{\bar{a}}(x_{\gk})}  \sum_{j \in [\I_k^{\bar{a}}]^c} \frac{d}{dt} \lvert \< {q}_{\gamma_k}(t), \nabla_x \hat{\psi}_j(t,x_{\gk}(t)) \> \rvert \;  dt\Big| \le M_4. \label{bdse23}
	\end{eqnarray} 
	Hence, combining \eqref{bdse21} and \eqref{bdse22}, we conclude that there exists a constant $M_5>0$ such that $$\Big|\int_{I^{\bar{a}}(x_{\gk})}  \sum_{ j =1}^r	\frac{d}{dt} \lvert \< {q}_{\gamma_k}(t), \nabla_x \hat{\psi}_j(t,x_{\gk}(t)) \> \rvert \;  dt\Big| \le M_5.$$
	This last inequality with \eqref{bdse23} yield that there exists a constant $M_6>0$ such that
	$$\Big|\int_{I^{\bar{a}}(x_{\gk})}  \sum_{ j \in \I_k^{\bar{a}}}	\frac{d}{dt} \lvert \< {q}_{\gamma_k}(t), \nabla_x \hat{\psi}_j(t,x_{\gk}(t)) \> \rvert \;  dt\Big| \le M_6.$$  	Hence, $\left|\int_{I^{\bar{a}}(\x)} \textbf{J}_1\; dt\right|$  is uniformly bounded, and by \eqref{I5}, $\textbf{I}_2$ is uniformly bounded. Hence, $\{\| \dot{{q}}_{\gk} \|_1\}$ uniformly bounded by a constant $M_q$.\\
	\\
\textbf{Step 3.} Construction of $p=(q,v)$, {$\lambda \ge 0$}, $\vartheta^i$ (for each $i$), $\nu^i$ (for each $i$ ), $\bar{A}$, $\bar{\mathcal{A}}$, $\bar{E}$ and $\bar{\mathcal{E}}$ satisfying conditions $(i)$-$(vii)$ of Theorem \ref{maxprincipleS}.  \\
In Step 1, we  proved the existence of $({\xi}^i)_{i=1}^r$ in $L^{\infty}([0,T]; \mathbb{R}_+)$ such that condition $(i)$ is satisfied in terms of the $\hat\psi_i$'s. From Step 2, we conclude that along a subsequence,  ${q}_{\gk} \xrightarrow[n \to \infty]{a.e.} {q}\in BV([0,T];\R^n)$ and ${v}_{\gk} \xrightarrow[n \to \infty]{unif}  {v}\in W^{1,2}([0,T];\R^l)$. Following steps similar to steps 3-10 in the proof of \cite[Theorem 6.1]{nourzeidan}, we deduce the existence of ${\lambda}$,  ${\bar{A}}$, ${A}$, ${\bar{\mathcal{A}}}$, ${\mathcal{A}}$, ${\bar{E}}$, ${\bar{\mathcal{E}}}$, ${E}$, ${\mathcal{E}}$, ${\vartheta}^i$ (for $i=1,\cdots, r$) and ${\nu}^i$ (for $i=1,\cdots,r$) such that all the conditions of Theorem \ref{maxprincipleS} are satisfied in terms of $\hat\psi_i$'s, except condition $(vi)$, which will be proved at the end of this step. Note that  the  terms 
in \eqref{ref4a} and \eqref{ref4b}    involving $\psi_{r+1}$ and $\varphi$  actually vanish upon taking the limit. This is due to \eqref{xi-zetaconv} in which we have  $(\gk\xi_{\gk}^{r+1}, \gk \zeta_{\gk}) \longrightarrow 0$ {\it uniformly}. More specifically we have, for all $(z(\cdot),w(\cdot)) \in \mathcal{C}([0,T];\mathbb{R}^n) \times \mathcal{C}([0,T];\mathbb{R}^l)$, we have
\begin{eqnarray} &&\hspace{-.5 in}\lim\limits_{k \to \infty}\int_{0}^T \<z(t),\gk e^{\gk \psi_{r+1}(t,x_{\gk}(t))} {q}_{\gamma_k}(t) \> dt =0, \;\;\;\; \lim\limits_{k \to \infty}\int_{0}^T \<w(t),\gk e^{\gk \varphi(t,y_{\gk}(t))}  {v}_{\gamma_k}(t) \> dt =0,\label{eqa} \\
&&\hspace{-.5 in}\lim\limits_{k \to \infty} \int_{0}^T \gamma_k^2 e^{\gk \psi_{r+1}(t,x_{\gk}(t))} \<\nabla_x \psi_{r+1}(t,x_{\gk}(t)), {q}_{\gamma_k}(t) \>  \<z(t), \nabla_x \psi_{r+1}(t,x_{\gk}(t))\> dt=0, \;\;\text{and,}\\
	&&\hspace{-.5 in}\lim\limits_{k \to \infty} \int_{0}^T \gamma_k^2 e^{\gk \varphi(t,y_{\gk}(t))} \<\nabla_y \varphi(t,y_{\gk}(t)), {v}_{\gamma_k}(t) \>  \<w(t),	 \nabla_y \varphi(t,y_{\gk}(t))\> dt=0. \;\;  \label{eqb} \end{eqnarray}
\\
  Observe that, similarly  to the proof of step 4 in \cite[Theorem 5.1]{verachadihassan}, the finite signed Radon measure on $[0,T]$, ${\nu}^i_{\gk}$, corresponding to each $\hat{\psi}_i$, is defined by 
\begin{equation} \label{dnuigk} d{\nu}^i_{\gk}(t):= \gk {\xi}^i_{\gk}(t)\<\nabla_x\hat{\psi}_i(t,x_{\gk}(t)),{q}_{\gk}(t)\>  dt,\; i=1,\cdots, r,\end{equation}
and the signed Radon measure ${\nu}^i$ in condition $(iii)$ of Theorem \ref{maxprincipleS} is the weak* limit of the corresponding ${\nu}_{\gk}^i$. 
\\
 For condition $(vi)$, equation \eqref{dnuigk} yields the following \begin{equation*}
\left\<{q}_{\gk}(t), \nabla_x\hat{\psi}_i(t,x_{\gk}(t)) \right\>	d{\nu}^i_{\gk}(t) =\gk {\xi}^i_{\gk}(t)\<\nabla_x\hat{\psi}_i(t,x_{\gk}(t)),{q}_{\gk}(t)\>^2\ge 0,
	\end{equation*}
and hence, upon taking the limit, we get the desired condition. Thus, Theorem \ref{maxprincipleS} is valid in terms of the $\hat\psi_i$'s.\\
Notice that ${\nu}^i$ and ${\xi}^i$ are supported in $\{t \in [0,T]: \hat{\psi}_i(t,\x(t))=0 \}\overset{(\ref{hatpsi})}=\{t \in [0,T]: {\psi}_i(t,\x(t))=0 \}= I^0_i(\bar{x})$, and on this set,  $\nabla_x  \hat{\psi}_i(t,\x(t))\overset{(\ref{hatgradpsi})}= \nabla_x\psi_i(t,\x(t))$. 
Therefore, we conclude that  Theorem \ref{maxprincipleS}  is proved  in terms of $\psi_i$, under the temporary assumption (A4.2). \\
\\
\textbf{Step 4.} Removing assumption (A4.2) when  the sets $U(t)$ are  uniformly bounded. \\
Consider $h$  not  satisfying $(A4.2)$. To remove (A4.2), that is, the convexity assumption  of $h(t,x,y,U(t))$   for $(x,y)\in \bar{\mathscr{N}}_{(\bar{\delta},\bar\delta)}(t)$ and  $t\in [0, T]$ a.e., we shall extend  the {\it relaxation technique}  in \cite[Section 5.2]{verachadihassan},  developed for {\it global} minimizers of Mayer optimal control problems over sweeping processes having {\it constant} compact sweeping sets and constant control set $U,$ to the case of {\it strong local} minimizers,  the sweeping sets being $\bar{\mathscr{N}}_{(\bar\varepsilon,\bar\delta)}(t)$,  which are {\it time-dependent} and {\it not} necessarily moving in an  absolutely continuous way,  $U$ being time-dependent, and  the {\it joint}-endpoints constraint  $S_{\frac{\delta}{2}}$ being present, where $\delta \in (0,\bar{\varepsilon})$ is fixed.\\
\textit{\underline{Step $4(a)$.} $(\bar{X}:=(\x,\y),\bar{u})$ is a $\delta$-strong local minimizer for $(\bar{P}_{\delta})$ with extended $J$}\\  
Fix $\delta\in (0,\bar\varepsilon)$. Using \cite[Theorem 1]{hiriart}, there is an $L_J$-Lipschitz function $\bar{J}: \R^{n+l}\times \R^{n+l} \rightarrow \R$ that extends $J$ to $\R^{2(n+l)}$ from $S(\bar\delta)$.  By Remark \ref{optimality}$(i)$,  $((\x,\y),\bar{u})$ being a $\bar{\delta}$-strong local minimizer  for $(P)$, then it is  also a $\delta$-{\textit strong} local minimum  for $ (\bar{P}_{ \delta})$ in which we use the extension $\bar{J}$ instead of $J$. \\  
  \\
\textit{\underline{Step $4(b)$.} $(\bar{X},\bar{u})$ is a global minimum for $(\mathscr{\bar{P}})$}\\
 Performing  appropriate modifications to the technique presented in the  proof of  \cite[Theorem 6.2]{nourzeidan},  we are then  able   to formulate the following   problem $(\mathscr{\bar{P}})$ associated with $(\bar{P}_{\delta})$ for which the same solution $(\bar{X},\bar{u})$ is a {\it global} minimum:    
 \begin{equation*}
	(\mathscr{\bar{P}}) \begin{cases}
		\text{minimize} \quad   {\bar{J}}(X(0),X(T)) +\bar{K} \int_0^T {\mathcal{L}(t,X(t))}\; dt\\
		\text{over } X:=(x,y) \in W^{1,1}( [0,T], \mathbb{R}^{n+l}),\; u\in \mathcal{U}, \;\textnormal{ such that } \\
		\hspace{.2 in}(\bar{D})\begin{cases}
			\dot{X}(t) \in h(t,x(t),y(t), u(t)) - N_{\bar{\mathscr{N}}_{(\bar\varepsilon,\bar\delta)}(t)}(X(t)), \text{ a.e. } t \in [0,T], \\
		\end{cases} \\
		(X(0),X(T)) \in {S}_{\frac{\delta}{2}}= S\cap \bar{\mathscr{B}}_{\frac{\delta}{2}},
	\end{cases}
\end{equation*} where   $\mathcal{L}:[0,T] \times \R^{n+l} \rightarrow \R$ and {$\bar{K}>0$} are defined by
\begin{eqnarray}
\hspace{-.5 in}&&\mathcal{L}(t,X)=\mathcal{L}(t,x,y):= \max\{\|x-\x(t)\|^2 -\frac{\delta^2}{4},\; \|y-\y(t)\|^2 -\frac{\delta^2}{4},\; 0\} \ge 0, \;
\label{L}\\ 
\hspace{-.5 in}&& \bar{K}:=\frac{512 \bar{M}_{\ell} M_J}{5 \delta^3}, \text{where} \; \; 2\bar{M}_{\ell}:= \max\{L_{(\x,\y)},\; M_h+\frac{\bar{\mu}}{4\bar{\eta}^2}\bar{L}\},\;\; M_J:=\max_{ {S}(\bar{\delta})} |J(X_1,X_2)|,\label{barK}\vspace{-.1 in}\end{eqnarray}
and hence, as $\mathcal{L}(t,\bar{X}(t)) \equiv 0$, we deduce ${\min(\mathscr{\bar{P}})= J(\bar{X}(0),\bar{X}(T)).}$\\ 
\\
\textit{\underline{Step $4(c)$.} $({\bar{X}},\bar{w}):=(\bar{X},((\overbrace{\u,...,\u}^{n+l+1}),(\overbrace{1,0,...,0}^{n+l+1})))$ is a global minimum for $	({\mathscr{\tilde{P}}}) $}\\
Define the problem $({\mathscr{\tilde{P}}})$
\begin{equation*}
	({\mathscr{\tilde{P}}}) \begin{cases}
		\text{minimize} \quad   {\bar{J}}(X(0),X(T)) +\bar{K} \int_0^T\mathcal{L}(t,X(t))\; dt\\
		\text{over } X:=(x,y) \in W^{1,1}( [0,T], \mathbb{R}^{n+l}),\;\\\;\;\;\;\;\; w(\cdot):= \big((u_0(\cdot),\cdots,u_{n+l}(\cdot)), (\lambda_0(\cdot),\cdots,\lambda_{n+l}(\cdot))\big) \in \mathscr{W}\textnormal{ such that } \\
		\hspace{.2 in}\tilde{({\mathscr{D}})}\begin{cases}
			\dot{X}(t) \in \tilde{h}(t,X(t), w(t)) - N_{\bar{\mathscr{N}}_{(\bar\varepsilon,\bar\delta)}(t)}(X(t)), \text{ a.e. } t \in [0,T], \\
		\end{cases} \\
		(X(0),X(T)) \in {S}_{\frac{\delta}{2}},
	\end{cases}\vspace{-.13 in}\end{equation*}\vspace{-.13 in}
where
\begin{eqnarray}
	&&\hspace{-.5 in} \tilde{h}: \text{Gr}\,[{\bar{\mathscr{N}}_{(\bar\delta,\bar\delta)}}(\cdot) \times (U(\cdot))^{n+l+1}]\times \Lambda \rightarrow \R^{n+l} \;\text{defined as}\;\;  \tilde{h}(t, X,w):=  \sum_{i=0}^{n+l}\lambda_i h(t, X,u_i)\label{tilde-h},\\
	&&\hspace{-.5in}\displaystyle\Lambda:=\big\{(\lambda_0,\cdots,\lambda_{n+l})\in \R^{n+l+1} : \l_i\geq 0\;\,\hbox{for}\;\,i=0,...,n+l\;\,\hbox{and}\;\,\sum_{i=0}^{n+l}\l_i=1\big\},\nonumber\\
	&&\hspace{-.5 in}\mathscr{W}:=\left\{ w\colon[0,T]\f\R^{(m+1)(n+l+1)}\;\,\hbox{measurable}\;: \;w(t) \in W(t):= (U(t))^{n+l+1}\times \Lambda \;\hbox{a.e.}\right\} \nonumber
\end{eqnarray}
First, we note the following two facts that are going to be useful for our goal: 
\begin{itemize}
	\item   Corollary \ref{corollary2} yields that for $X_0=(x_0,y_0) \in {\bar{\mathscr{N}}}_{(\bar\varepsilon,\bar\delta)}(0)$, and for $w\in \mathcal{W}$,  $(\tilde{\mathscr{D}})$ has a {\it unique} solution $X(\cdot)$  corresponding to $(X_0,w)$ which is $(M_h+\frac{\bar{\mu}}{4\bar{\eta}^2}\bar{L})$-Lipschitz and satisfies \eqref{lipschitz-xy1}-\eqref{dotx-f}.
	\item   Using that $0<\delta<\bar\varepsilon<\bar\delta$ and that the function $\bar{X}(\cdot):=(\bar{x}(\cdot),\bar{y}(\cdot))$ is $L_{(\x,\y)}$-Lipschitz, then the function $\mathcal{L}$, defined in \eqref{L}, is Lipschitz on Gr ${\bar{\mathscr{N}}_{(\bar\delta,\bar\delta)}}(\cdot)$ and satisfies \vspace{-.07 in}
	\begin{equation} \mathcal{L}\equiv 0 \;\;\;\text{on \; Gr}\; \bar{\mathscr{{N}}}_{(\frac{\delta}{2},\frac{\delta}{2})}(\cdot), \;\label{propL} \;\; \;\text{and}\;\;\;0\le \mathcal{L}(t, X) \le \ {\bar\delta}^2\, \; \text{on} \;   \text{Gr}\; {\bar{\mathscr{N}}_{(\bar\delta,\bar\delta)}}(\cdot).\end{equation}
	Hence, by  the convexity of $\tilde{h}(t,X,W(t))$,  and by Remark \ref{existenceL}, where $\mathbb{L}:=\bar{K}\;\mathcal{L}$, it follows that $ ({\mathscr{\tilde{P}}}) \text{ admits a }{\it global \;optimal \; minimizer\;} (\tilde{X},\tilde{w}).$ 
\end{itemize}
We  show that $\min (\tilde{\mathscr{P}})=\min (\mathscr{\bar{P}})$  and  $(\bar{X}, \bar{w})$ is optimal for $(\tilde{\mathscr{P}}).$ For, let  $\mathbb{U}$ the set in \eqref{mathbbU}, the compact set  $\mathbb{V}:= {\text{cl }} \mathbb{U}$, and $\mathscr{R}:=\left\{\sigma\colon[0,T]\rightarrow \mathfrak{M}^1_+(\mathbb{V})\;: \;\sigma\;\hbox{is measurable and } {\sigma(t)(U(t))=1},  \;\; t\in [0,T]\right\}$.  This set of  relaxed controls satisfies $\mathscr{R}\subset  L^1\left([0,T], \mathcal{C}(\mathbb{V}; \R)\right)^*$, which  is endowed with the weak* topology.  Each regular control  function $u\in\mathscr{U}$  is identified with its associated Dirac relaxed control $\sigma(\cdot)=\delta_{u(\cdot)}$, and thereby  $\mathscr{U} \subset \mathscr{R}$ (see e.g.,  \cite{warga}).  Define $h_\sigma (t,X)$ and the problem $({\mathscr{P}})_r$  by
$$h_\sigma (t,X):= \int_{U(t)} h(t,X,u)\sigma(t)(du),\;\;\;\forall (t,X)\in \text{Gr}\,{\bar{\mathscr{N}}_{(\bar\delta,\bar\delta)}}(\cdot),  \;\;\; \sigma\in \mathscr{R}, $$  
\begin{equation*}
	({\mathscr{P}})_r \begin{cases}
		\text{minimize} \quad   {\bar{J}}(X(0),X(T)) +\bar{K} \int_0^T\mathcal{L}(t,X(t))\; dt\\
		\text{over } X:=(x,y) \in W^{1,1}( [0,T], \mathbb{R}^{n+l}),\; \sigma\in \mathscr{R}, \textnormal{ such that } \\
		\hspace{.2 in}(\mathscr{D})_r \begin{cases}\dot{X}(t)\in  h_{\sigma}(t,X(t)) -N_{\bar{\mathscr{N}}_{(\bar\varepsilon,\bar\delta)}(t)}(X(t)),\;\;\hbox{a.e.}\;t\in[0,T],\end{cases}\\
	(X(0),X(T)) \in {S}_{\frac{\delta}{2}}.
	\end{cases}			 			
\end{equation*}
Since $h$ satisfies $(A4.1)$ and $\sigma(t) (U(t))=1 \,(\forall t\in[0,T])$, then ${h}_\sigma(t,X)$ is {\it uniformly} bounded by $M_h$,  a Carath\'eodory function in $(t, X)$, and $L_{h}(t)$-{\it Lipschitz} in $X,$ for all $t$, that is, $h_{\sigma}(t,X)$ satisfies $(A4)$. \\
\\
Using Corollary \ref{corollary2}  for $X_0=(x_0,y_0) \in {\bar{\mathscr{N}}}_{(\bar\varepsilon,\bar\delta)}(0)$,  $\sigma\in \mathcal{R}$,  and $(f,g)(t,X, u)= h(t,X, u):= h_{\sigma}(t,X)$, the Cauchy problem of $(\mathscr{D})_r$ corresponding to $(X_0,\sigma)$ admits a unique solution which is Lipschitz and satisfies \eqref{lipschitz-xy1}-\eqref{dotx-f}. {It follows that the results in \cite[Lemmas 5.1 \&5.2]{verachadihassan} remain valid for  the systems $(\tilde{\mathscr{D}})$ and $(\mathscr{D})_r$, defined above, and also for the system $(\mathscr{D})_c$}, where \vspace{-.1 in}
$$(\mathscr{D})_c \begin{cases}\dot{X}(t)\in \conv  h(t,X(t), U(t)) -N_{\bar{\mathscr{N}}_{(\bar\varepsilon,\bar\delta)}(t)}(X(t)),\;\;\hbox{a.e.}\;t\in[0,T].\;\; \end{cases} \vspace{-.1 in}$$
Therefore, for  $X:=(x,y) \in W^{1,1}( [0,T], \mathbb{R}^{n+l})$ and $X(0)\in {\bar{\mathscr{N}}}_{(\bar\varepsilon,\bar\delta)}(0)$, we have 
\begin{eqnarray*}
(X,w) \text{ satisfies } (\tilde{\mathscr{D}}), \text{ for some } w \in \mathscr{W} &&\iff (X,\sigma) \text{ satisfies }(\mathscr{D})_r \text{ for some } \sigma\in \mathscr{R} \\
	&& \iff X  \text{ satisfies } (\mathscr{D})_c.
\end{eqnarray*}	
{Furthermore, due to having \eqref{dotx-f} satisfied by the solutions of $(\mathscr{D})_r$ and due to the hypomonotonicity property of the uniform prox-regular sets ${\bar{\mathscr{N}}}_{(\bar\varepsilon,\bar\delta)}(t)$ (which we recall it to be the product of  the uniform 
	$\frac{2\bar\eta}{L_{\psi}}$-prox-regular set $C(t) \cap \bar{B}_{\bar\varepsilon}(\x(t))$ with $\bar{B}_{\bar\delta}(\y(t))$), it follows that \cite[Theorem 2]{edmond} (also \cite[Proposition 3.5]{cst}) is valid. Hence, using that $(\tilde{X},\tilde{w})$ is optimal for $(\tilde{\mathscr{P}})$, the proof of \cite[Proposition 5.2]{verachadihassan} holds true for our setting, and therefore, as $(\bar{X},\bar{u})$ is optimal for $(\mathscr{\bar{P}})$,  we conclude that
	$$  \min(\mathscr{{P}})_r=\min (\mathscr{\tilde{P}})= \min (\mathscr{\bar{P}})= J(\bar{X}(0),\bar{X}(T)). $$
Now, since   $({\bar{X}},\bar{w})$  is  admissible for $(\mathscr{\tilde{P}})$ at which the objective value is $J(\bar{X}(0),\bar{X}(T))$,  we deduce that $(\bar{X},\bar{w})$ is a {\it global} minimum  for $(\mathscr{\tilde{P}})$. This terminates proving Step 4$(c)$. \\
\\
\textit{\underline{Step $4(d)$.} $((\x,\y),\bar{w})$ a $\frac{\delta}{2}$-strong local minimum for $	({{\tilde{P}}}) $ to which we apply Theorem \ref{maxprincipleS}}\\
As $(\bar{X}, \bar{w})$ is a global minimizer for  $(\mathscr{\tilde{P}})$, it follows that it is also a  $\frac{\delta}{2}$-{\it strong} local minimum  for $(\mathscr{\tilde{P}})$}, which, by \eqref{propL}$(i)$, has now  $\bar{J}(X(0),X(T))$ as objective function. Hence, we conclude that  $((\x,\y), \bar{w})$ is a $\frac{\delta}{2}$-strong local minimum  for the problem $	({{\tilde{P}}})$ 
\begin{equation*}
	({{\tilde{P}}}) \begin{cases}
		\text{minimize} \quad   {\bar{J}}(x(0),y(0), x(T), y(T)) \\
		\text{over } X:=(x,y) \in W^{1,1}( [0,T], \mathbb{R}^{n+l}),\;\\\;\;\;\;\;\; w(\cdot):= \big((u_0(\cdot),\cdots,u_{n+l}(\cdot)), (\lambda_0(\cdot),\cdots,\lambda_{n+l}(\cdot))\big) \in \mathscr{W}\textnormal{ such that } \\
		\hspace{.2 in}\tilde{({{D}})}\begin{cases}
			\dot{x}(t) \in \tilde{f}(t,x(t), y(t), w(t)) - N_{C(t)}(x(t)), \text{ a.e. } t \in [0,T], \\
			\dot{y}(t) \in \tilde{g}(t, x(t), y(t), w(t)), \text{ a.e. } t \in [0,T], 			
		\end{cases} \\
		(x(0),y(0),x(T),y(T)) \in {S}_{\frac{\delta}{2}},
	\end{cases}
\end{equation*}\vspace{-.15 in}
where $(\tilde{f},\tilde{g})=\tilde{h}$ defined in \eqref{tilde-h}, that is, 
$$\tilde{f}(t,x,y,w):=  \sum_{i=0}^{n+l}\lambda_i f(t, x,y,u_i),\; \;\text{and}\; \;
\tilde{g}(t, x,y,w):=  \sum_{i=0}^{n+l}\lambda_i g(t, x,y,u_i).$$
Clearly  $({\tilde{P}})$ is of the form of $(P)$, where  $f(t,x,y,w):= \tilde{f}(t,x,y,w)$, $g(t,x,y,w): = \tilde{g}(t,x,y,w)$,  $S:={S}_{\frac{\delta}{2}}$,  $U(t):=W(t)$, and $J:= \bar{J}$. Furthermore, the  associated $\tilde{h}(t,x,y, w) =(\tilde{f},\tilde{g})(t,x,y,w)$ satisfies that $\tilde{h}(t,x,y,W(t))$  {\it convex} for each $(t,x,y)\in \text{Gr}\,{\bar{\mathscr{N}}}_{(\bar\delta,\bar\delta)}(\cdot)$. Thus, assumptions (A1)-(A5) hold at the strong local minimizer $((\x,\y),\bar{w})$ for $(\tilde{P})$ to which  the already proven $(i)$-$(vii)$ of Theorem \ref{pmp-ps} apply. Doing so, and noticing these facts: 
\begin{itemize}
	\item $\bar{J}=J$ on $ S(\bar{\delta})$, and hence, $\partial_{\ell}^L\bar{J}(\x(0),\y(0),\x(T),\y(T))=\partial_{\ell}^L{J}(\x(0),\y(0),\x(T),\y(T))$,
	\item $\tilde{h}(t, \x(t),\y(t), \bar{w}(t))=h(t,\x(t),\y(t),\bar{u}(t))$,
	\item $\partial^{ (x,y)}_\ell \tilde{f}(t,\x(t),\y(t),\bar{w}(t))\subset \partial^{ (x,y)}_\ell f(t,\x(t),\y(t),\bar{u}(t))$, 
	\item  $\partial^{ (x,y)}_\ell \tilde{g}(t,\x(t),\y(t),\bar{w}(t))\subset \partial^{ (x,y)}_\ell g(t,\x(t),\y(t),\bar{u}(t)),$
	\item $\<\tilde{h}(t,\x(t),\y(t),w), p(t)\>=\<h(t,\x(t),\y(t), u), p(t)\>,$  \; $\forall w=((u,...,u),(1,0,...,0))\in U^{n+l+1}\times \Lambda$,  
	\item $N^L_{{S}_{\frac{\delta}{2}}}(\x(0),\y(0))=N^L_{S}(\x(0),\y(0))$,  
\end{itemize}
 we conclude  that Theorem \ref{pmp-ps} holds for $(P)$ without assumption $(A4.2)$.\\
\\
\textbf{Step 5.} Proof of  the ``In addition'' part of the theorem. \\
When $S= C_0\times \mathbb{R}^{n+l}$, for $C_0\subset C(0)\times \mathbb{R}^l $  closed,  Proposition \ref{maxapp}  yields that  $\lambda =1$. \\
This completes the proof of  the theorem.
\end{proof} 

\section{Example}\label{examplesec}
This example illustrates how one can use Theorem \ref{maxprincipleS}  to find an optimal solution for a numerical problem.  
  
\begin{example}\label{example}
Consider the problem $(P)$ with the following data $($see Figure \ref{Fig1}$)$. 	
\begin{itemize}
	\item The perturbation mappings $f\colon [0,\frac{\pi}{2}] \times \R^3\times\R \times \mathbb{R}\f\R^3$ and $g\colon [0,\frac{\pi}{2}] \times \R^3\times\R \times \mathbb{R}\f\R$  are defined by \begin{eqnarray*}
	&&	f\left(t,(x_1,x_2,x_3),y,u\right)=(x_1-x_2-u+y^2,x_1+x_2+u+y^3,x_3+t-\pi-1);\\
	&& g\left(t,(x_1,x_2,x_3),y,u\right) = x_1^2+x_2^2-16+u+y.
	\end{eqnarray*}
	\item The two functions $\psi_1,\;\psi_2\colon [0,\frac{\pi}{2}]\times \R^3\f\R$ are defined by \begin{eqnarray*}
	\psi_1(t,x_1,x_2,x_3):=x_1^2+x_2^2+\frac{32}{\pi}x_3 + \frac{32}{\pi}t-48, \;\; 	\psi_2(t,x_1,x_2,x_3):=x_1^2+x_2^2-\frac{32}{\pi}x_3-\frac{32}{\pi}t+16,
	\end{eqnarray*} 
	and hence, for each $t \in [0, \frac{\pi}{2}]$, the set $C(t)$ is the {\it nonsmooth}, convex and {\it  bounded} set \begin{eqnarray*}
		C(t)&=&C_1(t)\cap C_2(t) \\
		&:=&\{(x_1,x_2,x_3) : \psi_1(t,x_1,x_2,x_3)\leq 0\}\cap \{(x_1,x_2,x_3) : \psi_2(t,x_1,x_2,x_3)\leq 0\}.
	\end{eqnarray*}
\item The objective function ${J}\colon\R^8\f\R\cup\{\infty\}$  is defined  by $$J(x_1,x_2,x_3,y_1,x_4,x_5,x_6,y_2):=\begin{cases} -x_4^2-x_5^2+16+ \abs{\frac{\pi}{2}-x_6}& \;\;(x_4,x_5,x_6)\in C(\frac{\pi}{2}), \vspace{0.1cm}\\ \infty&\;\;\hbox{Otherwise}. \end{cases}$$
\item \sloppy The control multifunction is the constant $U(t):=[0,1]$ for all $t\in [0,\frac{\pi}{2}]$. 
\item The set $S$ is given by  \begin{eqnarray*}S:=\{(x_1,x_2,x_3,y_1,x_4,x_5,x_6,y_2) \in\R^8: x_1^2+x_2^2=16, x_3=\pi, x_2+x_6^2=\frac{\pi^2}{4}, \\ \frac{x_1^2}{8}+x_4=2, y_1+x_2^2=0\}.\end{eqnarray*}
\end{itemize}	
\begin{figure}[tb]
	\centering
\end{figure}  
 Define, for each $ t\in [0, \frac{\pi}{2}]$, the curve
$$\Gamma(t):=\left\{(x_1,x_2,x_3) : x_1^2+x_2^2=16\;\hbox{and}\;x_3=\pi-t\right\}= \left(\bdry C_1(t)\cap \bdry C_2(t)\right)\subset\bdry C(t).$$ 
Since $S\subset \Gamma(0) \times \mathbb{R} \times \mathbb{R}^3 \times \mathbb{R}$ and  ${J}$ vanishes on $\mathbb{R}^3 \times \mathbb{R} \times \Gamma(\frac{\pi}{2}) \times \mathbb{R}$   and is strictly positive elsewhere in $\mathbb{R}^3 \times \mathbb{R} \times C(\frac{\pi}{2}) \times \mathbb{R}$, we may seek  for $(P)$  a candidate $((\x,\y),\u)$ for  optimality  with  $\x(t):=(\x_1(t),\x_2(t),\x_3(t))$ belonging to $\Gamma(t)$ for every $t$, if possible,  and hence we have \begin{equation} \label{ex1}\begin{cases} \x_1^2(t)+\x_2^2(t)=16\;\,\hbox{and}\,\;\x_3(t)=\pi-t\;\forall t\in [0,\tfrac{\pi}{2}]\,\;\hbox{and}\\[1pt]
		\x_1(t)\dot{\x}_1(t)+	\x_2(t)\dot{\x}_2(t)=0 \text{ a.e. }\,\;\hbox{and}\\[1pt]
		 (\x(0)\tran,\y(0)\tran,\x(\tfrac{\pi}{2})\tran,\y(\tfrac{\pi}{2})\tran)\in \{ (4,0,\pi,0,0,4,\frac{\pi}{2},a), (-4,0,\pi,0,0,4,\frac{\pi}{2},b),\\[1pt] \; \;\;\;\;\;\;\;\;\;\;\;\;\;\;\;\; \;(4,0,\pi,0,0,-4,\frac{\pi}{2},c), (-4,0,\pi,0,0,-4,\frac{\pi}{2},d) ; \; a,b,c,d \in \mathbb{R} \}.
\end{cases} \end{equation}
One can readily verify that all assumptions of Theorem \ref{maxprincipleS} are satisfied for any choice of $(\bar{x},\bar{y})$ such that $\bar{x}(t) \in \Gamma(t)$ for all $t$, with \textbf{$(A3.3)$} being satisfied for $\bar{\beta}=(1, 1)$. Applying  Theorem  \ref{maxprincipleS}\footnote{Note that for $(x_1,x_2,x_3)\in\Gamma(t)$ with $-\tfrac{\sqrt{3}}{2} < x_1< \tfrac{\sqrt{3}}{2}$, we have $\<\nabla\psi_1(x_1,x_2,x_3),\nabla\psi_2(x_1,x_2,x_3)\>=4x_1^2-3<0,$ and hence, the maximum principle of \cite{depinho3} cannot be applied to this sweeping set $C(t)$.}  to such candidate $((\x,\y),\u)$ we obtain the existence of an adjoint vector $p=(q,v)$ where $q:=(q_1,q_2,q_3)\in BV([0,\frac{\pi}{2}];\R^3)$, $v \in W^{1,1}([0,\frac{\pi}{2}];\R)$, two finite signed Radon measures $\nu_1$, $\nu_2$ on $\left[0,\frac{\pi}{2}\right]$, $\xi_1$, $\xi_2\in L^{\infty}([0,\tfrac{\pi}{2}];\R^{+})$, and $\l\geq 0$,  such that when incorporating  equations \eqref{ex1}  into Theorem  \ref{maxprincipleS}$(i)$-$(vii)$, we obtain 
\begin{enumerate}[(a)]
\item  $\|p(\frac{\pi}{2})\|+\l=1.$
\item The admissibility equation holds, that is, for $t\in[0,\tfrac{1}{2}]$ a.e.,  \begin{equation*}\begin{cases}\dot{\bar{x}}_1(t)= \x_1(t)-\x_2(t)-\u(t)+\y^2(t)-2\x_1(t)(\xi_1(t)+\xi_2(t)),\\
		\dot{\bar{x}}_2(t)= \x_1(t)+\x_2(t)+\u(t)+\y^3(t)-2\x_2(t)(\xi_1(t)+\xi_2(t)),\\
		 \dot{\bar{x}}_3(t)= \x_3(t)+t-\pi-1-\frac{32}{\pi}(\xi_1(t)-\xi_2(t)),\\
	\dot{\bar{y}}(t)= \x_1^2(t)+\x_2^2(t)-16+\u(t)+\y(t). 
\end{cases} \end{equation*} 
\item The adjoint equation is satisfied, that is, for  $t\in [0,\frac{\pi}{2}]$, 
\begin{eqnarray*} dq(t)&=&\begin{pmatrix*}\begin{array}{rrr}
			\bp -1\, & -1\;\;\, & 0 \\
			\bp 1\, & -1\;\;\, & 0 \\
			\bp 0\, & 0\;\;\,  & -1 
	\end{array}\end{pmatrix*}q(t)\; dt\; + \begin{pmatrix} 
	-2\x_1(t) \\
-	2\x_2(t)\\
	0
\end{pmatrix} v(t)dt \;\\[2pt]&+& (\xi_1(t)+\xi_2(t))\begin{pmatrix} 
		2\;\; & 0\;\; & 0  \\
		0 \;\; & 2\;\; & 0 \\
		0\;\; & 0 \;\; & 0
	\end{pmatrix}q(t)\; dt\;+\; \begin{pmatrix} 
		2\x_1(t) \\
		2\x_2(t)\\
		\frac{32}{\pi}
	\end{pmatrix} d\nu_1 + \begin{pmatrix} 
		2\x_1(t) \\
		2\x_2(t)\\
		-\frac{32}{\pi}
	\end{pmatrix} d\nu_2, \\
\dot{v}(t)&=&\begin{pmatrix*}\begin{array}{rrr}
		\bp 0\, & 0\;\;\, & 0 \\
\end{array}\end{pmatrix*}q(t)\; dt\; -v(t)
\end{eqnarray*}
\item The complementary slackness condition is valid, that is, for $t\in[0,\tfrac{\pi}{2}]$ a.e.,  $$\begin{cases}\xi_1(t)(2q_1(t)\x_1(t) + 2q_2(t)\x_2(t)+\frac{32}{\pi}q_3(t))=0, \\ \xi_2(t)(2q_1(t)\x_1(t) + 2q_2(t)\x_2(t)-\frac{32}{\pi}q_3(t))=0.\end{cases}$$
\item The transversality condition holds, that is, \begin{eqnarray*}
&	(q(0),v(0),-q(\frac{\pi}{2}), -v(\frac{\pi}{2}))\tran\in  \l\{(0,0,0,0,0,-8,\alpha,0): \a \in [-1,1]\}\\&+\{(8\a_1+\a_4,\a_3,\a_2,\a_5,\a_4,0,\a_3 \pi,0): \a_1, \a_2, \a_3, \a_4, \a_5 \in\R\} \\ &\hbox{if} \;\; (\x(0)\tran,\y(0)\tran,\x(\tfrac{\pi}{2})\tran,\y(\tfrac{\pi}{2})\tran)\in \{(4,0,\pi,0,0,4,\frac{\pi}{2},a): a \in \mathbb{R}\}.
\end{eqnarray*}
Similarly, we work on deriving transversality conditions for each of the four cases in \eqref{ex1}. 
\item $\max\{u\;(-q_1(t)+ q_2(t)+v(t)) : u\in [0,1]\}$ is attained at $\u(t)$ for $t\in[0,\frac{\pi}{2}]$ a.e.
\end{enumerate}
We temporarily assume that \begin{equation}\label{exam5} -q_1(t)+q_2(t)+v(t)< 0,\;\;\forall t\in[0,\tfrac{\pi}{2}]\;\hbox{a.e.}\end{equation} 
This gives from (f) that $\u(t)=0$ for $t\in[0,\frac{\pi}{2}]$ a.e. Now solving the differential equations of (b) and using \eqref{ex1}, we obtain that \begin{equation*} \label{lastformu} \xi_1(t)=\xi_2(t)=\frac{1}{4},\;\,\x(t)\tran=(4 \cos t,4 \sin t,\pi-t) \footnote{Note that another possible choice for $\x(\cdot)$ is $\x(t)\tran=(-4 \cos t,-4 \sin t,\pi-t)$.}, \;\;\hbox{and} \;\;\y(t)=0 \;\;\forall t\in[0,\tfrac{\pi}{2}].\end{equation*}
 Hence, from (d), we deduce that $q_3(t)=0$ for $t\in[0,\frac{\pi}{2}]$ a.e., and \begin{equation}\label{ex4} 
	\cos t\; q_1(t)+\sin t\; q_2(t)=0,\;\,\forall t\in[0,\tfrac{\pi}{2}]\;\textnormal{a.e.}, \end{equation}
and the adjoint equation (c) simplifies to the following
\begin{equation}\label{ex3} \begin{cases}
		\dot{v}(t)=-v(t), \\
		d{q}_1(t)= (-q_1(t)-q_2(t))dt -8 \cos t \;v(t)dt + q_1(t)\; dt + 8 \cos t\; (d\nu_1+d\nu_2),\\
d{q}_2(t)= (q_1(t)-q_2(t))dt -8 \sin t \;v(t)dt + q_2(t)\; dt + 8 \sin t\; (d\nu_1+d\nu_2),\\
	d{{q}}_3(t)=-q_3(t) \;dt+ \frac{32}{\pi}( d\nu_1- d\nu_2).  \end{cases} \end{equation}  
	Since $v(\frac{\pi}{2})=0$ then $v(t)=0 \;\forall t \in [0, \frac{\pi}{2}]$. Using (a), \eqref{ex4}, (e), and \eqref{ex3}, one can get the following
	$$\begin{cases} \l=\tfrac{2\pi}{2\pi + \sqrt{1+(16 \pi)^2}}\;\, \hbox{and}\;\, A =\tfrac{1}{2\pi + \sqrt{1+(16 \pi)^2}},\\[2pt]
		q(t)\tran=\left(A \sin t,-A \cos t,0\right)\;\hbox{on}\;[0,\tfrac{\pi}{2}),\;\;\quad	q(\tfrac{\pi}{2})\tran=(A,16A\pi,0),\\[2pt]
		 d\nu_1=d\nu_2=A\pi\delta_{\big\{\bbp\tfrac{\pi}{2}\bbp\big\}},
		\end{cases}\vspace{0.01cm}
		$$
	where $\delta_{\{a\}}$ denotes the unit measure concentrated on the point $a$. Note that for all $t\in [0,\tfrac{\pi}{2}]$, we have $-q_1(t)+q_3(t)+v(t)< 0 $, and hence, the temporary assumption \eqref{exam5} is satisfied.
	\\
	Therefore, the above analysis, realized via Theorem \ref{maxprincipleS}, produces an admissible pair $((\x,\y),\u)$, where $$\textstyle \x(t)\tran=(4 \cos t,4 \sin t,\pi-t),\;\;\y(t)=0,\;\;\hbox{and}\;\;\u(t)=0,\;\;\forall t\in [0,\tfrac{\pi}{2}],$$
	which is optimal for $(P)$. 
	\begin{figure}[H]
		\centering
		\includegraphics[scale=0.45]{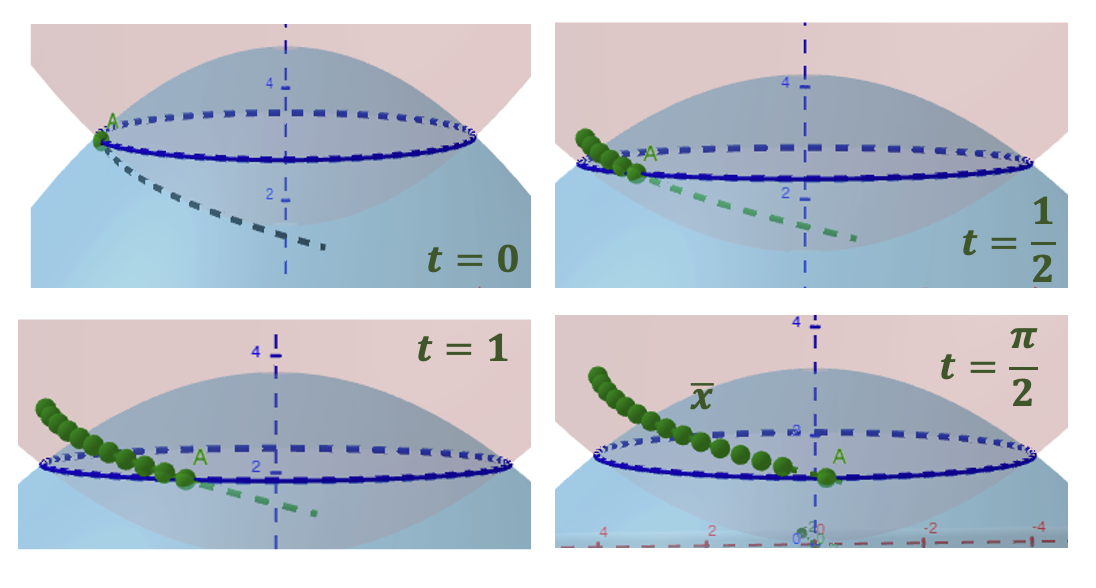}
		\caption{\label{Fig1} The solution $\bar{x}(t)$ (in green) evolving on the set $C(t)=C_1(t)\cap C_2(t)$ over different time instances.}
	\end{figure}
\end{example}
 
\section{Appendix} \label{appendix}
The following result holds for uniform prox-regular sets $S(t)$ not necessarily being  sublevel sets. It is an extension of a {\it special case} of \cite[Lemma 3.2]{nourzeidan} from a {\it constant compact } set $S$  to the case of set-valued maps $S(\cdot)$ with non-compact values. It requires $N_{S(\cdot)}(\cdot)$  to have closed graph. 
\begin{lemma}\label{proxregular}
	Let $S(\cdot): [0,T] \rightsquigarrow \mathbb{R}^n$ be such that, for all $t\in [0,T]$, $S(t)$ is nonempty, closed, and uniformly $\rho^*$-prox-regular, for some $\rho^*>0$. Let
 $\bar{x}(\cdot)\in \mathcal{C}([0,T],\mathbb{R}^n)$ with $\bar{x}(t)\in S(t) $ for all $ t \in [0,T]$.  Assume that for some $\delta \in (0,\rho^*)$,  the map $(t,y) \to N_{S(t)}(y)$  
	has closed graph on the domain   Gr\,$\left(S(\cdot)\cap\bar{B}_{\delta}(\bar{x}(\cdot))\right)$.Then, the following holds. 
	\begin{itemize}
		\item[$(i)$] There exists $ {\rho}_{\delta} >0$ such that for all $t\in[0,T]$ the set $ S(t) \cap \bar{B}_{\delta}(\bar{x}(t))$  is ${\rho}_{\delta}$-prox-regular. 
				\item[$(ii)$] For  $t\in [0,T]$, 
		$\pi(t,\cdot):=\pi_{\left (S(t) \cap \bar{B}_{\delta}(\bar{x}(t)) \right)}(\cdot)$
		is well-defined on $\left (S(t) \cap \bar{B}_{\delta}(\bar{x}(t)) \right) + {\rho}_{\delta}  B$ and $2$-lipschitz on $\left (S(t) \cap \bar{B}_{\delta}(\bar{x}(t)) \right)+ \frac{{\rho}_{\delta}}{2} \bar{B}$. 
	\end{itemize}
\end{lemma}
\begin{proof}
	$(i)$:  The result is derived  by following the proof of \cite[Lemma 3.2]{nourzeidan}, where for  $t \in [0,T]$, we take $S:=S(t)$ and  $x:=\bar{x}(t)$, and  hence, $\mathcal{N}_x$ and  $\rho_x$ there , are respectively  $\mathcal{N}_{t}:=\mathcal{N}_{\x(t)} $ and  $\rho_t:=\rho_{\x(t)}$.  It follows that   $\rho$ there is now  $\hat\rho:=\inf\{\rho_t: t\in [0,T]\}$.  Following the rest of the proof there, and employing that  the map $(t,y) \to N_{S(t)}(y)$  has  closed  graph on the domain   Gr\,$\left(S(\cdot)\cap\bar{B}_{\delta}(\bar{x}(\cdot))\right)$, we conclude that  $\hat\rho >0$. Thus,  for all $t\in [0,T]$, we deduce that $S(t)\cap \bar{B}_{\delta}(\bar{x}(t))$ is $\rho_{\delta}$-prox regular, where $\rho_{\delta}:=\frac{\rho^*\hat\rho}{2}$.\\
	\noindent$(ii)$: It follows  from $(i)$ and \cite[Theorem 4.8]{prox}. 
\end{proof}

Translating Lemma 6.2 in \cite{nourzeidan3} to our setting gives us the following lemma. 
\begin{lemma}\label{lemmaaux1} Assume that  $\psi_i$ is continuous for all $i=1,\dots,r$. Let $\a_n\ge 0$,  for all $n\in \N$, with   $\a_n\f\a_o$ and  let  $(t_n,c_n) \in $ Gr $C(\cdot)$ be a sequence such that  $\I^{\a_n}_{(t_n,c_n)}\neq\emptyset$, for all $n\in \N$, and $(t_n,c_n)\f (t_o,c_o)$. Then, $\I^{\a_o}_{(t_o,c_o)}\neq\emptyset$ and there exist $ \emptyset\not=\J_o\subset \{1,\dots,r\}$ and a subsequence of  $(\a_n,t_n, c_n)_n$  we do not relabel, such that  
	$$\I^{\a_n}_{(t_n,c_n)}= \J_o\subset \I^{\a_o}_{(t_o,c_o)}\;\,\hbox{for all}\;\,n\in\N.$$
	In particular, for all  $a\ge 0$, for any continuous function $x\colon[0,T]\f 	\mathbb{R}^n$  such that $x(t) \in C(t)$ for all $t \in [0,T]$, we have $I^a(x)$ is closed, and hence compact. 
\end{lemma}
	
This result shall be used in different places of this paper. 
 \begin{lemma}\label{convergence} 
 \textnormal{\bf(i).} 
 Let $(x_n,y_n) \in W^{1,\infty}([0,T];\mathbb{R}^{n+d}) $,  $(\xi_n^1,\cdots, \xi_n^{\mathcal{R}}, \zeta_n) \in L^{\infty}([0,T], \mathbb{R}_{+}^{\mathcal{R}+1})$,  \  be such that, for some positive constants $M_1, M_2, M_3$  we have, $\forall  n\in\mathbb{N}$ and  $\forall i \in \{ 1, \cdots, \mathcal{R}\},$   
 \begin{equation}\label{boundxyxi-n} 		 
 \|(x_{n}, y_n)\|_{\infty}\le M_1,\;\; \|(\dot{x}_{n}, \dot{y}_n)\|_{\infty}  \le M_2,\; \;
   \|(\xi^i_n, \zeta_n) \|_{\infty} \le M_3.  
 	\end{equation} 
	Then, there exist  $(x,y)\in W^{1,\infty}([0,T];\mathbb{R}^{n+d})$ and   $( \xi^1, \cdots, \xi^{\mathcal{R}}, \zeta)\in L^{\infty}([0,T]; \mathbb{R}^{\mathcal{R}+1}_{+})$ such that  $(x_n, y_n)$, and $ ( \xi_n^1, \cdots, \xi_n^{\mathcal{R}},\zeta_n)$  admit a subsequence  $($not relabeled$)$ satisfying  $\forall i \in \{ 1, \cdots, \mathcal{R}\},$ \begin{equation}\label{boundxyxi} \begin{cases}
 		(x_n, y_n) \xrightarrow[]{unif} (x, y),  \quad  (\dot{x}_n, \dot{y}_n) \xrightarrow[in \; L^{\infty}]{w*} (\dot{x}, \dot{y}), \;\; 
	(\xi^i_n,  \zeta_n) \xrightarrow[in \; L^{\infty}]{w*} ( \xi^i, \zeta),  \\
 		 \|(x, y)\|_{\infty}\le M_1,\;\; \|(\dot{x}, \dot{y})\|_{\infty}  \le M_2, \;\;
 		  \|(\xi^i, \zeta)\|_{\infty} \le M_3. \end{cases} 
 \end{equation}	
 \textnormal{\bf(ii).}  	 
For  given  $q_i: [0,T] \times \mathbb{R}^n \longrightarrow \mathbb{R}$, let       $Q(t): =\cap_{i=1}^{\mathcal{R}} \{ x \in \mathbb{R}^n: q_i(t,x) \le 0\}$  
 and  $(\bar{x}, \bar{y})\in \mathcal{C}([0,T];\R^n \times \R^l)$ be such that  $\bar{x}(t) \in Q(t) \; \forall t\in [0,T]$.  Assume $(A2)$ is satisfied by $C(t):=Q(t)$, and, for some $\bar{\delta}\, >0$,   $(A3.1)$ and {$(A4.1)$} hold at $((\bar{x},\bar{y}); \bar{\delta})$  respectively by  $\psi_i:=q_i$ and  $h=(f,g): [0,T] \times \mathbb{R}^n \times \mathbb{R}^l \times \mathbb{R}^m  \longrightarrow \mathbb{R}^n\times \R^l.$  Let $(x_n,y_n)$ and $(\xi_n,\cdots,\xi_n^{\r}, \zeta_n) $ be such that  $(x_n(t),y_n(t))\in [Q(t)\cap \bar{B}_{\bar{\delta}}(\bar{x}(t))]\times \bar{B}_{\bar{\delta}}(\bar{y}(t)) \, (\forall t\in[0,T])$ and    \eqref{boundxyxi-n}  is satisfied, and let $(x,y,\xi^i,\zeta)$ be their corresponding limits via \eqref{boundxyxi}.  Consider $u_n\in \mathcal{U}$ such that,  
 for all $n\in \N$, $((x_n,y_n),u_n)$, and $(\xi_n^1,\cdots,\xi_n^{\r}, \zeta_n)$ satisfy
 \begin{eqnarray} 
		\begin{cases}
			\dot{x}(t) = f(t,x(t), y(t),u(t)) - \sum_{i=1}^{\mathcal{R}}\xi^i(t) \nabla_x q_i(t,x(t)) \text{ a.e. } t \in [0,T],\\
			 \dot{y}(t) = g(t, x(t), y(t), u(t)) -\zeta(t) \nabla_y \varphi(t,y(t)) \text{ a.e. } t \in [0,T],
		\end{cases}  \label{equiv} \end{eqnarray}	
where $\varphi$ is given by \eqref{phi}.Then, in either of the following cases,  there exists $u\in \mathcal{U}$ such that $((x,y),u)$, and $( \xi^1, \cdots, \xi^{\mathcal{R}}, \zeta)$ also satisfy   system \eqref{equiv}.  \\
 \textbf{Case 1.} If there exists a subsequence of $u_n$ that converges pointwise a.e. to some $u \in \mathcal{U}$.  \\
 \textbf{Case 2.} If {$(A1)$ and}  $(A4.2)$  are satisfied.
 \end{lemma} 
 	\begin{proof}
	{\textnormal{\bf(i)}}: By \eqref{boundxyxi-n}, the sequence $(x_n,y_n)_n$ is equicontinuous and uniformly bounded. Hence, using Arzela-Ascoli theorem  and that $(\dot{x}_n,\dot{y}_n)$ is uniformly bounded  in $L^{\infty}$,  it follows that there exists $(x,y)\in W^{1,\infty}([0,T];\mathbb{R}^{n+d})$ such that along a subsequence (we do not relabel) of $(x_n, y_n)$, we have $(x_n,y_n) \xrightarrow[]{unif} (x,y)$, $(\dot{x}_n, \dot{y}_n) \xrightarrow[in \; L^{\infty}]{w*} (\dot{x}, \dot{y})$, 	with $(x,y)$ and $(\dot{x},\dot{y})$ satisfy the bounds in \eqref{boundxyxi}.
As  $\|(\xi^i_n, \zeta_n) \|_{\infty} \le M_3	$ for all $i=1,\cdots,\mathcal{R}$ and for all $n\in\N$,   some subsequences of $ ( \xi_n^1, \cdots, \xi_n^{\mathcal{R}}, \zeta_n)$  converge in the weak$^*$-topology to  some $ ( \xi^1, \cdots, \xi^{\mathcal{R}},\zeta) \in L^{\infty}$   which satisfy the required bound in \eqref{boundxyxi}.\\
\\
\textbf{ \textnormal{\bf(ii)} Case 1:}  	Let $ t \in [0,T)$ lebesgue point of $\dot{x}(\cdot), \dot{y}(\cdot), f(\cdot, x(\cdot), y(\cdot), u(\cdot)), G(\cdot, x(\cdot), y(\cdot), u(\cdot))$,  $\xi^i(\cdot)$ for all $i=1,\cdots, \r$ and $\zeta$,  and let $\tau \in (0,T-t)$. Then, \eqref{equiv} implies 
	\begin{eqnarray} \label{xkeq}\begin{cases}
\frac{x_{n}(t+\tau) - x_{n}(t)}{\tau} = \frac{1}{\tau} \int_{t}^{t +\tau} \left[f(s,x_{n}(s),y_{n}(s), u_{n}(s)) - \sum_{i=1}^{\mathcal{R}}\xi_{n}^i(s) \nabla_x q_i(s,x_{n}(s))\right] ds,   \\ 
 		\frac{y_{n}(t+\tau) - y_{n}(t)}{\tau} = \frac{1}{\tau} \int_{t}^{t +\tau} \left[ g(s,x_{n}(s), y_{n}(s), u_{n}(s)) - \zeta_n(s) \nabla_y \varphi(s, y_n(s)) \right]  ds.    \end{cases}
 	\end{eqnarray}
 	Using Dominated Convergence Theorem, and taking the limit as $n\rightarrow \infty$ of (\ref{xkeq}), we deduce that 
 	\begin{eqnarray}\label{xeq} \begin{cases}
 		\frac{x(t+\tau) - x(t)}{\tau} = \frac{1}{\tau} \int_{t}^{t +\tau} [f(s,x(s),y(s), u(s)) - \sum_{i=1}^{\mathcal{R}}\xi^i(s) \nabla_x q_i(s,x(s))] ds, \\
 		\frac{y(t+\tau) - y(t)}{\tau} = \frac{1}{\tau} \int_{t}^{t +\tau} \left[ g(s,x(s), y(s), u(s)) - \zeta(s) \nabla_y \varphi(s, y(s)) \right] ds. \label{yeq} \end{cases}
 	\end{eqnarray}
 	Now, let $\tau \to 0 $ in $(\ref{xeq})$, we get that \eqref{equiv} is satisfied for every $t$ lebesgue point, hence it hold for a.e. $t \in [0,T]$. \\
	\textbf{ \textnormal{\bf(ii)} Case 2:} For $s\in [0,T]$ a.e., define  in $\mathbb{R}^n \times \mathbb{R}^l$ the sets $S_n(s):= h(s,x_{n}(s), y_{n}(s), U(s))$ and $S(s):=  h(s, x(s), y(s), U(s))$.  Using  (A1), the continuity of $h(s,\cdot,\cdot,\cdot)$ in (A4.1),  and (A4.2), it follows that $S_n(s)$ and $S(s)$ are nonempty closed convex sets and $S_n(s)$ {\it Hausdorff}-converges to $S(s)$. 	
Hence, Filippov Selection Theorem   yields  that    $((x_n,y_n),u_n)$ and $(\xi^i_n,\zeta_n)$  satisfying  \eqref{equiv} is equivalent to, 	$\forall z \in \mathbb{R}^{n+d} $ and  $s\in[0,T]$  a.e.,
	
		 \begin{eqnarray}		\langle z,(\dot{x}_{n}(s), \dot{y}_{n}(s))  \rangle \le &&\sigma(z, h(s,x_{n}(s),y_{n}(s), U(s)))\nonumber\\ &&- \langle z, (\sum_{i=1}^{\mathcal{R}} \xi_n^i(s) \nabla_x q_i (t,x_n(s)), \zeta_n(s) \nabla_y \varphi(s,y_n(s)) ) \rangle.   \label{ineqk}	\end{eqnarray}
Furthermore,  by $\eqref{limitsetconvergence}$ and the positive homogeneity of $\sigma(\cdot, S_n)$ and $\sigma(\cdot,S),$  we deduce that
 		\begin{equation*}
 			\sigma(z, h(s,x_{n}(s),y_{n}(s), U(s))) \xrightarrow[k \to \infty]{} \sigma(z,h(s,x(s),y(s), U(s))),\, \forall z \in \mathbb{R}^{n+d} \, \mbox{and }\;  s\in [0,T]\; \mbox{a.e.}, 		\end{equation*}
 and  the bound of $h$ in (A4.1) gives  that, for $z\in \R^{n+d}$,
 		\begin{eqnarray*}
 			|\sigma\left(z, h(s,x_{n}(s), y_{n}(s), U(s))\right)| 
 			&\le& 2\|z\| M_h.
 		\end{eqnarray*} 
 Thus, for $ t \in [0,T)$  a lebesgue point of $\xi^{i}(\cdot), \zeta(\cdot), \dot{x}(\cdot), \dot{y}(\cdot), \sigma(z,h(\cdot,x(\cdot),y(\cdot), U(\cdot)))$, and for $\tau \in (0,T-t)$, when integrating $(\ref{ineqk})$ on $[t,t+\tau]$ and then taking the limit as $n \to \infty$,  the Dominated Convergence Theorem yields that, $\forall z=(z_1,z_2) \in \mathbb{R}^n \times \mathbb{R}^l$,  
 \begin{eqnarray*}
 	& &\hspace{-.6 in}\int_{t}^{t+\tau} \langle z, (\dot{x}(s), \dot{y}(s)) \rangle  ds 
	\nonumber\\ 	&\le& \int_{t}^{t+\tau} [ \sigma (z, h(s, x(s), y(s),  U(s))) -\langle z, (\sum_{i=1}^{\mathcal{R}} \xi^i(s) \nabla_x q_i (s,x(s)), \zeta(s) \nabla_y \varphi(s,y(s)) ) \rangle  ] ds. \label{ineq}
 \end{eqnarray*}	 
 Dividing the last equation by $\tau$ and taking the limit when $\tau \to 0$, we get that, $\forall z \in \mathbb{R}^n \times \mathbb{R}^l,$ and for $t$ lebesgue point,	
 \vspace{-0.3cm} 	
 \begin{eqnarray}
 		\langle z,(\dot{x}(t), \dot{y}(t)) \rangle \le \sigma(z, h(t,x(t),y(t), U(t))) -  \langle z, (\sum_{i=1}^{\mathcal{R}} \xi^i(t) \nabla_x q_i (t,x(t)), \zeta(t) \nabla_y \varphi(t,y(t)) ) \rangle, \nonumber
		\end{eqnarray}		
   and hence this inequality is valid for  $ t \in [0,T]$ a.e. Therefore, by means of  Filipov Selection Theorem,  there exists $u\in \mathcal{U}$, such that $((x,y),u)$, and $( \xi^1, \cdots, \xi^{\mathcal{R}}, \zeta)$ satisfy   system \eqref{equiv}. 	
 		\end{proof}
\begin{remark}\label{global-lemma3}  When $\bar\delta =\infty$,  Lemma \ref{convergence} remains valid with  $(\bar{x},\bar{y})$  and the assumptions  involving them  are now superfluous. In this case, recall that $(A3.1)$,$(A4.1)$, and $(A4.2)$  are replaced by  $(A3.1)_G$,  $(A4.1)_G$, and $(A4.2)_G$, respectively.
 \end{remark}
 
 \subsection*{Acknowledgment}
 The authors of this paper would like to thank the anonymous reviewer(s)  for their suggestions and comments that helped improving the presentation of the paper, and the associate editor for the positive feedback.

 \section*{Statements and Declarations}
 \subsection*{Funding}
 
 \noindent The research of Vera Zeidan is partially supported by SIMONS FOUNDATION under grant MP-TSM-00002506.
 \subsection*{Competing Interests}
 \noindent The authors have no relevant financial or non-financial interests to disclose.
 \subsection*{Authors Contributions}
 \noindent Both authors contributed to the study conception and proofs of the results. Material preparation,  and analysis were performed by Samara Chamoun and Vera Zeidan. The first draft of the manuscript was written by Samara Chamoun  and   Vera Zeidan edited,  revised, and commented on previous versions of the manuscript. Both authors read and approved the final manuscript.

\end{document}